\documentclass[11pt]{amsart}
\usepackage{latexsym}
\usepackage{amsxtra}
\usepackage{verbatim}
\usepackage{color}
\usepackage{amsthm,amsmath,amsfonts,amssymb,mathrsfs,amscd,graphics}
\usepackage{amssymb}
\usepackage{appendix}
\usepackage{ifpdf}
\usepackage{enumerate}
\usepackage{fancyhdr}
\usepackage{mathrsfs}
\usepackage[mathscr]{eucal}
\usepackage{upgreek}
 \usepackage{wasysym}
\usepackage{esint}
\usepackage{pstricks}
\usepackage{color}
\usepackage{cancel}
\usepackage{verbatim}
\usepackage[all]{xy}

\usepackage{stackrel}
\usepackage[utf8]{inputenc}
\usepackage{tabularx}
\usepackage{graphicx}
\usepackage[colorlinks=true]{hyperref}
\usepackage{color}
\usepackage{lipsum}
\usepackage{float}
\usepackage{mathtools}
\usepackage{longtable}
\allowdisplaybreaks
\usepackage{tikz}
\usepackage{amstext, amsbsy}
\usepackage{amsfonts, extarrows}

\usepackage[margin=1.5in]{geometry}

\newtheorem{theorem}{Theorem}[section]
\newtheorem{lemma}[theorem]{Lemma}
\theoremstyle{definition}
\newtheorem{definition}[theorem]{Definition}

\newtheorem{example}[theorem]{Example}

\newtheorem{proposition}[theorem]{Proposition}
\newtheorem{corollary}[theorem]{Corollary}

\newtheorem{conj}[theorem]{Conjecture}

\theoremstyle{remark}
\newtheorem{remark}[theorem]{Remark}
\numberwithin{equation}{section}
\newcommand{\bP}{{\mathbb P}}
\newcommand{\bD}{{\mathbf D}}
\newcommand{\cX}{{\mathcal X}}
\newcommand{\sJ}{{\bf e}}
  \newcommand{\fE}{{\mathfrak E}}
  \newcommand{\fN}{{\mathfrak N}}

\newcommand{\gmax}{G_{W}}

    \newcommand{\sE}{{\mathscr E}}
    \newcommand{\fG}{{\mathfrak G}}
    \newcommand{\fL}{{\mathfrak L}}

      \newcommand{\cN}{{\mathcal N}}

    \newcommand{\cH}{{\mathcal H }}

    \newcommand{\cA}{{\mathcal A}}
        \newcommand{\cD}{{\mathcal D}}
            
  \newcommand{\C}{{\mathbb C}}
    \newcommand{\CC}{{\mathbb C}}
  \newcommand{\R}{{\mathbb R}}
  \newcommand{\Z}{{\mathbb Z}}

\newcommand{\bt}{\mathbf{t}}

\newcommand{\bb}{\mathbb}
\newcommand{\ca}{\mathcal}

\newcommand{\gr}{\mathrm{Gr}}

\newcommand{\ovE}{\overline{E}}
\newcommand{\ovF}{\overline{F}}

\newcommand{\ga}{\gamma}

\newcommand{\one}{\mathbf{1}}

\newcommand{\LL}{\llangle[\big]}
\newcommand{\RR}{\rrangle[\big]}

\newcommand{\LD}{\Big\langle}
\newcommand{\RD}{\Big\rangle}

\makeatletter
\newsavebox{\@brx}
\newcommand{\llangle}[1][]{\savebox{\@brx}{\(\m@th{#1\langle}\)}%
  \mathopen{\copy\@brx\kern-0.5\wd\@brx\usebox{\@brx}}}
\newcommand{\rrangle}[1][]{\savebox{\@brx}{\(\m@th{#1\rangle}\)}%
  \mathclose{\copy\@brx\kern-0.5\wd\@brx\usebox{\@brx}}}
\makeatother

  \newcommand{\<}{\langle}
  \renewcommand{\>}{\rangle}

\makeatletter
\MHInternalSyntaxOn
\def\MT_leftarrow_fill:{%
  \arrowfill@\leftarrow\relbar\relbar}
\def\MT_rightarrow_fill:{%
  \arrowfill@\relbar\relbar\rightarrow}
\newcommand{\xrightleftarrows}[2][]{\mathrel{%
  \raise.55ex\hbox{%
    $\ext@arrow 0359\MT_rightarrow_fill:{\phantom{#1}}{#2}$}%
  \setbox0=\hbox{%
    $\ext@arrow 3095\MT_leftarrow_fill:{#1}{\phantom{#2}}$}%
  \kern-\wd0 \lower.55ex\box0}}
\MHInternalSyntaxOff
\makeatother

\usepackage{xcolor}
\newenvironment{Comment}[2]{\noindent \color{#1}{\texttt #2: }}{\par \noindent}
\newcommand{\Ming}[1]{\begin{Comment}{red}{Ming}#1 \end{Comment}}

\begin{document}
\title[Quantum spectrum and Gamma structures in FJRW theory]
{Quantum spectrum and Gamma structures for 
quasi-homogeneous polynomials 
of general type}


\author{Yefeng Shen}
\author{Ming Zhang}

\maketitle

\begin{abstract} 
Let $W$ be a quasi-homogeneous polynomial of general type and $\<J\>$ be the cyclic symmetry group of $W$ generated by the exponential grading element $J$. We study the quantum spectrum and asymptotic behavior in Fan-Jarvis-Ruan-Witten theory of the Landau-Ginzburg pair $(W, \<J\>)$.

Inspired by Galkin-Golyshev-Iritani's Gamma conjectures for quantum cohomology of Fano manifolds, we propose Gamma conjectures for Fan-Jarvis-Ruan-Witten theory of general type. We prove the quantum spectrum conjecture and the Gamma conjectures for Fermat homogeneous polynomials and the mirror simple singularities.

The Gamma structures in Fan-Jarvis-Ruan-Witten theory also provide a bridge from the category of matrix factorizations of the Landau-Ginzburg pair (the algebraic aspect) to its analytic aspect.
We will explain the relationship among the Gamma structures, Orlov's semiorthogonal decompositions, and the Stokes phenomenon.
\end{abstract}

{\hypersetup{linkcolor=black}\setcounter{tocdepth}{2} \tableofcontents}

\section{Introduction}

\subsection{FJRW theory for admissible Landau-Ginzburg pairs}
\subsubsection{A Landau-Ginzburg pair}
Let $W: \C^N\to \C$ be a quasi-homogeneous polynomial of degree $d\in\mathbb{Z}_{\geq2}$. 
The weight of the variables, denoted by ${\rm wt}(x_i)=w_i$, are positive integers such that for all $\lambda\in \C$, 
$$
\lambda^d W(x_1, \ldots, x_N)=W(\lambda^{w_1}x_1, \ldots, \lambda^{w_N}x_N).$$ 
Throughout this paper, we assume that the polynomial $W$ is {\em nondegenerate} and
\begin{equation}
\label{gcd-condition}
\gcd(w_1, \ldots, w_N)=1.
\end{equation}
 The nondegenerate condition means that $W$ has an isolated critical point only at the origin and the choices of ${w_i\over d}$ are unique. 
We sometimes abuse the language and call a nondegenerate quasi-homogeneous polynomial $W$ the {\em singularity} $W$.

The singularity $W$ is invariant under the action of the {\em maximal diagonal symmetry group of $W$} defined by
\begin{equation}
\label{maximal-group}
G_W:=\{g\in (\C^*)^N\mid W(g\cdot (x_1, \ldots, x_N))=W(x_1, \ldots, x_N)\}.
\end{equation}
Here the action $g\cdot (x_1, \ldots, x_N)$ is given by the coordinate-wise multiplication.
In particular, the group $G_W$ contains the {\em exponential grading element}
$$J=\left(e^{2\pi\sqrt{-1}w_1/d}, \ldots, e^{2\pi\sqrt{-1}w_N/d}\right).$$
By condition \eqref{gcd-condition}, the subgroup $\<J\>\subset G_W$ is a cyclic group of order $d$.

In general, we may consider a subgroup $G$ of $G_W$ that contains the element $J$. Such a pair $(W, G)$ is called an {\em admissible} Landau-Ginzburg (LG) pair.
\begin{definition}
\label{index-lg-pair}
For an admissible LG pair $(W, \<J\>)$, we define its  {\em weight system} by 
the $(N+1)$-tuple
\begin{equation}
\label{weight-system}
(d; w_1, \ldots, w_N)\in (\Z_{>0})^{N+1}.
\end{equation}
We define the {\em index} of the LG pair $(W, \<J\>)$ by the integer
\begin{equation}
\label{gorenstein-parameter}
\nu:=d-\sum_{i=1}^Nw_i\in \mathbb{Z}
\end{equation}
We say the pair $(W, \<J\>)$ is {\em Fano/Calabi-Yau/general type} if $\nu$ is negative/zero/positive.
\end{definition}

\begin{remark}
Usually, the integer $-\nu$ is called the Gorenstein parameter of the LG model \cite{Orl, BFK2}. 
We choose to use $\nu$ in this paper simply because we focus on the LG pair of general type and a positive parameter is convenient to deal with in many situations.
\end{remark}

\subsubsection{Fan-Jarvis-Ruan-Witten theory} 
Based on a proposal of Witten~\cite{Wit}, for any admissible LG pair $(W, G)$,
Fan, Jarvis, and Ruan constructed an intersection theory on the moduli space of $W$-spin structures of $G$-orbifold curves in a series of works~\cite{FJR-book, FJR}. 
This theory is called Fan-Jarvis-Ruan-Witten (FJRW) theory nowadays.
It is analogous to Gromov-Witten (GW) theory and generalizes the theory of $r$-spin curves, where $W=x^r$ and $G=\<J\>$.

Briefly speaking, the main ingredient of FJRW theory consists of a state space and a {\em cohomological field theory} (CohFT).
The state space, denoted by $\cH_{W,G}$, is the $G$-invariant subspace of the middle-dimensional relative cohomology for $W$ with a nondegenerate pairing $\< \cdot, \cdot\>$; see~\cite[Section 3.2]{FJR}.
The CohFT is a set of multilinear maps $\{\Lambda_{g,k}^{W,G}\}$~\eqref{cohft-notation} mapping from $\cH_{W,G}^{\otimes k}$ to the cohomology of the moduli spaces of stable curves of genus $g$ with $k$ markings. 
These maps satisfy a collection of axioms listed in~\cite[Theorem 4.2.2]{FJR}.
The CohFT is used to define the genus-$g$, $k$-pointed {\em FJRW invariants} 
$$\LD\gamma_1\psi_1^{\ell_1},\ldots,\gamma_k\psi_k^{\ell_k}\RD_{g,k}^{W,G};$$
see formula~\eqref{fjrw-inv}. Here $\{\gamma_i\in \cH_{W, G}\}$ are called {\em insertions}, $\psi_i$'s are the psi classes on the moduli space of stable curves, and $\ell_i\in\Z_{\geq 0}$. 
We will review some properties of the FJRW theory in Section~\ref{sec:fjrw}.

\begin{remark}
There are two other versions of CohFTs for the LG pair $(W, G)$. 
One version was constructed by Polishchuk and Vaintrob using matrix factorizations~\cite{PV-cohft}. 
Another version using cosection localized Gysin maps was initiated by Chang, Li, and Li in~\cite{CLL} and completed by Kiem and Li~\cite{KL}.
There is a folklore conjecture that all these three CohFTs are equivalent. 
Such an equivalence is only partially verified in~\cite{CLL}.
\end{remark}

\subsubsection{A quantum product}
We only consider genus-zero invariants in this paper.
There is an associative product $\star$ on $\cH_{W, G}$, determined by the nondegenerate pairing and the genus-zero three-pointed FJRW invariants via
$$\<\gamma_1\star\gamma_2, \gamma_3\>=\LD\gamma_1, \gamma_2, \gamma_3\RD_{0,3}^{W,G}.$$
The genus-zero $k$-pointed FJRW invariants with $k\geq 3$ produce a {\em quantum product} $\star_\gamma$ (see~\eqref{quantum-product}), which can be viewed as a deformation of the product $\star$ along the element $\gamma\in\cH_{W, G}$.

\subsection{Quantum spectrum for FJRW theory of general type}

In this paper, we will focus on the enumerative geometry of the LG pair $(W, \<J\>)$.

\subsubsection{The element $\tau(t)$}
Let $(d; w_1, \ldots, w_N)$ be the weight system of an admissible LG pair $(W, \<J\>)$. 
We define its set of {\em narrow indices} by 
\begin{equation}
\label{narrow-index}
{\bf Nar}=\left\{m\in \mathbb{Z}_{>0} \mid 0< m<d, \text{ and } d\nmid w_j\cdot m, \  \forall \ 1\leq j\leq N\right\}.
\end{equation}
As a vector space, $\cH_{W, \<J\>}$ admits a {\em narrow-broad decomposition} $\cH_{W, \<J\>}=\cH_{\rm nar}\oplus \cH_{\rm bro}$ 
where
\begin{equation*}
\cH_{\rm nar}:=\bigoplus_{m\in {\bf Nar}}\cH_{J^m}
\end{equation*}
is the narrow subspace and  $\cH_{\rm bro}$ is the broad subspace.
For each $m\in {\bf Nar}$, the subspace $\cH_{J^m}$ is one-dimensional and spanned by a {\em standard generator} $\sJ_m$ defined in~\eqref{standard-generator}.

Let $\{r\}$ be the fractional part of the number $r\in \mathbb{R}$ and $\Gamma(x)$ be the {\em gamma function}.
Following~\cite{Aco, CIR, Gue, RR}, we introduce a {\em small $I$-function} of the FJRW theory  for the LG pair $(W, \<J\>)$ by
\begin{equation}
\label{small-i-function}
I_{\rm FJRW}^{\rm sm}(t,z)
=z\sum_{m\in {\bf Nar}}
\sum_{\ell=0}^{\infty}
\frac
{t^{d\ell+m}}
{z^{d\ell+m-1}\Gamma(d\ell+m)}
\prod_{j=1}^{N}
{z^{\lfloor {w_j\over d}(d\ell+m)\rfloor}
\Gamma({w_j\over d}(d\ell+m))\over 
\Gamma\left(\left\{{w_j\over d}\cdot m\right\}\right)}\sJ_m.
\end{equation}
By definition, we have $I_{\rm FJRW}^{\rm sm}(t,z)\in \cH_{\rm nar}[\![t]\!][z, z^{-1}]\!].$
Let $t\tau(t)$ be the coefficient of $z^0$ of $I_{\rm FJRW}^{\rm sm}(t,z)$, namely,
\begin{equation}
\label{def-tau}
\tau(t)={1\over t}\left[I_{\rm FJRW}^{\rm sm}(t,z)\right]_{z^0}\in \cH_{\rm nar}[\![t]\!].
\end{equation}

\subsubsection{Quantum spectrum conjecture for $(W, \<J\>)$ of general type}

Now we focus on LG pair $(W, \<J\>)$ of general type.
We have $\tau(t)\in  \cH_{\rm nar}[t].$
Let $\tau'(t)={d\over dt}\tau(t).$ 
We set $\tau:=\tau(1)$ and $\tau':=\tau'(1)$.
We will consider the quantum multiplication 
\begin{equation}
\label{quantum-J2}
{\nu\over d}\tau'\star_{\tau}={\nu\over d}t\tau'(t)\star_{\tau(t)}\Big\vert_{t=1}.
\end{equation}
We denote the set of eigenvalues of the quantum multiplication~\eqref{quantum-J2} 
by
\begin{equation}
\label{def-evalue}
\fE:=\left\{\lambda\in \C\mid \lambda \text{ is an eigenvalue of } {\nu\over d}\tau'\star_{\tau}\in {\rm End}(\cH_{W, \<J\>})\right\}.
\end{equation}

Inspired by the {\em Conjecture $\mathcal{O}$} for {\em Fano} manifolds proposed by Galkin-Golyshev-Iritani in~\cite{GGI}, we propose a quantum spectrum conjecture for FJRW theory \emph{of general type}.
\begin{conj}
[{\em Quantum spectrum conjecture} for an admissible LG pair $(W, \<J\>)$ of general type]
\label{conjecture-C}
Let $W$ be a nondegenerate quasihomogeneous polynomial of general type, with weight system $(d; w_1, \ldots, w_N)$.
Then the set $\fE$ in~\eqref{def-evalue} contains a number 
\begin{equation}
\label{value-principal-spectrum}
T=\nu\left(d^{-d}\prod_{j=1}^Nw_j^{w_j}\right)^{1\over \nu}\in \mathbb{R}_{>0}
\end{equation}
 that satisfies the following properties:
\begin{enumerate}
\item For any $\lambda\in \fE$, the inequality $|\lambda|\leq T$
holds.
\item The following two sets are the same:
$$\{\lambda\in \fE\mid T= |\lambda|\}=\{e^{2\pi\sqrt{-1}j/\nu} \cdot T\mid j=0, \ldots, \nu-1\}.$$ 
\item The multiplicity of each $\lambda\in \fE$ with $|\lambda|=T$ is one.
\end{enumerate} 
\end{conj}

\begin{remark}
\begin{enumerate}
\item
In Section~\ref{sec-variant}, a similar conjecture (Conjecture~\ref{quantum-spectrum-conj-invariant}) is proposed to study the quantum spectrum on a $G_W$-invariant subspace. 

\item 
For any admissible LG pairs $(W, G)$, we can study the quantum spectrum of \eqref{quantum-J2}. 
The multiplicity-one property in Conjecture ~\ref{conjecture-C} could fail if $G\neq \<J\>$.

\item 
The study of quantum spectrum plays an important role in Katzarkov-Kontsevich-Pantev-Yu's proposal to extract birational invariants from quantum cohomology~\cite{Kon19, Kon20, Kon21} and Iritani's decomposition theorem of the quantum cohomology $D$-module of blowups~\cite{Iri23}.
The quantum spectrum conjectures in this paper can be seen as an analogy in FJRW theory. 
This can be generalized to gauge linear sigma models.
\end{enumerate}
\end{remark}

\subsubsection{Main results on quantum spectrum conjecture~\ref{conjecture-C}}
We will verify Conjecture~\ref{conjecture-C} for the following cases. 
In Section~\ref{sec-mirror-simple}, direct calculations of the quantum product show
\begin{theorem}
Quantum spectrum conjecture~\ref{conjecture-C} holds for the following singularities: 
\begin{equation}
\label{mirror-ade-singularity}
\begin{dcases}
A_n: & W=x^{n+1}, n\geq 1;\\
D_n^{T}: & W=x^{n-1}y+y^2, n\geq 4;\\
E_6: & W= x^3+y^4;\\
E_7: & W=x^3y+y^3;\\
E_8: & W=x^3+y^5.
\end{dcases}
\end{equation}
\end{theorem}
These polynomials are mirror to simple singularities (or $ADE$ singularities), whose central charge satisfies $\widehat{c}_W<1$. 
We call them {\em mirror simple singularities} or {\em mirror ADE-singularities}. 
Note that an $A$-type singularity or an $E$-type singularity is mirror to itself.

In Section~\ref{sec-spectrum-fermat}, we will use mirror symmetry to prove
\begin{theorem}
\label{main-thm-intro}
If $d-N>1$, the quantum spectrum conjecture~\ref{conjecture-C} holds for the Fermat polynomial $W=\sum\limits_{i=1}^{N}x_i^d$.
\end{theorem}

\subsection{Asymptotic expansion in FJRW theory}
Next we study certain asymptotic expansion in the FJRW theory of $(W, \<J\>)$ of general type, inspired by the work~\cite{GGI}. 
The key observation is that the coefficient functions of the small $I$-function $I_{\rm FJRW}^{\rm sm}(t,z)$ in~\eqref{small-i-function} can be transformed into certain {\em generalized hypergeometric functions}; see Proposition~\ref{i-function-via-ghf}.  
Then many results on asymptotic expansions of generalized hypergeometric functions studied in the literature since Barnes, Meijer, and many others~\cite{Bar, Mei, Luk, Fie, dlmf} can be applied here.

\subsubsection{A Dubrovin connection in FJRW theory}
Let $\{\phi_i\}$ be a homogeneous basis of $\cH_{W, \<J\>}$ and $\{\phi^i\}$ be a dual basis with respect to the nondegenerate pairing $\<\cdot, \cdot\>$. 
Let  $t_i$ be the coordinate of  $\phi_i$ and let $\bt:=\sum_i t_i\phi_i$. 
The FJRW invariants produce a Dubrovin-Frobenius manifold structure on $\cH_{W, \<J\>}$, with a {\em flat identity} $\one:=\sJ_1\in \cH_{J}$, a {\em Hodge grading operator}  $\gr\in {\rm End}(\cH_{W, \<J\>})$ defined in~\eqref{eq:modified-hodge-operator}, and an {\em Euler vector field} 
\begin{equation}
\label{euler-vector}
\sE(\bt)=\sum_{i}\Big(1-\deg_\C(\phi_i)\Big)\, t_i{\partial \over \partial t_i}.
\end{equation} 
Here the complex degree $\deg_\C\phi_i\in\mathbb{Q}$ is defined in~\eqref{complex-degree}.
We can identify the vector space $\cH_{W,  \<J\>}$ with its tangent space by sending $\phi_i$ to 
${\partial\over \partial t_i}$.
The quantum product induces a Dubrovin-type connection (or a quantum connection) $\nabla$ on $\cH_{W, \<J\>}\times\mathbb{C}^*$, given by
\begin{equation}
\label{quantum-connection}
\begin{dcases}
\nabla_{\partial\over \partial t_i}&=\frac{\partial }{\partial t_i}+\frac{1}{z}\phi_i\star_{\bt},\\
\nabla_{z{\partial\over \partial z}}&=z{\partial\over \partial z}-{1\over z} \sE(\bt)\star_{\bt}+\gr.
\end{dcases}
\end{equation}
By Proposition~\ref{fundamental-solution-infty}, the quantum connection~\eqref{quantum-connection} admits flat sections $\{S(\bt,z)z^{-\gr}\alpha\}$, where
$S(\bt,z)\in{\rm End}(\cH_{W, \<J\>})\otimes \C[\![\bt]\!][\![z^{-1}]\!]$
is an operator defined by 
$$S(\bt,z)\alpha=\alpha
  +\sum_{i=1}^{r}\sum_{n\geq1}\sum_{k\geq 0}\frac{\phi^i}{n!(-z)^{k+1}}
  \LD
  \alpha\psi_1^k,
  \bt,\dots,\bt,
  \phi_i
  \RD_{0,n+2}, \quad \forall\alpha\in \cH_{W, \<J\>}.
  $$

\subsubsection{Weak asymptotic classes}
We will focus on the asymptotic behavior of the flat sections $\{S(\bt,z)z^{-\gr}\alpha\}$. 
Let us consider the quantum connection by restricting to $\bt=\tau$.
According to Proposition~\ref{prop-euler-tau},
the restriction of $\nabla_{z{\partial\over \partial z}}$ to $\tau$ is 
given by 
\begin{equation}
\label{intro-meromorphic}
\nabla_{z{\partial\over \partial z}}=z{\partial\over \partial z}-{1\over z}\cdot {\nu\over d}\tau'\star_{\tau} +\gr.
\end{equation}
When $W$ is of general type, this is a meromorphic connection that has an {\em irregular} singularity at $z=0$ and a {\em regular} singularity at $z=\infty$. 
In the study of asymptotic analysis of irregular meromorphic connections, it is natural to consider the {\em smallest asymptotic expansion} along a certain ray when $z$ approaches to the irregular singularity. 

\begin{definition}
\label{def-weak-intro}
We say $\alpha\in  \cH_{\rm nar}\subset \cH_{W, \<J\>}$ is a {\em weak asymptotic class} with respect to  $\lambda\in \C^*$ if
 there exists $m\in \mathbb{R}$, such that when $\arg(z)=\arg(\lambda)\in [0, 2\pi)$, 
we have $$\left|e^{\lambda\over z}\cdot \<S(\tau, z)z^{-\gr}\alpha, \one\>\right|=O(|z|^{m})\quad\mathrm{as}\ |z|\to 0.$$ 
\end{definition}
Such asymptotic classes will play the central role in this paper.

\subsubsection{A mirror conjecture}
\label{sec-intro-wall}
The operator $S(\bt, z)$ is invertible. We define the {\em $J$-function}
$$J(\bt, z)=z S(\bt, z)^{-1}(\one).$$
Let $I_{\rm FJRW}^{\rm sm}(t, z)$ be the small $I$-function defined in~\eqref{small-i-function} and $\tau(t)$ be defined in~\eqref{def-tau}. 
These two functions are conjectured to be related by a {\em Givental-type mirror formula}. In particular, we have
\begin{conj}
\label{conjecture-i-function-formula}
If the singularity $W$ is of general type, then 
\begin{equation}
\label{I-J-relation}
I_{\rm FJRW}^{\rm sm}(t, -z)=t J(\tau(t), -z).
\end{equation}
\end{conj}
A general version of this conjecture, where $W$ is not necessarily of general type, has been proved for many examples including the Fermat polynomials~\cite{Aco, CIR, RR, Gue}. 
See Proposition \ref{mirror-theorem-fermat} for a more explicit description.

\subsubsection{Calculation of weak asymptotic classes via a Barnes' formula}
When the mirror formula~\eqref{I-J-relation} holds, we can calculate the weak asymptotic classes using the small $I$-function. 
It is more convenient to consider a modified $I$-function 
$$\widetilde{I}(t,z)=z^{{\widehat{c}_W\over 2}-1}z^{\gr} I_{\rm FJRW}^{\rm sm}(t,z).$$ 
After a change of variables $x=d^{-d}\prod\limits_{j=1}^{N}w_j^{w_j} z^{-\nu}$ and restricting to $t=1$, we will show in Proposition~\ref{i-function-via-ghf} that the coefficients of the modified $I$-function satisfies a generalized hypergeometric equation 
\begin{equation}
\label{ghe-equation-intro}
\Big(x{\partial \over \partial x}\prod_{j=1}^{q}(x{\partial \over \partial x}+\rho_j-1)-x\prod_{i=1}^{p}(x{\partial \over \partial x}+\alpha_i)\Big) \widetilde{I}(1,z(x)) =0,
\end{equation}
where $q+1=|{\bf Nar}|=p+\nu$, and $\rho_j, \alpha_i$'s are determined by the weight system completely. See Section~\ref{sec-basic-ghe} for the details.

For each $\ell \in \Z$, we consider the cohomology class 
\begin{equation}
\label{asymptotic-class-formula}
\cA_\ell:=\sum_{m\in {\bf Nar}} e^{-\ell\cdot m\, {2\pi\sqrt{-1}\over d}}\prod\limits_{j=1}^{N} {2\pi\over \Gamma\left(\left\{{w_j\cdot m\over d}\right\}\right)}\, \sJ_{m}\in \cH_{W, \<J\>}.
\end{equation}
Using an asymptotic expansion formula of generalized hypergeometric functions discovered by Barnes~\cite{Bar} in 1906, we obtain
\begin{proposition}
[Proposition~\ref{theorem-asymptotic-classes}]
\label{wall-to-asymptotic}
Let $(W, \<J\>)$ be an admissible LG pair of general type, with a weight system $(d; w_1, \ldots, w_N)$. If the mirror conjecture~\ref{conjecture-i-function-formula} holds, then for each integer $\ell=1-\nu, \ldots, 0$, the class $\cA_\ell$ in~\eqref{asymptotic-class-formula} is a weak asymptotic class with respect to
$Te^{2\pi\sqrt{-1}\ell/\nu}$, where $T$ is given in \eqref{value-principal-spectrum}.
\end{proposition}

\subsubsection{Strong asymptotic classes}
The result in Proposition \ref{wall-to-asymptotic} suggests that the asymptotic behavior and the conjectured quantum spectrum of ${\nu\over d}\tau'\star_{\tau}$ in Conjecture \ref{conjecture-C} are closely related.
Notice that in Definition~\ref{def-weak-intro}, we only use the asymptotic expansion of $\<S(\tau,z)z^{-\gr}\alpha, \one\>$ to define weak asymptotic classes. 
In general, we follow~\cite{GGI} to consider asymptotic classes (which we call {\em strong asymptotic classes} in this paper) using the asymptotic behavior of the flat section $S(\tau,z)z^{-\gr}\alpha$. 
See Definition~\ref{def-asymptotic-class} for the details.
The spaces of the two types of asymptotic classes are expected to be the same. 
This is true if both quantum spectrum conjecture~\ref{conjecture-C} and mirror conjecture~\ref{conjecture-i-function-formula} hold,
c.f. Proposition~\ref{weak=strong}.

\begin{remark}
Inspired from the work~\cite{Aco, AS}, we expect that the asymptotic classes considered in this paper will correspond to the massive vacuum solutions in the physics literature~\cite{HR}.
\end{remark}

\subsection{$\Gamma$-structures in FJRW theory}
In~\cite{CIR}, Chiodo, Iritani, and Ruan introduced Gamma structures in FJRW theory. 
These structures can be used to define Gamma classes for {\em matrix factorizations}~\cite{CIR} and D-brane central charge in LG models~\cite{KRS}. 
We propose a Gamma conjecture which relates certain Gamma classes to the asymptotic classes of the quantum connection.

\subsubsection{A Gamma map}
The key construction in~\cite{CIR} is a linear map, called a {\em Gamma map}, that gives integral structures for the FJRW state space.
We slightly modify the definition in~\cite{CIR} and define a Gamma map for the LG pair $(W, G)$ by
\begin{equation}
\label{gamma-map-intro}
\widehat{\Gamma}_{W, G}:=\bigoplus_{g\in G} (-1)^{-\mu(g)} \prod_{j=1}^{N}\Gamma(1-\theta_j(g))\cdot {\rm Id}_{\cH_{g}}\in {\rm End}(\cH_{W, G}).
\end{equation}
Here ${\rm Id}_{\cH_{g}}$ is the identity map on $\cH_g$, $\mu(g)$ is the Hodge grading number of $g$~\eqref{hodge-grading-operator}, and the numbers $\theta_j(g)\in [0, 1)\cap\mathbb{Q}$ are defined by 
$$g=\left(e^{2\pi\sqrt{-1}\theta_1(g)}, \ldots, e^{2\pi\sqrt{-1}\theta_N(g)}\right)\in(\C^*)^N.$$

\subsubsection{Gamma classes} 
Let ${\rm MF}_{G}(W)$ be the category of {\em $G$-equivariant matrix factorizations} of $W$ and ${\rm HH}_*({\rm MF}_G(W))$ be the {\em Hochschild homology group} of the category ${\rm MF}_{G}(W)$. 
According to~\cite[Theorem 2.5.4]{PV}, there is an isomorphism
$${\rm HH}_*({\rm MF}_G(W))\cong \cH_{W, G}.$$ 
The map $\widehat{\Gamma}_{W, G}$ can be also viewed as a bridge between ${\rm MF}_G(W)$ and the FJRW theory of $(W, G)$. 
In~\cite{PV}, for an object $\ovE$ in the category ${\rm MF}_{G}(W)$, 
Polishchuk and Vaintrob constructed a Chern character 
${\rm ch}_{G}(\ovE)\in {\rm HH}_0({\rm MF}_{G}(W)).$
We define the {\em Gamma class of the object $\ovE$} to be
$$\widehat{\Gamma}_{W, G}({\rm ch}_{G}(\ovE))\in\cH_{W, G}.$$

Furthermore, Polishchuk and Vaintrob proved a Hirzebruch-Riemann-Roch formula for matrix factorizations, generalizing the work of Walcher~\cite{Wal}.
The formula expresses the Euler characteristic $\chi(\ovE, \ovF)$ of the Hom space ${\rm Hom}^*(\ovE, \ovF)^G$ 
 in terms of a canonical pairing~\eqref{PV-pairing} between ${\rm ch}_{G}(\ovE)$ and ${\rm ch}_{G}(\ovF)$.
We will define a non-symmetric pairing $\left[\cdot, \cdot \right)$ on $\cH_{W, G}$~\eqref{non-symmetric-pairing-qst}. 
Using the HRR formula and the Gamma map~\eqref{gamma-map-intro}, we show in Corollary~\ref{cor-gamma-pairing} that 
$$\left[\widehat{\Gamma}_{W, G}({\rm ch}_{G}(\ovE)), \widehat{\Gamma}_{W, G}({\rm ch}_{G}(\ovF))\right)=\chi(\ovE, \ovF)\in \Z.$$

\subsubsection{Gamma conjectures for admissible LG pairs}
Let $R=\C[x_1, \ldots, x_N]$. Let $\{e_j\}$ be a standard basis of $R^N$ and $\iota(e_j^*)$ be the contraction of the dual element $e_j^*$.
We consider the {\em Koszul matrix factorization}
$$\C^{\rm st}:=\Big(\bigoplus_{k=0}\wedge^k_R R^N,  \quad \delta:=\sum_{j=1}^{N}x_j e_j\wedge+\sum_{j=1}^{N}{\partial W\over \partial x_j}\iota(e_j^*)\Big).$$
Following~\cite{Dyc}, this matrix factorization $\C^{\rm st}$ is called the {\em stabilization of the residue field $\C$}.
The Gamma classes of $\C^{\rm st}$ and its twist $\C(\ell)^{\rm st}$ can be calculated using Polishchuk-Vaintrob's Chern character formula~\cite[Theorem 3.3.3]{PV}. Direct calculation shows that these Gamma classes match the classes defined in~\eqref{asymptotic-class-formula} 
$$\widehat{\Gamma}_{W, \<J\>}\left({\rm ch}_{\<J\>}(\C(\ell)^{\rm st})\right)=\cA_\ell, \quad \forall \ell\in \Z.$$ 

If the LG pair $(W, \<J\>)$ is of general type, some of these Gamma classes has asymptotic meanings.
Inspired by Galkin-Golyshev-Iritani's {\em Gamma Conjecture I} for quantum cohomology of Fano manifold~\cite{GGI}, we propose a Gamma conjecture in FJRW theory.
\begin{conj}
[The weak/strong Gamma Conjecture]
\label{algebraic-analytic}
Let $(W, \<J\>)$ be an admissible LG pair of general type, with index $\nu>0$.
For integers $\ell=1-\nu, \ldots, 0$,
the Gamma class of the matrix factorization $\C(\ell)^{\rm st}$ is a weak/strong asymptotic class with respect to $Te^{2\pi\sqrt{-1}\ell/ \nu}$.
\end{conj}

By Proposition~\ref{wall-to-asymptotic}, 
Mirror conjecture~\ref{conjecture-i-function-formula} implies the weak Gamma Conjecture~\ref{algebraic-analytic}. 
Using Proposition~\ref{mirror-theorem-fermat}, Proposition~\ref{weak=strong}, and Theorem~\ref{main-thm-intro}, we have
\begin{theorem}
\label{corollary-chern-asymptotic}
Let $(W, \<J\>)$ be an admissible LG pair of general type. Then 
\begin{enumerate}
\item The weak Gamma Conjecture~\ref{algebraic-analytic} holds if $W$ is a Fermat polynomial of the form 
$W=\sum\limits_{i=1}^{N} x_i^{a_i}$ with $a_i\geq 2$ for all $i$.
\item The strong Gamma Conjecture~\ref{algebraic-analytic} holds if $W$ is a Fermat homogeneous polynomial of the form
$W=\sum\limits_{i=1}^{N}x_i^d$ with $d-N>1$.
\end{enumerate}
\end{theorem}



\subsection{Orlov's SOD and Stokes phenomenon}
Now we consider the relation between asymptotic classes and  {\em Orlov's semiorthogonal decomposition (SOD)}.
\subsubsection{Orlov's SOD and a Gram matrix}
Let $(W, \<J\>)$ be an admissible LG pair of general type.
Let $\cX_W=(W=0)$ be the hypersurface in the weighted projective space $\bP^{N-1}(w_1, \ldots, w_N)$ and ${\bf D}^b(\cX_W)$ be the derived category of bounded complex of coherent sheaves on $\cX_W.$
According to~\cite{Orl, BFK2}, the triangulated category ${\rm HMF}_{\<J\>}(W)$ admits a semiorthogonal decomposition 
\begin{equation}
\label{intro-sod}
{\rm HMF}_{\<J\>}(W)\simeq\left<\C^{\rm st}, \C(-1)^{\rm st}, \ldots, \C(1-\nu)^{\rm st}, {\bf D}^b(\cX_W)\right>.
\end{equation}
The matrix of the Euler pairing $\chi(\cdot, \cdot)$ of pairs in $\{\C^{\rm st}, \C(-1)^{\rm st}, \ldots, \C(1-\nu)^{\rm st}\}$ forms an upper triangular matrix
$$M:=\Big(\chi(\C(-i)^{\rm st}, \C(-j)^{\rm st})\Big)_{0\leq i, j\leq\nu-1}.$$
This matrix can be calculated by the Hirzebruch-Riemann-Roch formula  in \cite{PV}.  
Let $\left[f(x)\right]_n$ be the coefficient of $x^n$ in the formal power series $f(x)$. 
\begin{proposition}
[Proposition~\ref{inverse-gram-formula}]
The matrix $M^{-1}=(a^{i,j})$ is an upper triangular matrix with integer coefficients. 
Its entries are determined by the weight system of $(W, \<J\>)$: 
$$a^{i, i+n}=\left[\prod_{j=1}^{N}(1-x^{w_j})^{-1}\right]_n, \quad \forall \ 1\leq i\leq \nu, 0\leq n\leq \nu-i.$$ 
\end{proposition}

\subsubsection{Relation to Stokes phenomenon}
The topological aspect of irregular meromorphic connections can be described using {\em Stokes matrices}, see \cite{DM, CV, Dub, Dub-conj, DHMS} for example.
Following the ideas of Dubrovin~\cite{Dub} and Kontsevich, the SOD of ${\rm HMF}_{\<J\>}(W)$ is closely related to the Stokes phenomenon of the meromorphic connection~\eqref{intro-meromorphic}. 
Similar phenomenon in quantum cohomology has been discussed in the literature~\cite{Guz, GGI, CDG, SS}.

In a certain sector of $\arg x\in (a, b)\cap\mathbb{R}$, the ordinary differential equation~\eqref{ghe-equation-intro} admits a basis of analytic solutions given by Meijer $G$-functions.
The transformation laws between such functions in different sectors is linear \cite{Luk, Fie}.
These transformations characterize the Stoke phenomenon of the ODE system. 
We find one transformation law in Proposition~\ref{transformation-exponential-type} which characterizes the behavior of certain Meijer $G$-functions with exponential asymptotic expansions.
We call certain coefficients in Proposition~\ref{transformation-exponential-type} the {\em Stokes coefficients of exponential type}.
\begin{proposition}
[Proposition~\ref{theorem-gram-asymptotic}]
The nonzero entries of the upper triangular matrix $M^{-1}$ are Stokes coefficients of exponential type of the equation~\eqref{ghe-equation-intro}. 
\end{proposition}

For the $r$-spin case, the Stokes matrice can be calculated by the transformation law directly. 
We recover the calculation of the Stokes matrice in~\cite{CV}. 
The computation is similar to that for quantum cohomology of projective spaces in \cite{Guz}.

\subsection{Plan of the paper}
In Section 2, we introduce the construction of quantum product in FJRW theory and then propose a quantum spectrum conjecture for admissible LG pairs $(W, \<J\>)$ of general type.
In Section 3, we define the strong asymptotic classes in FJRW theory using the quantum connection. 
In Section 4, 
we use a formula of Barnes on the asymptotic expansions of generalized hypergeometric functions to calculate weak asymptotic classes.
In Section 5, we review the Gamma structures in FJRW theory, and discuss the relation to HRR formula and asymptotic classes. 
We propose a Gamma conjecture which relates the Gamma classes of certain $\<J\>$-equivariant matrix factorizations to the weak/strong asymptotic classes.
In Section 6, we discuss the relation between the Gamma structures and Orlov's semiorthogonal decompositions for the category of $\<J\>$-equivariant matrix factorizations. We also relate a Gram submatrix with Stokes coefficients of the ODE from the quantum connection.
In Section 7, we prove the  quantum spectrum conjecture for mirror ADE singularities and  Fermat homogeneous polynomials.

\subsection{Acknowlegement} 
We thank Weiqiang He for helpful discussions on Gamma structures and Yang Zhou for helpful discussions on mirror conjecture~\ref{conjecture-i-function-formula}.
Y.S. thank Changzheng Li, Alexander Polishchuk, Mark Shoemaker, and Jie Zhou for helpful discussions.
We also thank Amanda Francis for pointing out some typos in an early version of the paper.
Part of the work was completed while Y.S. was visiting  Simons Center of Geometry and Physics during the program ``Integrability, Enumerative Geometry and Quantization".
Y.S. would like to thank the organizers of the program and the hospitality of Simons Center of Geometry and Physics.
Y.S. thank the hospitality of the Institute for Advanced Study in Mathematics (IASM), Zhejiang University.
Y.S. is partially supported by Simons Collaboration Grant 587119.

\section{Quantum spectrum in FJRW theory}
\label{sec:spectrum}

\subsection{A brief review of FJRW theory}
\label{sec:fjrw}

\subsubsection{Admissible Landau-Ginzburg pairs}

Let $W: \mathbb{C}^N\to \mathbb{C}$  be a degree $d$ quasi-homogeneous polynomial, such that the weights of the variables $w_j:={\rm wt}(x_j)$ satisfying $\gcd(w_1, \ldots, w_N)$=1. 
Following \cite{FJR}, we only consider when $W$ is nondegenerate, that is (1) $W$ has an isolated critical point only at the origin and (2) the choices of $q_j:={w_j\over d}\in\left(0, {1\over 2}\right]\cap \mathbb{Q}$ are unique.

The FJRW theory constructed in~\cite{FJR-book, FJR} is an intersection theory on the moduli space of $W$-spin structures of $G$-orbifold curves, where $G$ is a certain group of symmetry of $W$. 
There are a few choices of the symmetry group $G$, including the group $G_W$ defined in~\eqref{maximal-group}, called the {\em maximal diagonal symmetry group} of $W$. 
\begin{definition}
We say that a subgroup $G\leq G_{W}$ is admissible if it contains the exponential grading element
$$J:=\left(e^{2\pi  \sqrt{-1} q_1}, \ldots, e^{2\pi \sqrt{-1} q_N}\right).$$
We call such a pair $(W,G)$ an {\em admissible LG pair}.
We denote the group generated by $J$  by $\<J\>$, and call it the minimal (admissible) group of $W$.
\end{definition}
According to~\cite[Proposition 3.4]{Kra}, the definition here is equivalent to the original definition in~\cite[Definition 2.3.2]{FJR}, which says there exists a Laurent polynomial $Z$, quasi-homogeneous with the same weights $q_i$ as $W$, but with no monomials in common with $W$, such that $G=G_{W+Z}$.

\subsubsection{The state space}
For each group element $g\in G$, let ${\rm Fix}(g)\subset \CC^n$ be  the $g$-fixed locus and $W_g$ be the restriction of $W$ on ${\rm Fix}(g)$.
Following~\cite[Definition 3.2.1]{FJR}, the {\em $g$-twisted sector} $\cH_g$ is defined to be the $G$-invariant part of the middle-dimensional relative cohomology for $W_g$.
The {\em state space} of the admissible LG pair $(W, G)$ is the {\em FJRW vector space} associated with a nondegenerate pairing $\<\cdot , \cdot \>$ defined by intersecting Lefschetz thimbles~\cite[Section 3.2]{FJR}
\begin{equation}
\label{state-space}
\bigg(\cH_{W,G}:=\bigoplus_{g\in G}\cH_{g}, \quad \< \cdot , \cdot \>\bigg).
\end{equation}

In this paper, we will use an equivalent description of the state space.
According to~\cite[Formula (74)]{FJR}, 
there is a pairing-preserving graded isomorphism 
\begin{equation}
\label{fjrw=orbifold}
\bigg(\cH_{W,G}, \< \cdot , \cdot \>\bigg)\cong\bigg(\bigoplus_{g\in G}\left({\rm Jac}(W_g)\cdot \omega_g\right)^G, \< \cdot , \cdot \>^{\rm Res}\bigg).
\end{equation}
Let us explain the meaning of the right-hand side. Let $\mathfrak{I}_W$ be the {\em Jacobian ideal} in $\C[x_1, \ldots, x_N]$, generated by the partial derivatives ${\partial{W}/ \partial x_j}$ of W. Let 
$${\rm Jac}(W):=\C[x_1, \ldots, x_N]/\mathfrak{I}_W$$ be the {\em Jacobian algeba} (or the {\em Milnor ring}, or the {\em local algebra}) of the singularity $W$.

Let ${\bf x}=(x_{1},\dots, x_{N})$ be a coordinate system on $\C^{N}$.
An element $h\in G$ acts on ${\bf x}$ and the form $\{d{\bf x}\}$ by the scalar products $h\cdot{\bf x}$ and $h\cdot d{\bf x}=d(h\cdot{\bf x})$. 
For each $g\in G$, let 
$\omega_g$ be the standard top form of ${\rm Fix}(g)$. That is, 
$$\omega_g:=\bigwedge_{\{i\mid g\cdot x_i=x_i\}} dx_{i}. $$
We denote the $G$-invariant part of the Jacobian algebra on the fixed locus ${\rm Fix}(g)$ by
\begin{equation}
\label{state-space-component}
\cH_g:=\left(
{\rm Jac}(W_g)\cdot \omega_g
\right)^G.
\end{equation}
Here if ${\rm Fix}(g)$ only contains the origin, that is, 
 ${\rm Fix}(g)=\{0\}$, 
then we say that $\cH_{g}$ is a \emph{narrow} sector.
Otherwise, $\cH_{g}$ is called a \emph{broad} sector.
Each narrow sector $\cH_{g}$ is defined to be one-dimensional, and spanned by the {\em standard generator} $1\vert g\>.$
In particular, if ${\rm Fix}(J^m)=\{0\}$, then we denote the standard generator of $\cH_{J^m}$ by
\begin{equation}
\label{standard-generator}
\sJ_m:=1\vert J^m\>\in \cH_{J^m}.\end{equation}
In general, we can denote an element $\phi_g\in\cH_g$ for any $g\in G$ by 
\begin{equation}\label{broad-element-notation}
\phi_g:=  f\cdot \omega_g \vert g\>, \quad \text{with} \quad f\in {\rm Jac}(W_g).
\end{equation}

The direct sum of all narrow (resp. broad) sectors is denoted by $\cH_{\rm nar}$ (resp. $\cH_{\rm bro}$).
The state space admits a {\em narrow-broad decomposition} 
\begin{equation}
\label{G-orb-space}
\cH_{W, G}=\bigoplus_{g\in G}\cH_g=\cH_{\rm nar}\bigoplus\cH_{\rm bro}.
\end{equation}

The vector space ${\rm Jac}(W)$ (and  ${\rm Jac}(W)\cdot d{\bf x}$) has a {\em Grothendieck residue pairing}
\begin{equation}
\label{residue-pairing}
\<f_1, f_2\>^{\rm Res}=\<f_1\cdot d{\bf x}, f_2\cdot d{\bf x}\>^{\rm Res}:={\rm Res}_{\C[{\bf x}]/\C}
\begin{bmatrix}
f_1f_2d{\bf x}\\
\partial_1W, \ldots, \partial_{n}W
\end{bmatrix}.
\end{equation}
See~\cite[III.9]{Har}.
We denote this pairing on each ${\rm Jac}(W_g)\cdot \omega_g$ by $\<\cdot, \cdot\>^{\rm Res}_{W_g}$.
This induces a canonical pairing $\<\cdot, \cdot\>^{\rm Res}$ on $\cH_{W, G}=\bigoplus_{g\in G}\cH_g$~\cite{PV},
 which is the direct sum of the pairings $\cH_g\otimes\cH_{g^{-1}}\rightarrow \C$ defined by
\begin{equation}
\label{residue-dual}
\<f_1\cdot \omega_g, f_2\cdot \omega_{g^{-1}}\>^{\rm Res}:=\<\zeta_*(f_1\cdot \omega_g), f_2\cdot \omega_{g^{-1}}\>^{\rm Res}_{W_g}.
\end{equation}
Here $\zeta:=(e^{\pi  \sqrt{-1} q_1}, \ldots, e^{\pi \sqrt{-1} q_N})\in(\C^*)^{N}$ is a special square root of $J$. See~\cite[Definition 3.1.1]{FJR}.

\subsubsection{Fan-Jarvis-Ruan-Witten invariants}
The FJRW theory of the admissible pair $(W, G)$
produces a
CohFT~\cite[Definition 4.2.1]{FJR}
\begin{equation}
\label{cohft-notation}\bigg\{\Lambda_{g,k}^{W,G}: \cH_{W,G}^{\otimes k}\to H^*(\overline{\mathcal{M}}_{g,k}, \mathbb{Q})\bigg\}_{2g-2+k>0},
\end{equation}
where each  $\Lambda_{g,k}^{W,G}$ is a multi-linear map from $\cH_{W,G}^{\otimes k}$ to the cohomology of moduli spaces of stable curves of genus $g$ with $k$-markings. These maps are compatible under gluing morphisms. 
Let  $\gamma_i\in \cH_{W,G}$ and let $\psi_i$ be psi classes in $H^*(\overline{\mathcal{M}}_{g,k}, \mathbb{Q})$.
We consider the genus-$g$ FJRW invariants defined in~\cite[Definition 4.2.6]{FJR}
\begin{equation}\label{fjrw-inv}
\LD\gamma_1\psi_1^{\ell_1},\ldots,\gamma_k\psi_k^{\ell_k}\RD_{g,k}^{W,G}
:=\int_{\overline{\mathcal{M}}_{g,k}}\Lambda_{g,k}^{W,G}(\gamma_1,\ldots,\gamma_k) \prod_{i=1}^k\psi_i^{\ell_i}, \quad \text{where} \quad\ell_{i}\in \Z_{\geq 0}.
\end{equation}

Now we recall some nonvanishing criteria for the (genus-zero) FJRW invariants in Lemma~\ref{nonvanishing-lemma} below. 
Later it will be used frequently. 
For more properties, see~\cite[Theorem 4.2.2]{FJR}.

For each $g\in G\subset G_W$, there are unique numbers $\theta_j(g)\in \mathbb{Q}\cap [0, 1)$, such that 
$$g=\left(\exp(2\pi\sqrt{-1} \theta_1(g)), \ldots, \exp(2\pi\sqrt{-1} \theta_N(g))\right)\in (\C^*)^N.$$
We denote {\em the degree shifting number} of $g\in G$ by
\begin{equation}
\label{degree-shifting-number}
\iota(g):=\sum_{j=1}^{N}(\theta_{j}(g)-q_j).
\end{equation}
Let $N_g:=\dim_\C {\rm Fix}(g).$
The {\em complex degree} of a homogeneous element $\phi_g\in \cH_g$ in~\eqref{broad-element-notation} is defined by
\begin{equation}
\label{complex-degree}
\deg_\C\phi_g:={N_g\over 2}+\iota(g).
\end{equation}
Note that the $g$-twisted sector $\cH_g$ (and any element in it) is {\em narrow} if and only if $N_g=0$. 

We will denote the  {\em central charge} of the singularity $W$ by 
\begin{equation}
\label{central}
\widehat{c}_W:=\sum_{j=1}^{N}\left(1-{2w_j\over d}\right) \in \mathbb{Q}_{\geq 0}.
\end{equation}

\begin{lemma}
\label{nonvanishing-lemma}
If the genus-zero FJRW invariant 
$$\LD\phi_{g_1}\psi_1^{\ell_1},\ldots,\phi_{g_k}\psi_k^{\ell_k}\RD_{0, k}^{W, G}\neq 0,$$ 
then 
\begin{eqnarray}
\widehat{c}_W-3+k=\sum_{i=1}^{k}(\deg_\C\phi_{g_i}+\ell_i),\quad\text{and} \label{virtual-degree}\\
{w_j\over d}\left(k-2\right)-\sum_{i=1}^{k}\theta_{j}(g_i)\in \Z. 
\label{selection-rule}
\end{eqnarray}
\end{lemma}
Here the equation~\eqref{virtual-degree} follows from~\cite[Theorem 4.1.8 (1)]{FJR}. We call it the {\em homogeneity property} or the degree constraint  (in genus zero). 
The equation~\eqref{selection-rule} is called the {\em selection rule}~\cite[Proposition 2.2.8]{FJR}.

\subsubsection{A quantum product}

We denote the rank of $\cH_{W, G}$ by $r$ and fix a homogeneous basis $\{\phi_i\}_{i=1}^{r}$ of $\cH_{W, G}$.
Let
$${\bf t}(z):=\sum_{m\geq 0}\sum_{i=1}^{r}t_{i, m}\phi_i z^m\in \cH_{W,G}[\![z]\!].$$
It is convenient to use the following double bracket notation to define a {\em correlation function} (as a formal power series of variables $\{t_{i, m}\}$)
\begin{align*}
&\LL\gamma_1\psi_1^{\ell_1},\ldots,\gamma_k\psi_k^{\ell_k}\RR_{g,k}^{W,G}(\bt(\psi))
\\:=&\sum_{n\geq 0}{1\over n!}\LD\gamma_1\psi_1^{\ell_1},\ldots,\gamma_k\psi_k^{\ell_k}, \bt(\psi_{k+1}), \ldots, \bt(\psi_{k+n})\RD_{g,k+n}^{W,G}.
\end{align*}

In this paper, we only consider genus zero correlation functions.
We restrict to $$\bt:=\bt(0)=\sum\limits_{i=1}^{r}t_i\phi_i, \quad \text{with} \quad t_i:=t_{i, 0}.$$
The FJRW invariants induce a {\em quantum product} $\star_{\bt}$ defined by 
\begin{equation}
\label{quantum-product}
\<\phi_i\star_{\bt}\phi_j, \phi_k\>
=\LL\phi_i, \phi_j, \phi_k\RR^{W, G}_{0,3}(\bt)\in \C[\![t_1, \ldots, t_r]\!].
\end{equation}
By the WDVV equations, the quantum product is associative.
Recall that the narrow element $\sJ_1$ is the unit of the quantum product $\star$~\cite[Theorem 4.2.2]{FJR}. 
We sometimes denoted it by 
$\one:=\sJ_1$.

\subsection{A quantum spectrum conjecture in FJRW theory}

Now we consider the admissible LG pairs $(W, \<J\>)$. 
We recall that for an admissible LG pair $(W, \<J\>)$, its weight system is defined in~\eqref{weight-system}. 
Let us denote its {\em weight system} by 
$$(d; w_1, \ldots, w_N)\in (\mathbb{Z}_{>0})^{N+1}.$$
Recall 
According to Definition~\ref{index-lg-pair}, the {\em index} of the LG pair $(W, \<J\>)$ is given by 
$$\nu=d-\sum\limits_{j=1}^{N}w_j,$$
and we say that the LG pair is {\em Fano/CY/general type} if $\nu$ is negative/zero/positive.


\subsubsection{The element $\tau(t)$}

Recall that the element $\tau(t)$ is defined in~\eqref{def-tau}.
We write
\begin{equation}
\label{tau-component}
\tau(t):=\sum_{m\in {\bf Mir}}\tau_m(t)\ \sJ_m.
\end{equation}
Here  ${\bf Mir}=\{m\in {\bf Nar}\mid \tau_m(t)\neq 0\}$.
In particular, it follows from formula~\eqref{small-i-function} that $1\in {\bf Mir}$ if and only if $\nu=1.$
Let 
\begin{equation}
\label{condition-m-in-tau}
{\bf Mir}_{\geq 2}:=\left\{m\in {\bf Nar}\ \bigg\vert\  m-1-\sum_{j=1}^{N}\left \lfloor {w_j\cdot m\over d}\right\rfloor=1\right\}\subset {\bf Nar}.
\end{equation}
Then we have 
$${\bf Mir}=
\begin{dcases}
\{1\}\cup {\bf Mir}_{\geq 2}, & \text{ if } \nu=1;\\
 {\bf Mir}_{\geq 2}, & \text{ if } \nu>1.
\end{dcases}
$$
We also have 
\begin{equation}
\label{coefficient-tau}
\tau_m(t)=
\begin{dcases}
{t^{d}\over d!}\prod_{j=1}^{N}
{\Gamma(w_j+{w_j\over d})\over 
\Gamma\left({w_j\over d}\right)}, & \text{ if } m=1;\\
{t^{m-1}\over (m-1)!}\prod_{j=1}^{N}
{\Gamma({w_j\cdot m\over d})\over 
\Gamma\left(\left\{{w_j\cdot m\over d}\right\}\right)}, & \text{ if } m>1.
\end{dcases}
\end{equation}

\begin{example}
If $W=\sum\limits_{i=1}^{N}x_i^d$ is a Fermat homogeneous polynomial of general type, then
\begin{equation}
\label{tau-fermat}
\tau(t)=
\begin{dcases}
t\sJ_2, & \textit{ if } d>N+1;\\
t\sJ_2+\frac{t^{d}}{d!d^N}\sJ_1, & \textit{ if } d=N+1>2;\\
{t^2\over 4}\sJ_1, & \textit{ if } d=N+1=2.
\end{dcases}
\end{equation}
\end{example}

The following example shows that the set ${\bf Mir}$ can contain many elements. 
\begin{example}
The $Q_{11}$-singularity $W=x_1^2x_2+x_2^3x_3+x_3^3$ has the weight system $(18; 7, 4, 6)$ and the index $\nu=1.$ We have 
$${\bf Mir}=\{1, 2, 4, 5, 7, 10, 11, 13, 14, 16\}={\bf Nar}\backslash\{17\}.$$
\end{example}

Using formulas~\eqref{degree-shifting-number} and~\eqref{complex-degree}, we have 
\begin{equation}
\label{degree-narrow}
1-\deg_\C\sJ_m=1-\sum_{j=1}^{N}\left(\{{w_j\cdot m\over d}\}-{w_j\over d}\right)
=\begin{dcases}
1, & \text{ if } m=1;\\
{\nu(m-1)\over d}, & \text{ if } m\in {\bf Mir}_{\geq 2}.
\end{dcases}
\end{equation}
Here the last equality follows from the equality~\eqref{condition-m-in-tau} when $m\neq 1$.

\begin{proposition}
\label{prop-euler-tau}
If $\nu>0$, then the Euler vector field $\sE(\tau(t))$ takes the form
\begin{equation}
\label{euler-tau}\sE(\tau(t))={\nu\over d}t\tau'(t).
\end{equation}
\end{proposition}
\begin{proof}
By the definition in ~\eqref{euler-vector}, we have 
$$\sE(\tau(t))=\sum_{m\in {\bf Mir}}(1-\deg_\C\sJ_m)\tau_m(t)\sJ_m.$$
Now the result follows directly from~\eqref{degree-narrow} and~\eqref{coefficient-tau}.
\end{proof}

\subsubsection{A quantum product ${\nu\over d}\tau'\star_\tau$}
For simplicity, we will write 
$$\tau:=\tau(1) \quad \text{and} \quad \tau':=\tau'(1).$$
In the following, we will focus on the quantum multiplication ${(\nu/d)}\tau'\star_\tau$.
The following proposition shows that
\begin{equation}
\label{c1-avatar}
{\nu\over d}\tau'\star_\tau\in  {\rm End}(\cH_{W,  \<J\>}).
\end{equation}
\begin{proposition}
\label{prop-convergence}
Let $(W, \<J\>)$ be an admissible LG pair of general type, then 
$${\nu\over d}t\tau'(t)\star_{\tau(t)} \in \C[t]\otimes_{\C} {\rm End}(\cH_{W,  \<J\>}).$$
\end{proposition}
\begin{proof}
For any pair of homogeneous elements $\phi_i, \phi_j\in \cH_{W,  \<J\>}$, we will prove
$$\LL \tau'(t), \phi_i, \phi_j\RR_{0,3}(\tau(t))\in \C[t].$$ 
Using~\eqref{tau-component} and~\eqref{coefficient-tau}, it suffices to show that for any $m_1, \ldots, m_n\in {\bf Mir}$,
\begin{equation}
\label{polynomiality-tau}
\LD\sJ_{m_1}, \phi_i, \phi_j, \sJ_{m_2}, \ldots, \sJ_{m_n}\RD_{0, n+2}
\end{equation}
vanishes if $n$ is large enough.
According to the {\em string equation}~\cite[Theorem 4.2.9]{FJR}, this is true if $n\geq 2$ and at least one of $m_1, \ldots, m_n$ is $1.$ So we can assume $m_k\geq 2$ for each $k.$
Now according to the homogeneity property~\eqref{virtual-degree}, if such an FJRW invariant~\eqref{polynomiality-tau} is nonzero, then we have 
\begin{eqnarray*}
\widehat{c}_W-3+n+2
&=&\deg_\C\phi_i+\deg_\C\phi_j+\sum_{k=1}^{n}\deg_\C\sJ_{m_k}\\
&=&\deg_\C\phi_i+\deg_\C\phi_j+n-\sum_{k=1}^{n}{\nu(m_k-1)\over d}\\
&\leq&\deg_\C\phi_i+\deg_\C\phi_j+n-{\nu\over d}n.
\end{eqnarray*}
Here the second equality follows from the degree formula~\eqref{degree-narrow}, and the last inequality holds as we assume $m_k\geq 2$ for each $k$.
Since $\nu>0$,  we have $$n\leq {d\over \nu}\left(\deg_\C\phi_i+\deg_\C\phi_j+1-\widehat{c}_W\right).$$
Thus $n$ is bounded.
\end{proof}

\subsubsection{A conjecture on quantum spectrum for admissible LG pairs}
In~\eqref{def-evalue}, we denote the set of eigenvalues of the quantum multiplication ${\nu\over d}\tau'\star_{\tau}$ on $\cH_{W, \<J\>}$ by $\fE$ and refer to it as the
\emph{quantum spectrum}.
Inspired by Galkin-Golyshev-Iritani's Conjecture $\mathcal{O}$ for quantum cohomology of Fano manifolds~\cite{GGI}, we propose the following {\em quantum spectrum conjecture for admissible LG pair $(W, \<J\>)$ of general type}.
\begin{conj}
\label{quantum-spectrum-conj}
Let $(W, \< J\>)$ be an admissible LG pair with $W$ a nondegenerate quasihomogeneous polynomial of general type.
Let $(d; w_1, \ldots, w_N)$ be its weight system. 
Then the set $\fE$ contains the number
\begin{equation}
\label{largest-positive-number}
T:=\nu\left(d^{-d}\prod_{j=1}^Nw_j^{w_j}\right)^{1\over \nu}\in \mathbb{R}_{>0}
\end{equation}
such that:
\begin{enumerate}
\item for any $\lambda\in 
\fE$, we have $|\lambda|\leq T$;
\item the following two sets are the same:
$$\Big\{\lambda\in \fE\mid |\lambda|=T\Big\}=\Big\{Te^{2\pi\sqrt{-1}j/\nu}\mid j=0, 1, \ldots, \nu-1\Big\};$$
\item the multiplicity of each $\lambda\in \fE$ with $|\lambda|=T$ is one.
\end{enumerate} 
\end{conj}

If the conjecture holds, we call $T$ in~\eqref{largest-positive-number} the {\em principal} eigenvalue of ${\nu\over d}\tau'\star_{\tau}$. 
Part (3) of Conjecture~\ref{quantum-spectrum-conj} is also called the {\em multiplicity-one conjecture}.

\subsubsection{Evidence of Conjecture~\ref{quantum-spectrum-conj}}
In Section~\ref{sec-evidence}, we will check Conjecture~\ref{quantum-spectrum-conj} in special cases.
If $W$ is a mirror simple singularity as in~\eqref{mirror-ade-singularity}, the conjecture follows from direct calculations of the quantum products.  
If $W$ is a Fermat homogeneous polynomial with $\nu>1$, the conjecture follows from the result below:
\begin{proposition}
\label{prop-quantum-spectrum-inv}
For the Fermat homogeneous polynomial of general type
$$W=x_1^d+\ldots +x_N^d, \quad \text{with} \quad \nu=d-N>1,$$
the quantum spectrum of ${\nu \over d}\tau'\star_{\tau}$ for the LG pair $(W, \<J\>)$ is given by 
$$
\fE=
\begin{dcases}
\{\nu d^{-d/\nu}e^{2\pi\sqrt{-1}j/\nu} \mid j=0, 1, \ldots, \nu-1\}, & \text{if } d>N=1,\\
\{0, \nu d^{-d/\nu}e^{2\pi\sqrt{-1}j/\nu} \mid j=0, 1, \ldots, \nu-1\}, & \text{if } d >N>1.
\end{dcases}
$$
Moreover, all nonzero eigenvalues have multiplicity one. 
\end{proposition}
The proof of this statement uses mirror symmetry and will be given in Section~\ref{sec-spectrum-fermat}.


\subsection{A variant of the quantum spectrum conjecture on a subalgebra}
\label{sec-variant}
The quantum multiplication on the state space $\cH_{W, \<J\>}$ is difficult to handle in general, especially when the subspace of broad sectors has a large dimension. 
When studying the Gromov-Witten theory of a hypersurface $X$, instead of considering the quantum product on the whole cohomology space $H^*(X)$, one often consider the quantum product on the subspace $H^*_{\rm amb}(X)$ of ambient cohomology classes. Here in FJRW theory, we will also consider the quantum product on a special subspace of $\cH_{W, \<J\>}$.

\subsubsection{The $G_W$-invariant subspace}
For any admissible LG pair $(W, G)$, we notice that the group $G_W$ acts on the state space $\cH_{W, G}$. We denote the $G_W$-invariant subspace by
$$\cH_{W, G}^{G_W}:=\{\phi\in \cH_{W, G}\mid g\cdot\phi=\phi, \forall g\in G_W\}.$$
\begin{lemma}
By viewing $\cH_{W, G}$ and $\cH_{W, G_W}$ as subspaces of $\bigoplus_{g\in G_{W}}{\rm Jac}(W_g)\cdot \omega_g$,
we have
$$\cH_{W, G}^{G_W}=\cH_{W, G}\cap\cH_{W, G_W}.$$
In particular, we have 
$\cH_{W, G_W}^{G_W}=\cH_{W, G_W}.$
\end{lemma}
\begin{proof}
Since $G\subset G_{W}$, each $G_W$-invariant element $\phi_g\in \cH_{W, G}$ belongs to $\cH_{W, G_{W}}$. Thus
$$\cH_{W, G}^{G_{W}}\subset\cH_{W, G}\cap \cH_{W, G_{W}}.$$
On the other hand, if $\phi_g\in {\rm Jac}(W_g)\cdot \omega_g\cap \cH_{W, G}\cap \cH_{W, G_{W}}$, then $g\in G\cap G_W=G$, and $\phi_g\in \cH_{W, G}$ is also $G_W$-invariant.
\end{proof}

\begin{example}
The vector space $\cH_{W, G}^{G_{W}}$ depends on choices of both $G$ and $G_{W}$.
In fact, the $G_{W}$-invariant subspaces $\cH_{W_{1}, \<J\>}^{G_{W_{1}}},\cH_{W_{2}, \<J\>}^{G_{W_{2}}}$ may be different even if the LG pairs $(W_{1}, \<J\>),(W_{2}, \<J\>)$ have the same weight system.
For example, we consider the following two LG pairs, 
$$\begin{dcases}
(W_1=x_1^4+x_2^8, \<J\>), \\
(W_2=x_1^4+x_1x_2^6, \<J\>).
\end{dcases}$$ 
Both LG models have the weight system $(8;2,1)$. However, 
$$\dim_\C\cH_{W_1, \<J\>}^{G_{W}}=6,\quad \dim_\C\cH_{W_2, \<J\>}^{G_{W}}=7.$$
Both state spaces contain narrow elements $\sJ_1, \sJ_2, \sJ_3, \sJ_5, \sJ_6, \sJ_7$. 
The latter also contains a broad element $x_2^5dx_1dx_2\vert J^0\>.$

Here is another example. 
Consider two LG pairs, 
$$
\begin{dcases}
(W_1=x_1^5+x_1x_2^2, \<J\>), \\
 (W_2=x_1^3x_2+x_1x_2^2, \<J\>=G_{W_{2}}).
\end{dcases}$$
Both LG pairs have the weight system $(5; 1, 2)$.
Both state spaces contain narrow elements $\sJ_1, \sJ_2, \sJ_3, \sJ_4$.
However, $\cH_{W_1, \<J\>}$ is 5-dimensional as it contains one broad element $x_2dx_1dx_2\vert J^0\>$, while $\cH_{W_2, \<J\>}$ is 6-dimensional as it contains two broad elements: $x_1^2dx_1dx_2\vert J^0\>$ and $x_2dx_1dx_2\vert J^0\>$.


Furthermore, the following three examples have the same weight system $(3;1,1)$. The associated spaces $\cH_{W, \<J\>}$ have the same dimension, whereas their $G_W$-invariant subspaces $\cH_{W, \<J\>}^{G_W}$ do not. Notice that $\omega_{J^0}=dx_1\wedge dx_2.$

\begin{table}[h]
\label{table-3-1-1}
\begin{center}
\begin{tabular}{|l|c|c|}
 \hline
$W$& basis of $\cH_{W, \<J\>}$&basis of $\cH_{W, \<J\>}^{G_W}$\\
\hline
$x_1^3+x_2^3$&$\{\sJ_1,\sJ_2,x_1\omega_{J^0}\vert J^0\>,x_2\omega_{J^0}\vert J^0\>\}$&$\{\sJ_1,\sJ_2\}$\\
\hline
$x_1^3+x_1x_2^2$&$\{\sJ_1,\sJ_2,x_1\omega_{J^0}\vert J^0\>,x_2\omega_{J^0}\vert J^0\>\}$&$\{\sJ_1,\sJ_2, x_2\omega_{J^0}\vert J^0\>\}$\\
\hline
$x_1^2x_2+x_1x_2^2$&$\{\sJ_1,\sJ_2,x_1\omega_{J^0}\vert J^0\>,x_2\omega_{J^0}\vert J^0\>\}$&$\{\sJ_1,\sJ_2,x_1\omega_{J^0}\vert J^0\>,x_2\omega_{J^0}\vert J^0\>\}$\\
\hline
\end{tabular}
\end{center}
\end{table}

\end{example}

\subsubsection{A comparison between the narrow subspace and $\cH_{W, \<J\>}^{G_{W}}$}
Let $(d; w_1, \ldots, w_N)$ be the weight system of the LG pair $(W,\<J\>)$. 
Recall that ${\bf Nar}$ is the set of narrow indices defined in~\eqref{narrow-index}.
We 
denote the space of narrow sectors by 
$$\cH_{\rm nar}=\bigoplus_{m\in {\bf Nar}}\cH_{J^m}.
$$ 
It is easy to see
\begin{proposition}
We have an inclusion $\cH_{\rm nar}\subset \cH_{W, \<J\>}^{G_{W}}$.
\end{proposition}
In general, $\cH_{\rm nar}\neq \cH_{W, \<J\>}^{G_{W}}$; here is an example.
\begin{example}
For the $E_7$-singularity $W=x_1^3x_2+x_2^3$, we have 
$$\cH_{W, \<J\>}^{G_{W}}=\cH_{\rm nar}\bigoplus \C\cdot x_2^2dx_1\wedge dx_2\vert J^0\>.$$
\end{example}
For Fermat polynomials, the two spaces are the same.
\begin{proposition}
\label{fermat-broad}
If $W=x_1^{a_1}+\cdots +x_N^{a_N}$, 
then $\cH_{W, \<J\>}^{G_{W}}=\cH_{\rm nar}$.
\end{proposition}
\begin{proof}
It is sufficient to show $\cH_{W, \<J\>}^{G_{W}}\subset\cH_{\rm nar}$. 
We assume there exists some $f\neq 0$ and $g=J^m$, $m\notin {\bf Nar}$, such that $f\cdot \omega_g\vert g\>\in\cH_g\subset \cH_{W, \<J\>}^{G_{W}}$.
Since $m\notin {\bf Nar}$, without loss of generality, we can assume ${\rm Fix}(J^m)={\rm Span} \{x_1, \cdots, x_k\}$ for some $k\geq 1$. 
Since $W$ is a Fermat polynomial, we have 
$${\rm Jac}(W_g)=\C[x_1, \ldots, x_k]/(x_1^{a_1-1}, \cdots, x_k^{a_k-1}).$$ 
So $f\cdot \omega_g$ must be a scalar multiple of  $\prod\limits_{i=1}^{k}x_i^{b_i}dx_i$ for some $b_i\in\Z$ such that $0\leq b_i\leq a_i-2$.
But $\prod\limits_{i=1}^{k}x_i^{b_i}dx_i$ is not $G_W$-invariant as the action of the element $(e^{2\pi\sqrt{-1}/a_1}, 1, \cdots, 1)$ in $G_W$ sends $\prod\limits_{i=1}^{k}x_i^{b_i}dx_i$ to 
$$e^{2\pi\sqrt{-1}(b_1+1)/a_1} \cdot \prod\limits_{i=1}^{k}x_i^{b_i}dx_i\neq \prod\limits_{i=1}^{k}x_i^{b_i}dx_i.$$ 
This is a contradiction. Hence we must have $\cH_{W, \<J\>}^{G_{W}}\subset\cH_{\rm nar}$.
\end{proof}

\subsubsection{A quantum multiplication on $\cH_{W, \<J\>}^{G_W}$}

The following vanishing result holds for any admissible LG pair $(W, G)$.

\begin{lemma}
\label{gmax-inv-inv}
For any admissible pair $(W, G)$, let $\alpha\in \cH_{W, G}$ be a homogeneous element which is not $G_W$-invariant. If $\gamma_j \in \cH_{W, G}^{G_{W}}$ for all $j=1, \ldots, n$, then 
\begin{equation}
\label{vanishing-non-invariant}
\<\gamma_1, \ldots, \gamma_n, \alpha\>_{0, n+1}=0.
\end{equation}
\end{lemma}
\begin{proof}
By the choice of the insertions, the FJRW invariant $\<\gamma_1, \ldots, \gamma_n, \alpha\>_{0, n+1}$ is not $G_W$-invariant unless it vanishes. 
On the other hand, according to  the {\em $G_{W}$-invariance axiom} in~\cite[Theorem 4.1.8 (10)]{FJR}, the underlying virtual cycle is invariant under the $G_W$-action. So the invariant must vanish.
\end{proof}

\begin{remark}
Lemma~\ref{gmax-inv-inv} is similar to Proposition 2.8 in~\cite{CIR}. The result there is stated for any admissible LG pair $(W, \<J\>)$ with $W$ a Fermat polynomial.
The narrow subspace in~\cite[Proposition 2.8]{CIR}  is replaced by the $G_W$-invariant subspace $\cH_{W, G}^{G_{W}}$ in Lemma~\ref{gmax-inv-inv}. \end{remark}

Now we use Lemma~\ref{gmax-inv-inv} to study the spectrum of the quantum multiplication ${\nu\over d}\tau'\star_\tau$ on the subspace $\cH_{W, \<J\>}^{G_W}$ when the admissible LG pair $(W, \<J\>)$ is of general type.
We remark that by the definition~\eqref{tau-component} and the explicit formulas~\eqref{coefficient-tau}, we have  $\tau,\tau'\in \cH_{W, \<J\>}^{G_W}.$  
Let $\alpha\in \cH_{W, \<J\>}$ be a homogeneous element which is not $G_W$-invariant.
Because both $\tau', \tau\in \cH_{W, \<J\>}^{G_{W}}$, we see that for any $\gamma\in \cH_{W, \<J\>}^{G_{W}}$,
\begin{equation}
\label{fake-broad-vanishing}
\LL\tau', \gamma, \alpha\RR_{0, 3}(\tau)=0.
\end{equation}

Since the flat identity $\one=\sJ_1\in \cH_{W, \<J\>}^{G_{W}}$ and $\<\cdot, \cdot \>=\<\one, \cdot, \cdot\>_{0,3}$~\cite[Section 4.2 (68)]{FJR}, the pairing $\<\cdot, \cdot \>$ is also $G_W$-invariant.
Hence the pairing~\eqref{residue-dual} induces a perfect pairing on $\cH_{W,  \<J\>}^{G_{W}}$.
Using Lemma~\ref{gmax-inv-inv} 
and a similar argument as in Proposition~\ref{prop-convergence}, we have 
\begin{proposition}
\label{closed-C-algebra}
 $\left(\cH_{W,  \<J\>}^{G_{W}}, \star_\tau\right)$ is a subalgebra of $\left(\cH_{W,  \<J\>}, \star_\tau\right)$.
\end{proposition}

Combining~\eqref{fake-broad-vanishing} with~\eqref{quantum-product}, we obtain
\begin{proposition}
\label{prop-narrow-closed}
The multiplication $\tau'\star_\tau$ is closed on $\cH_{W,  \<J\>}^{G_{W}}$. That is, 
$\tau'\star_{\tau}\in {\rm End}(\cH_{W,  \<J\>}^{G_{W}}).$
\end{proposition}
Similar to~\cite[Proposition 7.1]{GI}, we can consider the quantum spectrum of $\tau'\star_{\tau}$ on the subspace $\cH_{W,  \<J\>}^{G_{W}}$.
\begin{proposition}
\label{quantum-spectrum-invariant}
If $\lambda\in \C$ is an eigenvalue of  $\tau'\star_{\tau}$ on $\cH_{W,  \<J\>}$, it is also an eigenvalue of $\tau'\star_{\tau}$ on $\cH_{W,  \<J\>}^{G_{W}}$.
\end{proposition}
\begin{proof}
We have an injective algebra homomorphism $\cH_{W,  \<J\>}^{G_{W}}\hookrightarrow  {\rm End}(\cH_{W,  \<J\>})$ given by $a\mapsto a\star_{\tau}$.
Let $\lambda\in \C$ be an eigenvalue of $\tau'\star_{\tau}$ on $\cH_{W,  \<J\>}$. Then as an endomorphism of $\cH_{W,  \<J\>}$, $\tau'\star_{\tau}-\lambda$ is not invertible.
This implies that $\tau'-\lambda {\bf 1}$ is not invertible in $\cH_{W,  \<J\>}^{G_{W}}$ as well. Hence
$\lambda$ is an eigenvalue of
$\tau'\star_{\tau}$ on $\cH_{W,  \<J\>}^{G_{W}}$.
\end{proof}

%

\subsubsection{A quantum spectrum conjecture on the subspace $\cH_{W, \<J\>}^{G_W}$}
Based on Proposition~\ref{quantum-spectrum-invariant}, we propose a variant of the quantum spectrum conjecture for the LG pair $(W, \<J\>)$ on the $G_W$-invariant subspace $\cH_{W, \<J\>}^{G_W}$. 
We denote the set of eigenvalues of ${\nu\over d}\tau'\star_{\tau}$ on the subspace $\cH_{W, \<J\>}^{G_W}$ by $\fE_{G_W}$.
\begin{conj}
\label{quantum-spectrum-conj-invariant}
Let $(W, \< J\>)$ be an admissible LG pair with $W$ a nondegenerate quasihomogeneous polynomial of general type.
The set $\fE_{G_W}$ contains the positive real number
$$
T:=\nu\left(d^{-d}\prod_{j=1}^Nw_j^{w_j}\right)^{1\over \nu}
$$ that satisfies the following conditions:
\begin{enumerate}
\item For any $\lambda\in \fE_{G_W}$, we have $|\lambda|\leq T$.
\item The following two sets are the same:
$$\{\lambda\in \fE_{G_W}\mid |\lambda|=T\}=\{Te^{2\pi\sqrt{-1}j/\nu}\vert j=0, \ldots, \nu-1\}.$$ 
\item The multiplicity of each $\lambda\in \fE_{G_W}$ with $|\lambda|=T$ is one.
\end{enumerate} 
\end{conj}

This conjecture holds if $W$ is a Fermat homogeneous polynomial of general type, that is, $W=x_1^d+\ldots +x_N^d$ and $\nu=d-N>0$.
A proof will be given in Proposition~\ref{quantum-spectrum-fermat-invariant}.
If $\nu>1$, then the stronger version Conjecture~\ref{quantum-spectrum-conj} holds.



\section{Strong asymptotic classes in FJRW theory of general type}
\label{sec-asymptotic}
In this section, we study the asymptotic expansions in FJRW theory.
The eigenvalues that appear in the quantum spectrum conjectures will lead to some interesting asymptotic classes in the state space.
In Section~\ref{sec:connection}, we consider the quantum connection in FJRW theory and its basic properties.
In Section~\ref{sec-asymptotic-classes}, we define the {\em strong asymptotic classes} in FJRW theory using the asymptotic behavior of the quantum connection.
We also describe some properties of these asymptotic classes.

\subsection{Quantum connection and the $J$-function}
\label{sec:connection}
The genus-zero FJRW theory defines a formal Dubrovin-Frobenius manifold structure on $\cH_{W, G}$~\cite[Corollary 4.2.8]{FJR}.
Similar to the {\em Dubrovin connection} in GW theory, there is a quantum connection in FJRW theory, defined using the quantum product $\star_{\bt}$ in~\eqref{quantum-product}.
Let us describe this quantum connection and its properties.

\subsubsection{Hodge grading operator}
For each $g\in G$ and a homogeneous element $\phi_g\in \cH_g$, 
we define the {\em Hodge grading number} by
\begin{equation}
\label{hodge-grading-operator}
\mu(\phi_g)=\mu(g)
:={N_g\over 2}+\iota(g)-{\widehat{c}_W\over 2}
={N_g-N\over 2}+\sum_{j=1}^{N}\theta_j(g).
\end{equation}
We define the {\em Hodge grading operator} 
$\gr \in {\rm End}(\cH_{W, G})$
by extending the map $\gr\in {\rm End}(\cH_g)$ below to $\cH_{W, G}$ linearly:
\begin{equation}
\label{eq:modified-hodge-operator}
\gr(\phi_g):=\mu(g)\phi_g.
\end{equation}

According to~\eqref{complex-degree}, we have an identity
\begin{equation}
\label{degree-hodge}
\deg_\C\phi_g={\widehat{c}_W\over 2}+\mu(g).
\end{equation}
\begin{lemma}
Let $\phi_g\in \cH_g$ and $\phi_{g'}\in\cH_{g'}$ be a pair of homogeneous elements. If the pairing $\<\phi_g, \phi_{g'}\>\neq 0$, then $g'=g^{-1}$ and \begin{equation}
\label{adjoint-hodge-operator}
\mu(g)+\mu(g')=0.
\end{equation}
\end{lemma}
\begin{proof}
The first equality $g'=g^{-1}$ follows from the definition of the pairing.
We also have $N_g=N_{g^{-1}}$ and
$$
\theta_j(g)+\theta_j(g^{-1})=
\begin{dcases}
1, & \text{ if } g \text{ does not fix } x_j;\\
0, & \text{ if } g \text{ fixes } x_j.
\end{dcases}
$$
Using the Hodge grading number formula~\eqref{hodge-grading-operator}, we must have
$$\mu(g)+\mu(g')={N_g-N\over 2}+\sum_{j=1}^{N}\theta_j(g)+{N_{g^{-1}}-N\over 2}+\sum_{j=1}^{N}\theta_j(g^{-1}).$$
The result follows from the fact that 
$\theta_j(g)+\theta_j(g^{-1})=0$ if $g$ fixes $x_j$ and $\theta_j(g)+\theta_j(g^{-1})=1$ otherwise. 
\end{proof}

\subsubsection{A non-symmetric pairing in FJRW theory}
Recall that $\C^g\subset \C^N$ be the $g$-invariant subspace and $\C^N/\C^g$ be the quotient space for any $g\in G$.
Let $\phi_g\in \cH_g$ be the component of $\phi\in\cH_{W, G}$ in the sector $\cH_g$.
We define a non-symmetric pairing on the state space $\cH_{W, G}$
$$\left[ \cdot, \cdot \right): \cH_{W,G}\times\cH_{W,G}\to \C,$$ 
given by 
\begin{equation}\label{non-symmetric-pairing-qst}
\left[\phi, \phi'\right)={1\over |G| }\sum_{g\in G}{(-1)^{-\mu(g)}\over (-2\pi)^{N-N_g}}\left\<\phi_{g^{-1}}, \phi'_{g}\right\>.
\end{equation}
In particular, if both $\phi, \phi'\in \cH_{\bf Nar}$, we obtain
$$\left[\phi, \phi'\right)={1\over |G|\cdot (-2\pi)^N}\sum_{g\in {\bf Nar}}(-1)^{-\mu(g)}\left\<\phi_{g^{-1}}, \phi'_{g}\right\>.$$

\subsubsection{A quantum connection}
For a homogeneous basis $\{\phi_i\}_{i=1}^{r}$ of $\cH_{W, G}$,
we will denote its dual basis by $\{\phi^i\}_{i=1}^{r}$, such that $\<\phi_i, \phi^j\>=\delta_i^j.$
Let 
\begin{equation}
\label{basis-element-form}
\mathbf{t}=\sum_{i=1}^{r}t_i\phi_i\in \cH_{W, G}.
\end{equation}
Let $\gr$ be the Hodge grading operator in~\eqref{eq:modified-hodge-operator} and  $\sE(\bt)$ be the Euler vector field in~\eqref{euler-vector}.
There is a quantum connection $\nabla$ on the trivial bundle $T\cH_{W, G}\times \mathbb{C}^*\to \cH_{W, G}\times \mathbb{C}^*$, given by
\begin{equation}
\label{quantum-connection-formula}
\begin{dcases}
\nabla_{\partial\over \partial t_i}&=\frac{\partial }{\partial t_i}+\frac{1}{z}\phi_i\star_{\bt},\\
\nabla_{z{\partial\over \partial z}}&=z{\partial\over \partial z}-{1\over z} \sE(\bt)\star_{\bt}+\gr.
\end{dcases}
\end{equation}
Here we write the complex number $z\in \C^*$ as 
\begin{equation}
\label{multivalue-z}
z=|z|e^{\sqrt{-1}\arg(z)}\in\mathbb{C}^*. 
\end{equation}
For any $\bt$ in~\eqref{basis-element-form},
the operator $z^{-\gr}=e^{-(\log z)\gr}$ takes the form 
\begin{equation}
\label{zmu-formula}
z^{-\gr}(\bt)=\sum_{i=1}^{r}t_i |z|^{-\mu(\phi_i)} e^{-\sqrt{-1}\mu(\phi_i)\arg(z)} \phi_i.
\end{equation}
We denote the so-called {\em  $S$-operator} in FJRW theory by 
$$S(\bt,z)\in\operatorname{End}(\ca{H}_{W,G})\otimes \bb{C}[\![\bt]\!][\![z^{-1}]\!],$$ 
where\footnote{This is the $L$-operator in~\cite[Section 2, formula (20)]{CIR}.}
\begin{equation}
\label{S-operator}
S(\bt,z)\alpha=\alpha
  +\sum_{i=1}^{r}\sum_{n\geq1}\sum_{k\geq 0}\frac{\phi^i}{n!(-z)^{k+1}}
  \LD
  \alpha\psi_1^k,
  \bt,\dots,\bt,
  \phi_i
  \RD_{0,n+2}.
\end{equation}
By the {\em homogeneity property}~\eqref{virtual-degree} and genus zero {\em Topological Recursion Relation}~\cite[Theorem 4.2.9]{FJR}, the operator $S(\bt, z)$ provides solutions of the quantum connection.
\begin{proposition}
\cite[Proposition 2.7]{CIR}
\label{fundamental-solution-infty}
The set  $\{S(\bt,z)z^{-\gr}\alpha\mid \alpha\in\cH_{W,G}\}$ is a set of flat sections of $\nabla$.
That is,
\begin{equation}
\label{flat-section-nabla}
\begin{dcases}
\nabla_{\partial\over \partial t_i}(S(\bt,z)z^{-\gr}\alpha)=0,\\
\nabla_{\partial\over \partial z}(S(\bt,z)z^{-\gr}\alpha)=0.
\end{dcases}
\end{equation}
For $\alpha, \beta\in\ca{H}_{W,G}$, we have
\begin{equation}
\label{adjoint-S-operator}
    \<S(\bt,-z)\alpha, S(\bt,z)\beta\>=\<\alpha,\beta\>.
\end{equation}
\end{proposition}

\subsubsection{The FJRW $J$-function and Lagrangian cone}
\label{sec-small-j-function}
Using the genus zero topological recursion relation, the inverse operator of $S(\bt, z)$ is given by
\begin{equation}
\label{inverse-S-operator}
S(\bt,z)^{-1}\alpha=\alpha
  +\sum_{i=1}^{r}\sum_{n\geq1}\sum_{k\geq 0}\frac{\phi^i}{n!z^{k+1}}
  \langle
  \phi_i\psi_1^k,
  \bt,\dots,\bt,
  \alpha
  \rangle^{\mathrm{FJRW}}_{0,n+2}.
\end{equation}
Applying the string equation and~\cite[(68)]{FJR-book} to~\eqref{inverse-S-operator}, we have 
\begin{equation}
\label{S-operator-unit}
z\cdot S({\bf t}, z)^{-1}(\one)=z\one +{\bf t}+\sum_{i=1}^{r}\sum_{n\geq 2}{\phi^{i}\over n!}\LD{\phi_i\over z-\psi_1}, \bt, \ldots, \bt \RD_{0, n+1}.
\end{equation}
We define the {\em FJRW $J$-function} of an LG pair $(W, G)$ by
\begin{equation*}
\begin{aligned}
  J(\bt,-z)=-z\one+\bt
  +\sum_{i=1}^{r}\sum_{n\geq2}\sum_{k\geq0}\frac{\phi^i}{n!(-z)^{k+1}}
  \LD\bt,\ldots,\bt, \phi_i\psi_{n+1}^k\RD_{0,n+1}.
\end{aligned}
\end{equation*}

As a consequence of~\eqref{S-operator-unit}, we have
\begin{equation}
\label{J-equal-S}
J(\bt, -z)=-z S({\bf t}, -z)^{-1}(\one).
\end{equation}
Sometimes it is more convenient to consider {\em the modified $J$-function}
\begin{equation}
\label{modified-j-function}
\widetilde{J}(\bt, z):=z^{\widehat{c}_W\over 2} z^{\gr}  S({\bf t}, z)^{-1}(\one).
\end{equation}

Following Givental~\cite{Giv}, the {\em loop space} $\cH_{W, G}(\!(z^{-1})\!)$ consisting of Laurent series in ${1\over z}$ with vector coefficients carries a symplectic form induced by the pairing $\<\cdot, \cdot\>$. This infinite dimensional symplectic vector space can be identified with $T^*\cH_{W, G}[z]$, which is the cotangent bundle of $\cH_{W, G}[z]$.
Using the {\em dilaton shift} 
$${\bf q}(z)=-z\one+{\bf t}(z)$$
and the Darboux coordinates $({\bf p}, {\bf q})\in T^*\cH_{W, G}[z]$,
the graph of the differential of the genus zero generating function 
$$\mathcal{F}_0(\bt(z))=\sum_{n}{1\over n!}\LL \bt(z), \ldots, \bt(z)\RR_{0, n}(\bt(z))$$
defines a formal germ of Lagrangian submanifold $\mathcal{L}$ in $\cH_{W, G}(\!(z^{-1})\!)$, where
$$\mathcal{L}=\{({\bf p}, {\bf q})\in T^*\cH_{W, G}[z]: {\bf p}=d_{\bf q}\mathcal{F}_0\}.$$
Following~\cite[Theorem 1]{Giv}, the FJRW $J$-function lies on the Lagrangian cone. 
\begin{proposition}
The $J$-function $J(\bt, -z)$ belongs to the Lagrangian cone $\mathcal{L}.$
\end{proposition}

\subsubsection{Restriction to $\tau(t)$.}
We want to restrict the meromorphic connection to $\bt=\tau(t)$, 
we obtain the {\em small $J$-function}
$J(\tau(t), z)$
and the {\em modified small $J$-function}
$$\widetilde{J}(\tau(t),z):=\widetilde{J}(\bt, z)\vert_{\bt=\tau(t)}=z^{\widehat{c}_W\over 2} z^{\gr}  S(\tau(t), z)^{-1}(\one).$$

\begin{lemma}
If $\tau(t)=t\sJ_2,$
we have a formula for the modified small $J$-function
\begin{equation}
\label{modified-j-formula}
\widetilde{J}(\tau(t), z)
=
\one
  +
({t\cdot z^{-{\nu\over d}}})\sJ_2
  +\sum_{i=1}^{r}\sum_{n\geq 2}{({t\cdot z^{-{\nu\over d}}})^{n}\over n!}
\LD\phi_i\psi_1^{k}, \sJ_2, \ldots, \sJ_2\RD_{0, n+1} \ \phi^i.
\end{equation}
where the integer $k$ satisfies 
\begin{equation}
\label{value-k-jfunc}
k=-\mu(\phi_i)+{n\nu\over d}+{\widehat{c}_W\over 2}-2.
\end{equation}
\end{lemma}
\begin{proof}
Using the homogeneity property~\eqref{virtual-degree} and the relation~\eqref{degree-hodge}, we see if the FJRW invariant 
$\<\phi_i\psi_1^{k}, \sJ_2, \ldots, \sJ_2\>_{0, n+1}\neq 0$, 
we must have
\begin{equation}
\label{modify-j-degree}
\widehat{c}_W-3+n+1={\widehat{c}_W\over 2}+\mu(\phi_i)+k+n\left({\widehat{c}_W\over 2}+\mu(J^2)\right).
\end{equation}
By~\eqref{hodge-grading-operator}, we have
$$\mu(J^2)
=-{\widehat{c}_W\over 2}+1-{\nu\over d}.$$
Plugging it into~\eqref{modify-j-degree}, we see that the integer $k$ is determined by $\phi_i$ and $n$, as in~\eqref{value-k-jfunc}.
So for a fixed $n$, the coefficients of $\phi^i$ has a factor of $$z^{{\widehat{c}_W\over 2}+\mu(\phi^i)-1-(k+1)}=z^{-{n\nu\over d}}.$$
Here the equality follows from~\eqref{adjoint-hodge-operator} and~\eqref{value-k-jfunc}.
\end{proof}
\begin{remark}
In fact, we can simplify the formula further by using the {\em selection rule}~\eqref{selection-rule}. 
The nonvanishing condition $\<\phi_i\psi_1^{k}, \sJ_2, \ldots, \sJ_2\>_{0, n+1}\neq 0$ implies that 
$\phi_i\in \cH_{J^{-n-1}}.$
\end{remark}

\subsubsection{Adjoint operators}
For any operator $$A(z)\in {\rm End}(\cH_{W,G})\otimes_{\C}\C[z][\![z^{-1}]\!],$$ 
we denote by $A^*(z)$ its {\em adjoint operator} with respect to the pairing $\<\cdot, \cdot\>.$ That is 
\begin{equation}\label{adjoint-operator}
\<A(z)\alpha, \beta\>=\<\alpha, A^*(z)\beta\>.
\end{equation}

\begin{lemma}
\label{adjoint-hodge}
\begin{enumerate}
\item The quantum product ${\nu\over d}\tau'\star_{\tau}$ is self-adjoint, i.e.,
$$
\left
\langle
{\nu\over d}\tau'\star_{\tau}
\alpha,
\beta
\right\rangle
=
\left\langle
\alpha,
{\nu\over d}\tau'\star_{\tau}
\beta
\right\rangle.
$$
\item
The Hodge grading operator $\gr$ is skew-adjoint, i.e, 
$$
\left
\langle
\mathrm{Gr}
(\alpha),
\beta
\right\rangle
=
-
\left\langle
\alpha,
\mathrm{Gr}
(\beta)
\right\rangle.
$$
\end{enumerate}
\end{lemma}
\begin{proof}
(1) follows easily from the definition of quantum product and (2) follows from formula~\eqref{adjoint-hodge-operator}.
\end{proof}

\begin{lemma}
\label{lem:orthogonal-complement}
Let $\lambda\in \C^*$ be an eigenvalue of  ${\nu\over d}\tau'\star_{\tau}$ and 
let $E(\lambda)\subset \cH_{W, \<J\>}$ be its eigenspace.
Suppose that $E(\lambda)$ is one-dimensional.
Then $H'_{\lambda}:=\mathrm{Im}(\lambda-{\nu\over d}\tau'\star_{\tau})$ is the
orthogonal complement to $E(\lambda)$ in $\cH_{W, \<J\>}$ with respect to
the pairing $\< \cdot, \cdot\>$ and the Hodge grading operator $\mathrm{Gr}$ satisfies
$\mathrm{Gr}(E(\lambda))\subset H'_{\lambda}$.
\end{lemma}
\begin{proof}
The same arguments as in the proof of~\cite[Lemma 3.2.2]{GGI} work here. The proof relies on the self-adjointness of the quantum product and the skew-adjointness of the Hodge grading operator, which are stated in Lemma~\ref{adjoint-hodge}.
\end{proof}

\begin{remark}
\label{rem:nonzero-norm}
The above lemma shows that for any nonzero $\lambda$-eigenvector $\alpha$, we have $\<\alpha,\alpha \>\neq 0$.
The element $\alpha^{*}:=\alpha/\<\alpha,\alpha\>$ is dual to $\alpha$ in the sense that $\<\alpha^{*},\alpha\>=1$ and $\alpha$ is orthogonal to any element in the orthogonal complement of $\C\psi_{\lambda}$.
\end{remark}

\subsection{Strong asymptotic classes}
\label{sec-asymptotic-classes}
Now we consider strong asymptotic classes in FJRW theory of $(W, \<J\>)$ when $W$ is of general type.
The discussion here is parallel to~\cite{GGI}.
One minor difference is that we consider the asymptotic classes with respect to any eigenvalue with the maximal norm.
However, the proof easily follows from those for $\lambda\in \mathbb{R}_+$ (see Proposition~\ref{weak-to-strong}).


\subsubsection{Strong asymptotic classes}
According to Proposition~\ref{prop-euler-tau},
when we restrict the quantum connection in~\eqref{quantum-connection-formula} to $\bt=\tau$,
we obtain a meromorphic connection on $\C$ which takes the form
\begin{equation}
\label{meromorphic-connection-tau}
\nabla_{z{\partial\over \partial z}}=z{\partial\over \partial z}-{1\over z}\cdot {\nu\over d}\tau'\star_{\tau} +\gr.
\end{equation}
Proposition~\ref{fundamental-solution-infty} implies 
\begin{lemma}
For each $\alpha\in  \cH_{W, \<J\>},$  the (multi-valued) section
\begin{equation}
\label{flat-section-formula}
\Phi(\alpha):=(2\pi)^{-{{\widehat{c}_W}\over 2}}S(\tau, z)z^{-\gr}\alpha
\end{equation}
is a flat section of the meromorphic connection.
\end{lemma}
Using the notations in~\eqref{multivalue-z} and~\eqref{zmu-formula}, it is convenient to fix $\arg(z)=\theta\in \bb{R}$ 
and consider a single-valued flat section 
$$\Phi_\theta(\alpha):=\Phi(\alpha)\vert_{\arg(z)=\theta}
$$
defined over the ray $L_{\theta}=\mathbb{R}_{>0} \exp({\sqrt{-1}\theta})$.
Now we define an asymptotic class based on the asymptotic behavior of this flat section.
\begin{definition}
[Strong asymptotic classes]
\label{def-asymptotic-class}
Let $\lambda\in \C^*$ be an eigenvalue of ${\nu\over d}\tau'\star_{\tau}$
with the maximal norm. 
We say a nonzero element $\alpha\in \cH_{W, \<J\>}$ is a {\em strong asymptotic class} with respect to $\lambda$ if there exists some $m\in\mathbb{R}$ such that
\begin{equation}
\label{asymp-class-behavior}
|\!|
e^{\lambda \over z} \Phi_{\arg(\lambda)}(\alpha)
|\!|
= O(|z|^{-m})
\end{equation}
as $z\to 0$ along the ray $L_{\arg(\lambda)}=\mathbb{R}_{>0} \exp({\sqrt{-1}\arg(\lambda)})$.
Such a strong asymptotic class is called {\em principal} if $\lambda\in \mathbb{R}_+$.
\end{definition}


Let $\lambda\in\C^*$ be an eigenvalue of ${\nu\over d}\tau'\star_{\tau}\in {\rm End}(\cH_{W, \<J\>})$ with the maximal norm.
We denote the distinct eigenvalues
of
${\nu\over d}\tau'\star_{\tau}$
by 
$$\lambda_1=\lambda, \lambda_2, \ldots, \lambda_k,$$
where $\lambda_i$ has multiplicity $N_i$.
Note that $N_{1}+\dots+N_{k}=r$.
We assume that the multiplicity of $\lambda$ is one, i.e. $N_{1}=1$.
Define
$$U_\lambda=
\begin{bmatrix}
u_1&&&\\
&u_2&&\\
&&\ddots&\\
&&&u_r
\end{bmatrix}
=
\begin{bmatrix}
\lambda&&&\\
&\lambda_2 {\rm Id}_{N_2}&&\\
&&\ddots&\\
&&&\lambda_k  {\rm Id}_{N_k}
\end{bmatrix}
,
$$
where ${\rm Id}_{N_i}$ is the identity matrix of size $N_i$.
Let $\Psi:\C^{r}\rightarrow \cH_{W, \<J\>}$ be a linear isomorphism that transforms
${\nu\over d}\tau'\star_{\tau}$ to the block-diagonal form
\[
\Psi^{-1}
\left(
{\nu\over d}\tau'\star_{\tau}
\right)
\Psi
=
\begin{bmatrix}
A_1&&&\\
&A_2&&\\
&&\ddots&\\
&&&A_k
\end{bmatrix},
\]
where $A_{i}$ is a matrix of size $N_{i}$ and $A_{i}-\lambda_{i} I_{N_{i}}$ is nilpotent.

\begin{proposition}
\label{weak-to-strong}
With the notation as above,
there exists a fundamental matrix solution for the quantum connection~\eqref{meromorphic-connection-tau} of the form
\begin{equation*}
P(z)e^{-U_{\lambda}/z}
\begin{bmatrix}
1&&&\\
&F_2(z)&&\\
&&\ddots&\\
&&&F_k(z)
\end{bmatrix}
\end{equation*}
over a sufficiently small sector $$\mathcal{S}_\lambda=\{z\in \C^*\mid |z|\leq 1, |\arg(z)-\arg(\lambda)|\leq \epsilon\}$$
with $\epsilon>0$ such that
\begin{enumerate}
\item $P(z)$ has an asymptotic expansion $P(z)\sim \Psi+P_1 z+P_2 z^2+\dots$ as $|z|\to 0$ in $\mathcal{S}_{\lambda}$.
\item $F_i(z)$ is a $\mathrm{GL}_{N_i}(\C)$-valued function satisfying the estimate 
$$\max\left(
|\!|F_i(z)|\!|, |\!|F_i(z)^{-1}|\!|
\right)\leq C e^{\delta |z|^{-p}}$$
on $\mathcal{S}_\lambda$ for some $C, \delta>0,$ and $0<p<1$.
\end{enumerate}
\end{proposition}

\begin{proof}
In the case when $\lambda$ is a positive real number, the same arguments as in the proof of~\cite[Proposition 3.2.1]{GGI} work here. 
The proof relies on Sibuya and Wasow's basic results on ODEs and Lemma~\ref{lem:orthogonal-complement} on the $\lambda$-eigenspace of ${\nu\over d}\tau'\star_{\tau}$.

For the general case, let $\theta=\arg(z)$.
We consider the change of coordinate $z=e^{2\pi \sqrt{-1}\theta}\tilde{z}$, and the
quantum connection~\eqref{meromorphic-connection-tau} becomes
\begin{equation}
\label{eq:new-quantum-connection}
\tilde{z}{\partial\over \partial \tilde{z}}-{e^{-2\pi\sqrt{-1}\theta}\over \tilde{z}}\cdot {\nu\over d}\tau'\star_{\tau} +\gr.
\end{equation}
Note that the eigenvalues of $e^{-2\pi\sqrt{-1}\theta} ({\nu\over d}\tau'\star_{\tau})$ are $e^{-2\pi\sqrt{-1}\theta}\lambda_{i}$, and $|\lambda|$ is the one with the maximal norm.
Hence according to the previous special case of the proposition, there exists a fundamental matrix solution for~\eqref{eq:new-quantum-connection} of the form
\begin{equation*}
\tilde{P}(\tilde{z})e^{-U/\tilde{z}}
\begin{bmatrix}
1&&&\\
&\tilde{F}_2(\tilde{z})&&\\
&&\ddots&\\
&&&\tilde{F}_k(\tilde{z})
\end{bmatrix}
\end{equation*}
over
$\mathcal{S}=\{\tilde{z}\in \C^*\mid |\tilde{z}|\leq 1, |\arg(\tilde{z})|\leq \epsilon\}$, where $U=e^{-2\pi\sqrt{-1}\theta}U_{\lambda}$, and $\tilde{P}(\tilde{z})$ and $\tilde{F}_{i}(\tilde{z})$ satisfy asymptotic conditions similar to (1) and (2) over $\ca{S}$.
This implies the proposition in the original $z$-coordinate.  
\end{proof}

The following proposition is the counterpart of~\cite[Proposition 3.3.1]{GGI} in the LG setting and is an easy consequence of Proposition~\ref{weak-to-strong}.

\begin{proposition}
\label{prop:existence-of-limit}
Let $\lambda\in\C^*$ be an eigenvalue of ${\nu\over d}\tau'\star_{\tau}\in {\rm End}(\cH_{W, \<J\>})$ with the maximal norm. 
If the multiplicity of $\lambda$ is one, then 
\begin{enumerate}
\item the space 
of strong asymptotic classes with respect to $\lambda$ is 1-dimensional.
\item For a strong symptotic classes $\alpha$ with respect to $\lambda$,
$\lim\limits_{|z|\to+0}e^{\lambda \over z} \Phi_{\arg(\lambda)}(\alpha)$ exists and
lies in the $\lambda$-eigenspace $E(\lambda)$ of ${\nu\over d}\tau'\star_{\tau}$.
\end{enumerate}
\end{proposition}


\subsubsection{A formula for strong asymptotic classes}

We introduce a {\em dual quantum connection} $\nabla^{\vee}$ for the meromorphic connection~\eqref{meromorphic-connection-tau}, given by
\begin{equation}
\label{dual-connection}
\nabla^\vee_{z{\partial\over \partial z}}=z{\partial\over \partial z}+{1\over z} \left({\nu\over d}\tau'\star_{\tau} \right)^*- \gr^*.
\end{equation}

We identity $\cH_{W,G}$ with its dual $(\cH_{W,G})^{\vee}$ via the nondegenerate pairing $\<\cdot, \cdot\>$.
Under this identification, Lemma~\ref{adjoint-hodge} implies that
\begin{equation}
\label{eq:sym-anti-sym}
 \left({\nu\over d}\tau'\star_{\tau} \right)^*={\nu\over d}\tau'\star_{\tau},\quad
  \gr^*=- \gr,
\end{equation}
and formula~\ref{adjoint-S-operator} implies 
\begin{equation}
\label{eq:S-operator-adjoint}
S^*(\bt, z)=S^{-1}(\bt, -z).
\end{equation}



For each $\alpha\in  \cH_{W, \<J\>},$ we consider the inverse adjoint of the flat section~\eqref{flat-section-formula} given by 
\begin{equation}
\label{dual-section}
\Phi^\vee(\alpha)
:=(2\pi)^{{\widehat{c}_W\over 2}}(S^{-1})^{*}(\tau,z) z^{\gr^{*}}\alpha
=(2\pi)^{{\widehat{c}_W\over 2}}S(\tau,-z) z^{-\gr}\alpha.
\end{equation}
Here the second equality follows from~\eqref{eq:sym-anti-sym}.

\begin{proposition}

The following properties hold:
\begin{enumerate}
\item
The section $\Phi^\vee(\alpha)$ is a flat section of the connection $\nabla^{\vee}$.
\item
We have $\Phi^\vee=(\Phi^{-1})^{*}$. Equivalently, we have
$$\<\Phi^\vee(\alpha), \Phi(\beta)\>=\<\alpha, \beta\>, \quad \forall \quad \alpha, \beta \in \cH_{W, \<J\>}.$$
\end{enumerate}
\end{proposition}
\begin{proof}
Part (2) follows from the definition~\eqref{dual-section} of $\Phi^\vee$.
Using the identities~\eqref{eq:sym-anti-sym}, we can identify the dual flat connection $\nabla^{\vee}$ with the quantum connection $\nabla$ with the opposite sign of $z$.
Since $\Phi$ and $\Phi^{\vee}$ agree up to a constant after changing the sign of $z$ (see~\eqref{flat-section-formula},~\eqref{dual-section}), we obtain the proof of part (1).
\end{proof}

Similar to~\cite[Definition 3.6.3]{GGI}, we have the following nice relation between the dual flat sections and the modified small $J$-function.
\begin{lemma}
\label{dual-flat-section-via-j-function}
For any $\alpha\in \cH_{W, \<J\>}$, we have
$$\left({z\over 2\pi}\right)^{{\widehat{c}_W\over 2}}\<\Phi^\vee(\alpha), \one\>=\<\alpha, \widetilde{J}(\tau, z)\>.$$
\end{lemma}
\begin{proof}
We have
\begin{eqnarray*}
\left({z\over 2\pi}\right)^{{\widehat{c}_W\over 2}}\<\Phi^\vee(\alpha), \one\>
&=&z^{{\widehat{c}_W\over 2}}\<S(\tau, -z)z^{-\gr}\alpha, \one\>\\
&=&z^{{\widehat{c}_W\over 2}}\<\alpha, z^{-\gr^*}S^*(\tau, -z)\one\>\\
&=&\<\alpha, z^{{\widehat{c}_W\over 2}}z^{\gr}S(\tau, z)^{-1}\one\>\\
&=&\<\alpha, \widetilde{J}(\tau, z)\>.
\end{eqnarray*}
Here the first equality is the definition of the section $\Phi^\vee(\alpha)$.
The third equality uses Lemma~\ref{adjoint-hodge} 
and~\eqref{eq:S-operator-adjoint}. 
The last equality follows from the definition~\eqref{modified-j-function} of the modified $J$-function.
\end{proof}

Similar to~\cite[Theorem 3.6.8]{GGI}, the strong asymptotic classes are related to the modified $J$-function under some mild assumption.


\begin{proposition}
\label{prop-strong-class}
Let $\lambda\in\C^*$ be a  multiplicity-one eigenvalue of ${\nu\over d}\tau'\star_{\tau}\in {\rm End}(\cH_{W, \<J\>})$ with the maximal norm. 
Let $A_\lambda$ be a strong asymptotic class with respect to $\lambda$. If $\<\one, A_\lambda\>\neq 0$, then 
\begin{equation}
\label{limit-j-function}
\lim_{\substack{\arg(z)=\arg(\lambda)\\|z|\to+0}}
{\widetilde{J}(\tau, z) \over \<\one,  \widetilde{J}(\tau, z) \>}
={A_\lambda \over \<\one, A_\lambda\>}.
\end{equation}
\end{proposition}
\begin{proof}
By Proposition~\ref{prop:existence-of-limit}, we have $$\psi_\lambda:= \lim_{|z|\to+0}e^{\lambda \over z} \Phi_{\arg(\lambda)}(A_\lambda)\in E(\lambda).$$ 
By Remark~\ref{rem:nonzero-norm}, the element $\psi_\lambda^*:=\psi_\lambda/\< \psi_\lambda,\psi_\lambda\>\in \cH_{W, \<J\>}$ is dual to $\psi_\lambda$.
Follow the argument in the proof of~\cite[Proposition 3.6.2]{GGI}, for each $\alpha\in \cH_{W, \<J\>}$, the following limit exists and 
\begin{equation}
\label{leading-term-negative}
\lim_{\substack{\arg(z)=\arg(\lambda)\\|z|\to+0}}
e^{-{\lambda\over z}}\Phi^\vee(\alpha)(z)=\<\alpha, A_\lambda\>\psi_\lambda^*.
\end{equation}
Thus we have
\begin{eqnarray*}
\lim_{\substack{\arg(z)=\arg(\lambda)\\|z|\to+0}}
{\< \alpha, \widetilde{J}(\tau, z) \>\over \<\one,  \widetilde{J}(\tau, z) \>}
&=&\lim_{\substack{\arg(z)=\arg(\lambda)\\|z|\to+0}}
{e^{-{\lambda\over z}}\<\Phi^\vee(\alpha), \one\>\over e^{-{\lambda\over z}}\<\Phi^\vee(\one), \one\>}
\\&=&{ \<\alpha, A_\lambda\>\<\psi_\lambda^*, \one\>\over \<\one, A_\lambda\>\<\psi_\lambda^*, \one\>}
\\&=&{\<\alpha, A_\lambda\> \over \<\one, A_\lambda\>}.
\end{eqnarray*}
Here the first equality uses Lemma~\ref{dual-flat-section-via-j-function}.
The second equality uses the formula~\eqref{leading-term-negative}.
The last equality uses Lemma~\ref{dual-eigenvector-nonzero} below.
\end{proof}
Follow the argument in the proof of~\cite[Lemma 3.6.10]{GGI}, we have
\begin{lemma}
\label{dual-eigenvector-nonzero}
The element $\psi_\lambda^*$ satisfies $\<\psi_\lambda^*, \one\>\neq 0$.
\end{lemma}

\section{Weak asymptotic classes via mirror symmetry}
\label{sec-weak}

The asymptotic condition~\eqref{asymp-class-behavior} in the definition of strong asymptotic classes is a strong constraint. 
Now we consider elements that satisfies a different asymptotic condition concerning only the coefficient of $\sJ_{d-1}$ in
$e^{\lambda \over z} \Phi_{\arg(\lambda)}(\alpha)$.
\begin{definition}
[Weak asymptotic classes]
\label{def-weak-asymptotic-class}
We say 
 $\alpha\in H_{\rm nar}$ is a {\em weak asymptotic class} with respect to $\lambda\in \C^*$ if there exists some $m\in\mathbb{R}$, 
such that when $$\arg(z)=\arg(\lambda)\in [0, 2\pi),$$ 
we have 
\begin{equation}
\label{leading-asymp-class-behavior}
\left|e^{\lambda\over z}\cdot \<S(\tau, z)z^{-\gr}\alpha, \one\>\right|=O(|z|^{m})\quad\mathrm{as}\ |z|\to 0.
\end{equation}
We say a weak asymptotic class is principal if $\lambda\in \R_+$.
\end{definition}

\begin{remark}
At a first glance, the two asymptotic classes (strong and weak) may have different properties. 
For example, neither we require the complex number $\lambda$ to be an eigenvalue of the quantum multiplication in the definition of weak asymptotic classes, nor we require the strong asymptotic class to be a narrow element.
However, we expect they are the same (up to a scalar). 
See Proposition~\ref{weak=strong}.
\end{remark}



In this section, we will calculate the weak asymptotic classes using mirror symmetry (without the assumption on the quantum spectrum conjecture). 
Here is the plan.
In Section~\ref{sec-i-function}, we study the differential equations for the small $I$-function and state a mirror conjecture.
In Section~\ref{sec-basic-ghe}, we recall some basics of generalized hypergeometric functions. 
After a change of coordinates, we prove the coefficients of the modified small $I$-function satisfy a generalized hypergeometric equation~\eqref{ghe-equation}.
In Section~\ref{sec-barnes-formula}, we review an asymptotic expansion formula of generalized hypergeometric functions, found by Barnes~\cite{Bar} in the year of 1906. 
When the mirror conjecture holds, Barnes' formula allows us to calculate the weak asymptotic classes. 
In Section~\ref{sec-dominance-order}, we explain the meaning of the asymptotic classes in terms of the dominance order of asymptotic expansions.

\subsection{The small $I$-function and mirror symmetry}
\label{sec-i-function}

We now study some properties of the small $I$-function defined in~\eqref{small-i-function}
$$I_{\rm FJRW}^{\rm sm}(t,z)=z\sum_{m\in {\bf Nar}}\sum_{\ell=0}^{\infty}\frac{t^{d\ell+m}}{z^{d\ell+m-1}\Gamma(d\ell+m)}\prod_{j=1}^{N}{z^{\lfloor {w_j\over d}(d\ell+m)\rfloor}\Gamma({w_j\over d}(d\ell+m))\over \Gamma\left(\left\{{w_j\over d}\cdot m\right\}\right)}\sJ_m.$$

\subsubsection{Differential equations for the small $I$-function}
We denote the cardinality of the set 
$${\bf Nar}=\left\{m\in \mathbb{Z}_{>0} \mid 0< m<d, \text{ and } d\nmid w_j\cdot m, \  \forall \ 1\leq j\leq N\right\}
$$ in~\eqref{narrow-index} by 
\begin{equation}
\label{definition-q}
q+1=|{\bf Nar}|.
\end{equation}
It is convenient to reindex the set by 
the increasing bijection
\begin{equation}
\label{narrow-bijection}
\mathscr{N}:  {\bf Nar}\to \{0, 1, \ldots, q\}.
\end{equation}
For each integer $j$ such that $0\leq j\leq q$, we write $m=\mathscr{N}^{-1}(j)\in {\bf Nar}$. 
Now we define a set of $q$ positive rational numbers 
\begin{equation}
\label{rho-tuple}
\rho_Q=\{\rho_1, \ldots, \rho_q\},
\end{equation}
where each $\rho_j$ is given by
\begin{equation}
\label{reindex-narrow}
\rho_j=\rho_{\mathscr{N}(m)}={1\over d}+1-{m\over d}, \quad m\in{\bf Nar}.
\end{equation}
This formula includes the case $\rho_0=1$.
The set $\{\rho_0, \rho_1, \ldots, \rho_q\}$ is arranged in a decreasing order.

Now we consider a collection of $\sum\limits_{j=1}^{N}w_j$ positive rational numbers
\begin{equation}
\label{alpha-tilde}
\widetilde{\alpha}=\left\{{1\over d}+{k\over w_j} \mid \quad j=1, 2, \ldots, N, \quad k=0, 1, \ldots, w_j-1 \right\}.
\end{equation}
Numbers may appear repeatedly in $\widetilde{\alpha}$.
For each $n\in \{0, 1, \ldots, d-1\}\backslash{\bf Nar},$ by the definition of ${\bf Nar}$,
there exists some pair $(k, w_j)$ such that $k d=nw_j$.
For each such $n$, we delete the number 
\begin{equation}
\label{common-numbers}
{1\over d}+{k\over w_j}={n+1\over d}
\end{equation}
once from the set $\widetilde{\alpha}$. We eventually obtain a reduced set of $p$ numbers, denoted by
\begin{equation}
\label{alpha-tuple}
\alpha_P=\{\alpha_1, \ldots, \alpha_p\}.
\end{equation}
According to the construction, the integer $p$ satisfies 
\begin{equation}
\label{definition-p}
p=\sum_{j=1}^{N}w_j-(d-|{\bf Nar}|)=|{\bf Nar}|-\nu.
\end{equation}
Using~\eqref{definition-q} and~\eqref{definition-p}, we have
\begin{equation}
\label{p-q-nu-relation}
q+1-p
=\nu.
\end{equation}

We provide a list of examples in Table \ref{table-indicies} below.


\begin{table}[h]
\caption{Examples of indicies}
\label{table-indicies}
 \resizebox{\textwidth}{!}{
\begin{tabular}{|c|l| l|c|c|c|c|c|c|c|l|l|}
 \hline
  & $W$ &   ${\bf Nar}$ &  $p$ &$q$ & $\{\alpha_i\}_{i=1}^{p}$ & $\{\rho_1, \ldots, \rho_q\}$\\
 \hline
 $A_{d-1}$&$x^{d}$ & $1, 2,\ldots, d-1$ & $0$ & $d-2$ & $\emptyset$ & $\{{d-1\over d}, \ldots, {2\over d}\}$\\
 \hline
$D_{2k+1}$&$x^{2k}+xy^2$ & $\{1, \ldots, 4k-1\}\backslash\{2k\}$ & $2k-1$  & $4k-3$ &  $\{{1\over 4k}+{i-1\over 2k-1}\}$ & $\{{4k-1\over 4k},\ldots,{2\over 4k}\}\backslash\{{2k+1\over 4k}\}$\\
 \hline
$D_{2k+2}$&$x^{2k+1}+xy^2$ & $1,2,\ldots, 2k$& $k$  & $2k-1$ &$\{{1\over 2k+1}+{i-1\over k}\}$ & $\{{2k\over 2k+1}, \ldots, {2\over 2k+1}\}$ \\
\hline
$D^T_{n+1}$&$x^{n}y+y^2$ & $1,3,\ldots, 2n-1$& $1$  & $n-1$ &$\{{1\over 2n}\}$ & $\{{1\over n}, {2\over n}, \ldots, {n-1\over n}\}$ \\
\hline
$E_6$&$x^4+y^3$ & $1,2,5,7,10,11$ &   $1$  & $5$ & $\{{1\over 12}\}$ & $\{ {11\over 12}, {10\over 12}, {7\over 12}, {5\over 12}, {2\over 12}\}$ \\
\hline
$E_7$&$x^3y+y^3$ & $1,2,4,5,7,8$ &   $2$  & $5$ & $\{{1\over 9}, {1\over 9}+{1\over 2}\}$  & $ \{{8\over 9}, {6\over 9}, {5\over 9}, {3\over 9}, {2\over 9}\}$\\
\hline
$E_8$&$x^5+y^3$ & $1,2,4,7,8,11,13,14$ & $1$  & $7$ &  $\{{1\over 15}\}$ & $ \{{14\over 15}, {12\over 15}, {9\over 15}, {8\over 15}, {5\over 15}, {3\over 15}, {2\over 15}\}$\\
\hline
$F_N^d$&$\sum\limits_{j=1}^{N}x_j^d$ & $1,2,\ldots, d-1$ & $N-1$  & $d-2$ & $\{{1\over d}, \ldots, {1\over d}\}$ & $\{ {d-1\over d}, \ldots, {2\over d}\}$ \\
\hline
\end{tabular}
}
\end{table}




Using the notations above, we have
 \begin{proposition}
 The function $I_{\rm FJRW}^{\rm sm}(t,z)$ in~\eqref{small-i-function} satisfies the differential equation
\begin{equation}
\label{i-function-ode}
\left(t^{d}z^p\prod_{j=1}^N{w_j^{w_j}\over d^{w_j}}\prod_{i=1}^{p}\left(t{\partial\over\partial t}+\alpha_i d-1\right)
-z^{q+1}\prod_{j=0}^{q}\left(t{\partial \over \partial t}+(\rho_j-1)d-1\right)\right) I_{\rm FJRW}^{\rm sm}(t,z)=0.
\end{equation}
\end{proposition}
\begin{proof} 
We first prove that $I_{\rm FJRW}^{\rm sm}(t,z)$ satisfies the differential equation
\begin{equation}\label{eq:dfq-quasihomogeneous-Fermat}
  \left(t^{d}\prod_{j=1}^{N}\prod_{c=0}^{w_j-1}z\left({w_j\over d}t{\partial\over\partial t}+c\right)
  \Big/\left(zt{\partial\over\partial t}\right)
-\prod_{c=1}^{d-1}z\left(t{\partial \over \partial t}-c\right)\right) I_{\rm FJRW}^{\rm sm}(t,z)=0.
\end{equation}
The proof is straight forward.
We have
$$\left({w_j\over d}t{\partial\over\partial t}+c\right){t^{d\ell+m}}=\left({w_j\over d}(d\ell+m)+c\right){t^{d\ell+m}}.$$
Since $m\in {\bf Nar}, \ell\geq 0, c\geq 0$, the coefficient ${w_j\over d}(d\ell+m)+c$ is always positive. 
Thus when we apply the first differential operator in~\eqref{eq:dfq-quasihomogeneous-Fermat} to the small $I$-function, the coefficient of the term 
\begin{equation}
\label{i-function-monomial}
t^{d\ell+d+m}z^{1+\left(\sum_{j=1}^{N}w_j\right)-1+\sum_{j=1}^{N}\lfloor {w_j\over d}(d\ell+m)\rfloor-(d\ell+m-1)}
\end{equation}
is given by 
\begin{equation}
\label{coefficient-monomial-1st}
{1\over \Gamma(d\ell+m)}\cdot\prod_{j=1}^{N}{\Gamma({w_j\over d}(d\ell+m))\over \Gamma\left(\left\{{w_j\over d}\cdot m\right\}\right)} 
\cdot {\prod\limits_{j=1}^{N}\prod\limits_{c=1}^{w_j-1}\left({w_j\over d}(d\ell+m)+c\right)\over d\ell+m}.
\end{equation}

Now let us apply the second differential operator in~\eqref{eq:dfq-quasihomogeneous-Fermat} to the small $I$-function. 
When we take the same $m\in {\bf Nar}$ but $\ell+1$, the  exponent of $t$ is $d(\ell+1)+m$ while the exponent of $z$ is
$$1+\sum_{j=1}^{N}\lfloor {w_j\over d}\left(d(\ell+1)+m\right)\rfloor-(d(\ell+1)+m-1)+d-1.$$
We obtain the same monomial in~\eqref{i-function-monomial} and the coefficient of the monomial is 
$${1\over \Gamma(d\ell+d+m)}\cdot\prod_{j=1}^{N}{\Gamma({w_j\over d}(d\ell+d+m))\over \Gamma\left(\left\{{w_j\over d}\cdot m\right\}\right)}
\cdot \prod_{c=1}^{d-1}(d\ell+d+m-c).$$
By the property $\Gamma(x+1)=x\Gamma(x)$ for $x>0$, it is the same as the coefficient in~\eqref{coefficient-monomial-1st}. 

Finally, 
 if we take $m\in {\bf Nar}$ and $\ell=0$, then $t^m$ is annihilated by the second differential operator 
$\prod\limits_{c=1}^{d-1}\left(t{\partial \over \partial t}-c\right)$.
This completes the proof of the equation~\eqref{eq:dfq-quasihomogeneous-Fermat}.

Now we reduce the order to prove the equation~\eqref{i-function-ode}.
From the proof of the equation~\eqref{eq:dfq-quasihomogeneous-Fermat}, we can see that for an integer $k$ such that $0<k<w_j$, if ${k\cdot d\over w_j}$ is also an integer, then we can reduce the order of the differential equation by one by deleting the factor $z\left(t{\partial \over \partial t}+{k\cdot d\over w_j}\right)$ from the first different operator and a factor $z\left(t{\partial \over \partial t}-(d-{k\cdot d\over w_j})\right)$ from the second different operator.
In particular, we see that $d-{k\cdot d\over w_j}\notin {\bf Nar},$ so this does not affect the annihilation of $t^m$ with $m\in {\bf Nar}$.
We can repeat the process until the order of the differential equation equation reaches to $|{\bf Nar}|=q+1$, which is the equation~\eqref{i-function-ode}.
\end{proof}

\begin{remark}
We make a few remarks on the small $I$-function and the differential equation.
\begin{enumerate}
\item 
The small $I$-function formula in~\eqref{small-i-function} takes the same form as formula (15) in~\cite{Aco} and the formula (102) in~\cite{Gue} (up to some sign). 
In~\cite{Aco}, the polynomial $W$ needs to satisfy the Gorenstein condition $w_j\vert d$ for any $j$. 
The element $\sJ_m$ here is the same as the element $\phi_{m-1}$ in~\cite{Aco}.
In~\cite{Gue}, the polynomial $W$ is assumed to be of chain type.

\item The differential equation~\eqref{eq:dfq-quasihomogeneous-Fermat} becomes a {\em Picard-Fuchs} equation if there is a mirror B-model, and the period integrals in the $B$-model satisfy the equation. 
\end{enumerate}
\end{remark}

\subsubsection{A genus-zero mirror conjecture}

Now the {\em mirror conjecture~\ref{conjecture-i-function-formula}} proposed in Section~\ref{sec-intro-wall} is a special case of the following {\em Givental type genus-zero mirror symmetry conjecture for small $I$-function of $(W, \<J\>)$}.
\begin{conj}
\label{mirror-conjecture}
The small $I$-function $I_{\rm FJRW}^{\rm sm}(t, -z)$ lies on the Lagrangian cone. 
In particular, let $t\tau(t)$ be the coefficients of $z^{0}$ in the expansion of the small $I$-function, then 
the $J$-function $J(\tau(t),z)$ is determined by an equation
$$ tJ(\tau(t), -z)=I_{\rm FJRW}^{\rm sm}(t, -z)+c(t, z)z{\partial I_{\rm FJRW}^{\rm sm}(t, -z)\over \partial t}.$$
where $c(t, z)$ is a formal power series in $t$ and $z$.
\end{conj}
Let $(W, \<J\>)$ be an admissible LG pair of general type. Using $d-\sum\limits_{j=1}^Nw_j>0$ and the degree calculation, we see that the only term in $I_{\rm FJRW}^{\rm sm}(t,z)$ with a positive power of $z$ is the leading term $zt\one$.  
\begin{corollary} 
If Conjecture~\ref{mirror-conjecture} holds, then we have
$$
I_{\rm FJRW}^{\rm sm}(t,z)=tJ(\tau(t),z).
$$
\end{corollary}

Similar to Lemma~\ref{dual-flat-section-via-j-function},
we can express $\<\Phi(\alpha), \one\>$ in terms of the small $J$-function. 
\begin{proposition}
\label{weak-asymptotic-i-function}
For any element $\alpha\in \cH_{W, G}$, we have
$$(-2\pi z)^{{\widehat{c}_W\over 2}}\<\Phi(\alpha), \one\>=\left\<e^{\pi\sqrt{-1}\gr} \alpha, (-z)^{{\widehat{c}_W\over 2}-1}(-z)^{\gr} J(\tau,  -z)\right\>.$$
\end{proposition}
\begin{proof}
We have
\begin{eqnarray*}
(2\pi)^{{\widehat{c}_W\over 2}}\<\Phi(\alpha), \one\>
&=&\left\<S(\tau,z)z^{-\gr}\alpha, \one\right\>\\
&=&\left\<z^{-\gr}\alpha, S(\tau,-z)^{-1}\one\right\>\\
&=&\left\<(-z)^{-\gr} e^{\pi\sqrt{-1}\gr}\alpha, S(\tau,-z)^{-1}\one\right\>\\
&=&\left\<e^{\pi\sqrt{-1}\gr}\alpha, (-z)^{\gr}S(\tau, -z)^{-1}\one\right\>\\
&=&\left\<e^{\pi\sqrt{-1}\gr}\alpha, (-z)^{\gr}(-z)^{-1} J(\tau,-z)\right\>.
\end{eqnarray*}
The second equality uses the property~\eqref{adjoint-S-operator} of $S(\tau,z)$.
The fourth equality uses skew-adjointness
of the Hodge grading operator. 
The last equality follows from the definition of $J$-function~\eqref{J-equal-S}.
\end{proof}

By Proposition~\ref{weak-asymptotic-i-function}, we have
\begin{corollary}
\label{weak-via-i-function}
If the mirror conjecture~\ref{conjecture-i-function-formula} holds, then
$$(-2\pi z)^{{\widehat{c}_W\over 2}}\<\Phi(\alpha), \one\>=
\left\<e^{\pi\sqrt{-1}\gr}\alpha, (-z)^{{\widehat{c}_W\over 2}-1}(-z)^{\gr} I^{\rm sm}_{\rm FJRW}(1, -z)\right\>.$$
\end{corollary}

\subsubsection{Mirror symmetry for invertible polynomials}
A polynomial is called {\em invertible} if it can be linearly rescaled to the form 
\begin{equation}
\label{invertible}
W=\sum_{i=1}^{N}\prod_{j=1}^{N}x_j^{a_{ij}},
\end{equation}
such that the {\em exponent matrix} $E_W:=\left(a_{ij}\right)_{N\times N}$ is an invertible matrix.
According to ~\cite{KrS}, each invertible polynomial can be written as a disjoint sums of three {\em atomic types} of polynomials (up to permutation of variables): 
\begin{enumerate}
\item {\em Fermat type}: $x_1^{a_1}+\dots+x_m^{a_m}$;
\item {\em Chain type}: $x_1^{a_1}x_2+\dots+x_{m-1}^{a_{m-1}}x_m+x_m^{a_m}$;
\item {\em Loop type}: $x_1^{a_1}x_2+\dots+x_{m-1}^{a_{m-1}}x_m+x_m^{a_m}x_1$.
\end{enumerate}

In the literature, the computations in the FJRW theory of $(W, G)$ have been considered mostly when the polynomial $W$ is invertible, see~\cite{FJR, Kra, CIR, HLSW}. 
In particular, when the group $G=G_W$ is the maximal group, the mirror symmetry for invertible polynomials can be upgraded to any genus. In~\cite{HLSW}, it has been proved that if $W$ has no $x^2$ monomial, the FJRW invariants of $(W, G_W)$ at any genus are equal to B-model invariants from the {\em Saito-Givental theory} of the {\em mirror polynomial} $$W^T=\sum\limits_{i=1}^{N}\prod\limits_{j=1}^{N}x_j^{a_{ji}}.$$ 
Furthermore, the underlying Dubrovin-Frobenius manifolds are semisimple. 


In general, when $G=\<J\>\neq G_W$, Dubrovin-Frobenius manifold structure has not been constructed on the mirror side. 
Mirror conjecture~\ref{mirror-conjecture} has been only proved for the following invertible polynomials in~\cite[Theorem 1.2]{Aco} and~\cite[Theorem 4.2]{Gue}.
\begin{proposition}
\label{mirror-theorem-fermat}
The mirror conjecture~\ref{mirror-conjecture} holds if
\begin{enumerate}
\item $W$ is a Fermat polynomial;
\item $W$ is a chain polynomial, $\<J\>=G_W$ and $W$ has no $x^2$ monomial when $N\geq 3$. 
\end{enumerate}
\end{proposition}


\subsection{Generalized hypergeometric functions in FJRW theory}
\label{sec-basic-ghe}

We will recall some basics on generalized hypergeometric functions
\begin{align*}
\, _pF_q\left(
\begin{array}{l}
a_1, \ldots, a_{p}\\
b_1, \ldots, b_{q}
\end{array}; x\right)
:=\sum_{k\geq0} {\prod\limits_{i=1}^{p}(a_i)_k\over \prod\limits_{i=1}^{q}(b_i)_k}\cdot {x^k\over k !},
\end{align*}
where $(c)_k$ is the Pochhammer symbol 
$$(c)_k:={\Gamma(c+k)\over \Gamma(c)}=c(c+1)\cdots(c+k-1).$$
Details can be found in~\cite{Luk, Fie, dlmf}.

\subsubsection{Generalized hypergeometric equations and its solutions near $x=0$}

Let $\{\rho_j\}$ and $\{\alpha_i\}$ be the rational numbers introduced in~\eqref{rho-tuple} and~\eqref{alpha-tuple}.
We consider the generalized hypergeometric equation
\begin{equation}
\label{ghe-equation}
\Big(\vartheta_x\prod_{j=1}^{q}(\vartheta_x+\rho_j-1)-x\prod_{i=1}^{p}(\vartheta_x+\alpha_i)\Big) f(x) =0, \quad\text{where}\quad \vartheta_x:=x{\partial\over \partial x}.
\end{equation}

We will focus on the cases  when $p\leq q$, since the polynomial $W$ under consideration is of general type. When $p\leq q$, the equation~\eqref{ghe-equation} has a {\em regular singularity} at $x=0$ and an {\em irregular singularity} at $x=\infty$. The order of irregularity is one.
We remark that when $W$ is of Calabi-Yau type, namely  $p=q+1$, then the equation has regular singularities at $x=0, 1$, and $\infty$. See~\cite[Section 16.8]{dlmf}.

Now we discuss the formal solutions of~\eqref{ghe-equation} near $x=0$.
We notice by~\eqref{reindex-narrow} that 
\begin{enumerate}
\item no two $\rho_j$'s  differ by an integer, $\rho_0=1$, and $\rho_j\notin\mathbb{Z}$ $(\forall \ 1\leq j\leq q)$;
\item 
for any pair $(i, j)$, $\rho_j\neq\alpha_i$.
 \end{enumerate}
According to~\cite[Section 16.8]{dlmf}, near $x=0$, there is a fundamental set of formal solutions of the equation~\eqref{ghe-equation}, denoted by $\{f_{j}(x) \vert j=0, 1, \ldots, q\}$, where 
\begin{equation}
\label{ghe-standard}
f_{j}(x)=
x^{1-\rho_j}\, _pF_q\left(
\begin{array}{l}
1+\alpha_1-\rho_j, \ldots, 1+\alpha_p-\rho_j\\
1+\rho_0-\rho_j, \ldots, (1+\rho_j-\rho_j)^*, \ldots, 1+\rho_q-\rho_j
\end{array}; x\right).
\end{equation}
Here $(1+\rho_j-\rho_j)^*$ means the term $1+\rho_j-\rho_j$ is deleted. 

\begin{lemma}
The following identity hold:
\begin{equation}
\label{difference-alpha-rho}
\sum_{h=1}^{p}\alpha_h-\sum_{h=1}^{q}\rho_h
=-\nu\left({1\over 2}+{1\over d}\right)+{3-N\over 2}. 
\end{equation}
\end{lemma}
\begin{proof}
By adding all the numbers except $\rho_0=1$ that are deleted in~\eqref{common-numbers} to both sums on the left-hand side of~\eqref{difference-alpha-rho}, we obtain
\begin{equation*}
\begin{split}
\sum_{h=1}^{p}\alpha_h-\sum_{h=1}^{q}\rho_h
&=\sum_{j=1}^{N}\sum_{k=0}^{w_j-1}\left({1\over d}+{k\over w_j}\right)-\sum_{m=0}^{d-2}{m+1\over d}\\
&=\sum_{j=1}^{N}\left({w_j\over d}+{w_j-1\over 2}\right)-{d-1\over 2}.
\end{split}
\end{equation*}
Now the result follows from the formula~\eqref{gorenstein-parameter} of $\nu$.
\end{proof}

For each $m\in {\bf Nar}$, we introduce a $p$-tuple $\alpha_P^{(m)}$ and a $q$-tuple $\rho_Q^{(m)}$ by
\begin{equation}
\label{alpha-rho-m}
\begin{dcases}
\alpha_P^{(m)}=\left(\alpha_1^{(m)}, \ldots, \alpha_p^{(m)}\right):=(1+\alpha_1-\rho_{\mathscr{N}(m)},\ldots, 1+\alpha_p-\rho_{\mathscr{N}(m)});\\
\rho_Q^{(m)}=\left(\rho_1^{(m)}, \ldots, \rho_q^{(m)}\right)
:=(1+\rho_0-\rho_{\mathscr{N}(m)}, \ldots, 1^*, \ldots, 1+\rho_{q}-\rho_{\mathscr{N}(m)}).
\end{dcases}
\end{equation}
\begin{lemma}
\label{difference-arbitrary-m}
For each $m\in {\bf Nar}$, we have an identity
\begin{equation}
\label{difference-alpha-rho-m}
\sum_{h=1}^{p}\alpha_h^{(m)}-\sum_{h=1}^{q}\rho_h^{(m)}
=-\nu\left({1\over 2}+{m\over d}\right)+{3-N\over 2}.
\end{equation}
\end{lemma}
\begin{proof}
By definition, we have 
\begin{eqnarray*}
\sum_{h=1}^{p}\alpha_h^{(m)}-\sum_{h=1}^{q}\rho_h^{(m)}
&=&(1-\rho_{\mathscr{N}(m)})(p-q-1)+\sum_{h=1}^{p}\alpha_h-\sum_{h=1}^{q}\rho_h.
\end{eqnarray*}
Now the result follows from~\eqref{reindex-narrow} and~\eqref{difference-alpha-rho}.
\end{proof}

\subsubsection{The modified small $I$-function}

Now we relate the small $I$-function in~\eqref{small-i-function} to the solutions of the generalized hypergeometric equation~\eqref{ghe-equation}.

For each $m\in {\bf Nar}$, 
the degree shifting number of $\sJ_m$ defined by~\eqref{degree-shifting-number} has the expression
$$\iota(\sJ_m)=\sum_{j=1}^{N}\left(\left\{{w_j\cdot m\over d}\right\}-{w_j\over d}\right).
$$
According to~\eqref{hodge-grading-operator}, we have 
\begin{equation}
\label{hodge-J^m}
\mu(J^m)=-{N\over 2}+\sum_{j=1}^{N}\left\{{w_j\cdot m\over d}\right\}.
\end{equation}

We define a {\em modified small $I$-function}
\begin{equation}
\label{modified-i-function}
\widetilde{I}(t,z):=z^{{\widehat{c}_W\over 2}-1}z^{\gr} I_{\rm FJRW}^{\rm sm}(t,z).
\end{equation}
Combining the formula~\eqref{small-i-function}, for fixed  $\ell\in\Z$, we compute the power of $z$ in the coefficient of $\sJ_m$, which is 
$$
{\widehat{c}_W\over 2}+\mu(J^m)-1+1+\sum\limits_{j=1}^{N}\lfloor {w_j\over d}(d\ell+m)\rfloor-(d\ell+m-1)
={\nu\over d}(d\ell+m-1)-1.$$
Thus we have
\begin{equation}
\label{modified-formula}
\widetilde{I}(t,z)=\sum_{m\in {\bf Nar}}
\sum_{\ell=0}^{\infty}
\frac
{t^{d\ell+m}}
{z^{{\nu\over d}(d\ell+m-1)}\Gamma(d\ell+m)}
\prod_{j=1}^{N}
{\Gamma({w_j\over d}(d\ell+m))\over 
\Gamma\left(\left\{{w_j\over d}\cdot m\right\}\right)}\sJ_m.
\end{equation}
We consider the following change of variables
\begin{equation}
\label{change-of-variable-formula}
x=\left({1\over d}\right)^{d}\cdot\left(\prod\limits_{j=1}^{N}w_j^{w_j}\right) z^{-\nu}.
\end{equation}
When we restrict $\widetilde{I}(t,z)$ to $t=1$ and use the change of variables~\eqref{change-of-variable-formula}, the coefficient functions become generalized hypergeometric series that appear in~\eqref{ghe-standard}. 
\begin{proposition}
\label{i-function-via-ghf}
For the modified small $I$-function in~\eqref{modified-i-function}, we have
\begin{eqnarray}
\widetilde{I}(1, z)
&=& {(2\pi)^{N+\nu-1\over 2}\over d^{1\over 2}} \prod_{j=1}^{N}w_j^{{w_j\over d}-{1\over 2}}\sum_{m\in {\bf Nar}}{{\Gamma(\alpha^{(m)}_P)\over \Gamma(\rho_Q^{(m)})}\, f_{\mathscr{N}(m)}(x)\over \prod\limits_{j=1}^{N}\Gamma(\left\{{w_j\cdot m\over d}\right\})}\,\sJ_m.\label{I-function-hypergeometric}
\end{eqnarray}
\end{proposition}
\begin{proof}

Applying the {\em Legendre duplication formula}
\begin{equation}
\label{legendre-prod}\prod_{k=0}^{d-1}\Gamma\left(y+{k\over d}\right)=(2\pi)^{d-1\over 2}d^{{1\over 2}-d\cdot y}\Gamma(d\cdot y),
\end{equation}
we obtain
\begin{equation*}
\begin{split}
{\Gamma(\alpha^{(m)}_P)\over \Gamma(\rho_Q^{(m)})}\, f_{\mathscr{N}(m)}(x)
=&\sum_{\ell=0}^{\infty}{\prod\limits_{j=1}^{N}\prod\limits_{k=0}^{w_j-1}\Gamma\left({m\over d}+{k\over w_j}+\ell\right)\over \prod\limits_{k=0}^{d-1}\Gamma\left({m+k\over d}+\ell\right)}
\, x^{\ell+{m-1\over d}} \\
=&\sum_{\ell=0}^{\infty} 
{(2\pi)^{1-N-\nu\over 2}\over d^{{1\over 2}-m-d\ell}}
\, \prod\limits_{j=1}^{N}w_j^{{1\over 2}-w_j({m\over d}+\ell)}
\, {\prod\limits_{j=1}^{N}\Gamma\left({w_j\over d}(m+d\ell)\right)\over \Gamma(m+d\ell)}
\, x^{\ell+{m-1\over d}}\\
=&{d^{{1\over 2}}\over (2\pi)^{N+\nu-1\over 2}}
\, \prod_{j=1}^{N}w_j^{{1\over 2}-{w_j\over d}}
\, \sum_{\ell=0}^{\infty} {t^{d\ell+m-1}\over z^{\nu\cdot(\ell+{m-1\over d})}}
\, {\prod\limits_{j=1}^{N}\Gamma\left({w_j\over d}(m+d\ell)\right)
\over \Gamma(m+d\ell)}.
\end{split}
\end{equation*}
The last equality is a direct consequence of the change of variable formula~\eqref{change-of-variable-formula}.
Using~\eqref{modified-formula}, this implies the formula~\eqref{I-function-hypergeometric} immediately.
\end{proof}


\subsubsection{Barnes' $Q$-function}
For each $m\in {\bf Nar}$, we introduce 
\begin{equation}
\label{barnes-Q}
Q_{\mathscr{N}(m)}(x):={\prod\limits_{i=0}^{q}\Gamma(\rho_{\mathscr{N}(m)}-\rho_i)^*\over \prod\limits_{i=1}^{p}\Gamma(\rho_{\mathscr{N}(m)}-\alpha_i)} \, x^{1-\rho_{\mathscr{N}(m)}}\, \ _pF_q\left(
\begin{array}{l}
\alpha_P^{(m)}\\
\rho_Q^{(m)}
\end{array}; (-1)^{\nu}x\right).
\end{equation}
Here the notation $\Gamma(\rho_{\mathscr{N}(m)}-\rho_i)^*$ means the term is omitted if $\rho_{\mathscr{N}(m)}-\rho_i=0.$
Since we work with the cases when $\rho_0=1$, these functions are exactly Barnes' functions  $\{Q_r(x)\}$ in~\cite{Bar}, where $r=\mathscr{N}(m)\in \{0, 1,\ldots, q\}$ by the map~\eqref{narrow-bijection}.

These functions are multi-valued. For any integer $k$, we have
\begin{equation}
\label{Q-multivalued}
Q_{\mathscr{N}(m)}(x\cdot e^{2\pi\sqrt{-1}k})=e^{2\pi\sqrt{-1}k(1-\rho_{\mathscr{N}(m)})}Q_{\mathscr{N}(m)}(x).
\end{equation}
For each $m\in {\bf Nar}$, we define a constant
\begin{equation}
\label{constant-upsilon}
\Upsilon(m):=(-1)^{\nu(1-\rho_{\mathscr{N}(m)})} {\Gamma(\alpha^{(m)}_P) \Gamma(1-\alpha^{(m)}_P)\over \Gamma(\rho_Q^{(m)})\Gamma(1-\rho_Q^{(m)})}.
\end{equation}
Using the definitions in~\eqref{ghe-standard} and~\eqref{alpha-rho-m}, we immediately obtain
$$Q_{\mathscr{N}(m)}(x)={1\over \Upsilon(m) }\cdot {\Gamma(\alpha^{(m)}_P)\over \Gamma(\rho_Q^{(m)})}\cdot f_{\mathscr{N}(m)}\Big((-1)^{\nu}x\Big).$$

We denote the primitive $d$-th root of unity by
\begin{equation}
\label{d-th-root}
\omega:=\exp\left({2\pi\sqrt{-1}/d}\right).
\end{equation} 
Using formula~\eqref{I-function-hypergeometric} and~\eqref{Q-multivalued}, the coefficient functions of $\widetilde{I}(1, -z)$ can be written as Barnes' $Q$-functions:
\begin{equation}
\widetilde{I}(1, -z)
= {(2\pi)^{N+\nu-1\over 2}\over d^{1\over 2}}  \prod_{j=1}^{N}w_j^{{w_j\over d}-{1\over 2}}\sum_{m\in {\bf Nar}}
{\omega^{-\nu(m-1)} \Upsilon(m) \over\prod\limits_{j=1}^{N}\Gamma(\left\{{w_j\cdot m\over d}\right\})} \, Q_{\mathscr{N}(m)}(x)\, \sJ_m.
\label{i-function-via-upsilon}
\end{equation}

\subsubsection{Another formula for $\Upsilon(m)$}
Using the {\em Euler reflection formula}
\begin{equation}
\label{euler-reflection}
\Gamma(x)\Gamma(1-x)(1-e^{-2\pi\sqrt{-1}x})=2\pi\sqrt{-1}e^{-\pi\sqrt{-1}x}
\end{equation}
and the relation~\eqref{reindex-narrow},
we find an identity for $\Upsilon(m)$.
\begin{lemma}
For any $m\in {\bf Nar}$, we have
\begin{equation}
\label{upsilon-2nd-formula}
\Upsilon(m)
=d(2\pi)^{1-\nu}e^{\pi\sqrt{-1}\left(-{\nu\over d}-1-{N\over 2}\right)}\prod_{j=1}^{N}{1\over 1-\omega^{mw_j}}.
\end{equation}
\end{lemma}
\begin{proof}
We have 
\begin{eqnarray*}
\Upsilon(m)
&=&(-1)^{\nu({m-1\over d})}\prod_{h=1}^{p}{2\pi\sqrt{-1}e^{-\pi \sqrt{-1}\alpha_h^{(m)}}\over 1-e^{-2\pi\sqrt{-1}\alpha_h^{(m)}}}\prod_{i=1}^{q}{1-e^{-2\pi\sqrt{-1}\rho_i^{(m)}}\over 2\pi\sqrt{-1}e^{-\pi \sqrt{-1}\rho_i^{(m)}}}\\
&=&
(2\pi)^{p-q}e^{\pi\sqrt{-1}\left({\nu({m-1\over d})}+{p-q\over 2}+\nu\left({1\over 2}+{m\over d}\right)-{3-N\over 2}\right)}\prod\limits_{\substack{1\leq i\leq d\\ i\neq d-\rho_{\mathscr{N}(m)}}}\left(1-e^{2\pi\sqrt{-1}\left(\rho_{\mathscr{N}(m)}-1+{i\over d}\right)}\right)\\
&& \cdot\prod\limits_{j=1}^{N}\prod\limits_{k=0}^{w_j-1}\left(1-e^{2\pi\sqrt{-1}\left(\rho_{\mathscr{N}(m)}-1-({1\over d}+{k\over w_j})\right)}\right)^{-1}\\
&=&
(2\pi)^{1-\nu}e^{\pi\sqrt{-1}\left({2m\nu\over d}-{\nu\over d}-1+{N\over 2}\right)}d \prod\limits_{j=1}^{N}\prod\limits_{k=0}^{w_j-1}\left(1-e^{2\pi\sqrt{-1}(-{m\over d}-{k\over w_j})}\right)^{-1}\\
&=&d(2\pi)^{1-\nu}e^{\pi\sqrt{-1}\left({2m\nu\over d}-{\nu\over d}-1+{N\over 2}\right)} \prod_{j=1}^{N}(1-\omega^{-mw_j})^{-1}\\
&=&d(2\pi)^{1-\nu}e^{\pi\sqrt{-1}\left(-{\nu\over d}-1-{N\over 2}\right)}\prod_{j=1}^{N}\left(1-\omega^{mw_j}\right)^{-1}.
\end{eqnarray*}
Here in the second equality, we use the identities~\eqref{alpha-rho-m},~\eqref{difference-alpha-rho-m}, and multiply both the denominator and the numerator with the product of terms of the form $1-e^{2\pi\sqrt{-1}\left(\rho_{\mathscr{N}(m)}-1-{n+1\over d}\right)}$, where $n$ is given by~\eqref{common-numbers}.
\end{proof}

\subsection{Meijer G-function and Barnes asymptotic formula}
\label{sec-barnes-formula}
Now we introduce a Barnes asymptotic formula~\eqref{Barnes-formula} \cite{Bar}. 
We need some knowledge on Meijer $G$-functions \cite{Mei}, which are integral solutions of generalized hypergeometric equations. We mainly follow the work of Luke~\cite{Luk} and Fields~\cite{Fie}.


\subsubsection{Meijer $G$-function}
For an arbitrary $p$-tuple 
$(a_1, \ldots, a_p)\in \mathbb{R}^p$ and an arbitrary $q$-tuple 
$(b_1, \ldots, b_q)\in \mathbb{R}^q,$ such that 
$$\{a_k-b_j\mid 1\leq k\leq p, 1\leq j\leq q\}\cap \Z_+=\emptyset,$$ 
we define a Meijer $G$-function by integrals~\cite[Section 5.2]{Luk}
\begin{equation}
\label{meijer-def}
G_{p,q}^{m,n}
\left(
\begin{array}{l}
a_1, \ldots, a_p\\
b_1, \ldots, b_q
\end{array}
; x
\right)
={1\over 2\pi\sqrt{-1}}
\int_L{\prod\limits_{j=1}^{m}\Gamma(b_j-s)\prod\limits_{j=1}^{n}\Gamma(1-a_j+s)\over\prod\limits_{j=m+1}^{q}\Gamma(1-b_j+s)\prod\limits_{j=n+1}^{p}\Gamma(a_j-s)}x^s ds,
\end{equation}
where $L$ is a path goes from $-\sqrt{-1}\infty$ to $\sqrt{-1}\infty$ so that 
all poles of $\Gamma(b_j-s)$, $j=1, 2, \ldots, m$, lies to the rights of the path, 
and all poles of $\Gamma(1-a_k+s)$, $k=1,2,\ldots, n,$ lies to the left of the path.
The Meijer $G$-function in~\eqref{meijer-def} is a solution of the linear differential equation (see~\cite[Section 5.8 (1)]{Luk})
\begin{equation}
\label{diff-eqn-meijer}
\Big(\prod_{j=1}^{q}(\vartheta_x-b_j)+(-1)^{p+1-m-n}\,x\prod_{i=1}^{p}(\vartheta_x+1-a_i)\Big) f(x) =0.
\end{equation}

\begin{remark}
\label{residue-Meijer-G}
We can evaluate the integral appears in~\eqref{meijer-def} as a sum of residues by deforming the path to a loop beginning and ending at $+\infty$ and encircling all poles of $\Gamma(b_j-s)$, $j=1, 2, \ldots, m$, once in the negative direction, but none of the poles of  $\Gamma(1-a_k+s)$, $k=1,2,\ldots, n.$ 
If no two of $\{b_j\}_{j=1}^{m}$ differ by an integer, the residue expressions of the Meijer $G$-function will be generalized hypergeometric functions~\cite[Section 5.2 (7)]{Luk}. 
\end{remark}

We set the convention
\begin{equation}
\label{replace-index}
\begin{dcases}
a_i:=1-\alpha_i, & i=1,2,\ldots, p;
\\
b_j:=1-\rho_j, & j=0,1,\ldots, q.
\end{dcases}
\end{equation}
Here  $\{\alpha_i\}_{i=1}^{p}$ and $\{\rho_j\}_{j=1}^{q}$ are not necessarily the sets defined in~\eqref{alpha-tuple} and~\eqref{rho-tuple}. 
According to~\cite[Section 5.2 (14)]{Luk}, when $|x|<1$, the generalized hypergeometric function is the residue expression of a Meijer $G$-function
\begin{equation}
\label{ghe-meijer}
_pF_q\left(
\begin{array}{l}
\alpha_1, \ldots, \alpha_p\\
\rho_1,\ldots, \rho_q
\end{array}; x\right)=
{\prod\limits_{j=1}^{q}\Gamma(\rho_j)\over \prod\limits_{j=1}^{p}\Gamma(\alpha_j)}
G_{p,q+1}^{1,p}\left(
\begin{array}{l}
1-\alpha_1, \ldots, 1-\alpha_p\\
0, 1-\rho_1,\ldots, 1-\rho_q
\end{array}; -x\right).
\end{equation}

We may also rewrite the equation~\eqref{ghe-equation} as 
\begin{equation}
\label{ghe-meijer-form}
\Big(\prod_{j=0}^{q}(\vartheta_x-b_j)-x\prod_{i=1}^{p}(\vartheta_x+1-a_i)\Big) f(x) =0.
\end{equation}


\subsubsection{Barnes' asymptotic formula of exponential type}
Let $p<q+1$, we consider the Meijer $G$-function 
\begin{equation}\label{exponential-G-function}
G_{p,q+1}^{q+1,0}
\left(
\begin{array}{l}
a_1, \ldots, a_p\\
b_0, b_1,\ldots, b_q
\end{array}
; x \right)
={1\over 2\pi\sqrt{-1}}
\int_L{\prod\limits_{j=0}^{q}\Gamma(b_j-s)\over\prod\limits_{j=1}^{p}\Gamma(a_j-s)}x^s ds.
\end{equation}

If no two of $\{b_j\}_{j=0}^{q}$  differ by an integer, the residue calculation in Remark~\ref{residue-Meijer-G} for the Meijer $G$-function in~\eqref{exponential-G-function} is a special case of \cite[Section 5.2 (7)]{Luk}:
\begin{equation}\label{Barnes-Meijer}
 G_{p,q+1}^{q+1,0}
\left(
\begin{array}{l}
a_1, \ldots, a_p\\
b_0, b_1,\ldots, b_q
\end{array}
; x \right)=\sum_{r=0}^{q}Q_r(x),
 \quad |x|<1.
\end{equation}

In~\cite[Section 47]{Bar}, Barnes proved an asymptotic expansion formula for $\sum_{r=0}^{q}Q_r(x)$ 
 by analyzing a contour integral using residues and asymptotic expansion of Gamma functions.
Barnes' asymptotic expansion formula (of exponential type) can be interpreted as the following asymptotic expansion of Meijer $G$-function in~\eqref{exponential-G-function} near $x=\infty$ (cf. \cite[Theorem C]{Mei} and \cite[Section 5.7, Theorem 5]{Luk})
\begin{equation}
\label{exponential-asymptotic}
G_{p,q+1}^{q+1,0}
\left(
\begin{array}{l}
a_1, \ldots, a_p\\
b_0, b_1,\ldots, b_q
\end{array}
; x \right)\sim  H_{p, q+1}(x),
\end{equation}
as $x\to \infty$, $|\arg x|< (\nu+\epsilon)\pi$, with $\epsilon=1/2$ if $\nu=1$ and $\epsilon=1$ if $\nu>1$.
Here 
 \begin{equation}
 \label{exponential-H-part}
 H_{p, q+1}(x):=
 {(2\pi)^{{1\over 2}(\nu-1)} \over \nu^{{1\over 2}}} \exp(-\nu x^{1\over \nu})\, x^{\theta}\left(1+\sum_{n=1}^{\infty}M_n x^{-{n\over\nu}}\right).
 \end{equation}
The constants $\{M_n\}$ are determined recursively and the constant $\theta$ satisfies 
\begin{equation}
\label{asymptotic-polynomial-exponent}
\theta\nu={1-\nu\over 2}+\sum_{j=0}^{q}b_j-\sum_{j=1}^{p}a_j.
\end{equation}

\begin{theorem}\cite{Bar}\label{Barnes-theorem}
If $p<q+1$ and no two of $\{\rho_j\}_{j=0}^{q}$ differ by an integer,
there is an asymptotic expansion
\begin{equation}
\label{Barnes-formula}
\sum_{r=0}^{q}Q_r(x)
\sim  {(2\pi)^{{1\over 2}(\nu-1)} \over \nu^{{1\over 2}}} \exp(-\nu x^{1\over \nu}) \, x^{\theta} \left(1+\sum_{n=1}^{\infty}M_n x^{-{n\over\nu}}\right),
\end{equation}
as $x\to \infty$, $|\arg x|<(\nu+\epsilon)\pi$, with $\epsilon=1/2$ if $\nu=1$ and $\epsilon=1$ if $\nu>1$.
\end{theorem}

\subsection{Calculation of weak asymptotic classes}
\label{calculation-weak-strong}

If mirror conjecture~\ref{conjecture-i-function-formula} holds for the admissible LG pair $(W, \<J\>)$, we can find weak asymptotic classes 
 in a few steps:
\begin{enumerate}
\item
We express the coefficients of the small $I$-function in terms of generalized hypergeometric functions after the change of variable formula~\eqref{change-of-variable-formula}.
\item
The term $(-2\pi z)^{{\widehat{c}_W\over 2}}\<\Phi(\alpha), \one\>$ is a linear combination of Barnes' functions $\{Q_r(x)\}$, for any $\alpha\in  \cH_{\rm nar}.$
\item 
We can choose a cohomology class such that the corresponding linear combination match the one in Barnes' asymptotic formula~\eqref{Barnes-formula}. Barnes' result will imply that the class is a weak asymptotic class.
\item
We can find weak asymptotic classes with respect to other rays by rotation~\eqref{Q-multivalued}.
\item
We show the weak asymptotic classes for each ray forms  a 1-dimensional space, if exists.
\end{enumerate}

The first two steps have been completed in Section~\ref{sec-basic-ghe}. For step (3) and (4), we consider 
\begin{equation}
\label{asymptotic-class-definition}
\cA_\ell:=\sum_{m\in {\bf Nar}} \omega^{-\ell\cdot m} \prod\limits_{j=1}^{N}{2\pi \over \Gamma(\left\{{w_j\cdot m\over d}\right\})}\,\sJ_m\in \cH_{W, \<J\>}, \quad \ell\in \mathbb{Z}.
\end{equation}

Obviously, $\cA_\ell=\cA_{\ell+d}.$

\begin{lemma}
\label{lemma-dual-asymptotic}
If mirror conjecture~\ref{conjecture-i-function-formula} holds, then for any $\ell\in\Z$, 
\begin{equation*}
\begin{split}
(-2\pi z)^{{\widehat{c}_W\over 2}}\<\Phi(\cA_\ell), \one\>
=
C \sum_{m\in {\bf Nar}}
\omega^{\ell\cdot m}Q_{\mathscr{N}(m)}(x),
\end{split}
\end{equation*}
where the constant
$$C=(2\pi)^{N-\nu+1\over 2}\, d^{{1\over 2}}\, e^{\pi\sqrt{-1}\left({\nu\over d}-1-{3N\over 2}\right)} \, \prod_{j=1}^{N}w_j^{{w_j\over d}-{1\over 2}}$$
  is independent of the choice of $\ell$.
\end{lemma}
\begin{proof}
Using Corollary~\ref{weak-via-i-function} and~\eqref{i-function-via-upsilon}, the contribution of $(-2\pi z)^{{\widehat{c}_W\over 2}}\<\Phi(\cA_\ell), \one\>$ from each $m\in {\bf Nar}$ is given by the product of the constant in~\eqref{i-function-via-upsilon} and 
\begin{equation}
\label{component-weak}
e^{\pi\sqrt{-1}\mu(J^{d-m})}\,{(2\pi)^N \omega^{-\ell\cdot (d-m)} \over \prod\limits_{j=1}^{N}\Gamma\left(\left\{{w_j(d-m)\over d}\right\}\right)}\cdot {\omega^{-\nu(m-1)} \Upsilon(m) \over\prod\limits_{j=1}^{N}\Gamma(\left\{{w_j\cdot m\over d}\right\})} \, Q_{\mathscr{N}(m)}(x).
\end{equation}
Using formulas~\eqref{hodge-J^m},~\eqref{upsilon-2nd-formula}, and~\eqref{euler-reflection}, the formula~\eqref{component-weak} can be simplified to 
$$d(2\pi)^{1-\nu}e^{\pi\sqrt{-1}\left(-{\nu\over d}-1-{N\over 2}\right)}\, (-1)^{-N}\, \omega^{\nu+\ell m}\, Q_{\mathscr{N}(m)}(x).$$
Now the result follows from further simplification.
\end{proof}

According to~\eqref{reindex-narrow} and~\eqref{Q-multivalued}, we have 
\begin{equation}
\label{Barnes-formula-shift}
Q_{\mathscr{N}(m)}(x\cdot e^{2\pi\sqrt{-1}\ell})=e^{2\pi\sqrt{-1}\ell\cdot {m-1\over d}}Q_{\mathscr{N}(m)}(x)=\omega^{\ell(m-1)}Q_{\mathscr{N}(m)}(x).
\end{equation}




Let $\{\alpha_i\mid 1\leq i\leq p\}$ and $\{\rho_j\mid 0\leq j\leq q\}$ in~\eqref{replace-index} be the sets of rational numbers defined in~\eqref{alpha-tuple} and~\eqref{reindex-narrow}. 
Thus no two of $\{\rho_j\}_{j=0}^{q}$ differ by an integer. 
Using identities in~\eqref{difference-alpha-rho} and in~\eqref{asymptotic-polynomial-exponent}, Theorem~\ref{Barnes-theorem} implies
\begin{corollary}
\label{barnes-corollary}
For each integer $\ell=1-\nu, \ldots, 0$,
there is an asymptotic expansion
\begin{eqnarray*}
&&\sum\limits_{m\in {\bf Nar}}\omega^{\ell\cdot m}Q_{\mathscr{N}(m)}(x)\\
&\sim&e^{(2-N)\pi\sqrt{-1}{\ell\over \nu}} 
{(2\pi)^{{1\over 2}(\nu-1)} \over \nu^{{1\over 2}}}
e^{(-\nu (x\, e^{2\pi\sqrt{-1}\ell})^{1\over \nu})}
x^{{2-N\over 2\nu}-{1\over d}}
\left(1+\sum_{n=1}^{\infty}M_n x^{-{n\over\nu}}\right),
\end{eqnarray*}
as $x\to \infty$, $|\arg x+2\pi\ell |<(\nu+\epsilon)\pi$.
\end{corollary}

Using the change of variable formula~\eqref{change-of-variable-formula} and the formula~\eqref{value-principal-spectrum} for the value $T$, we obtain
$$-\nu (x\, e^{2\pi\sqrt{-1}\ell})^{1\over \nu}=-{T\,e^{2\pi\sqrt{-1}\ell/\nu}\over z}.$$
Now the following result on weak asymptotic classes is a consequence of the Barnes' formula~\eqref{Barnes-formula} and Corollary~\ref{barnes-corollary}.
\begin{proposition}
\label{theorem-asymptotic-classes}
If mirror conjecture~\ref{conjecture-i-function-formula} holds 
, then the class $\cA_\ell$ defined in~\eqref{asymptotic-class-definition} is a weak asymptotic class with respect to $T e^{2\pi\sqrt{-1}\ell/\nu}$ for each integer $\ell=1-\nu, \ldots, 0$. 
In particular, $\cA_0$ is a principal weak asymptotic class. 
\end{proposition}

Now the result below is a consequence of Proposition~\ref{mirror-theorem-fermat}.
\begin{corollary}
\label{corollary-weak-classes}
Let $W$ be an invertible Fermat polynomial.
The class $\cA_\ell$ defined in~\eqref{asymptotic-class-definition} is a weak asymptotic class with respect to $T e^{2\pi\sqrt{-1}\ell/\nu}$ each integer $\ell=1-\nu, \ldots, 0$.
\end{corollary}

We will show in Section~\ref{sec-dominance-order} that the weak asymptotic classes for each such $\ell$ form a one-dimensional subspace.

\subsection{Smallest asymptotic versus dominant asymptotic}
\label{sec-dominance-order}
Now we study the solutions of generalized hypergeometric equations near $x=\infty$ in more details.
Using the dominance order for complex functions, we introduce the smallest asymptotic expansion and dominant asymptotic expansion.
Then we relate these concepts to the weak asymptotic classes and the strong asymptotic classes.

\subsubsection{A basis of solutions near $x=\infty$}

If no two of $\{a_i\}_{i=1}^{p}$ differ by an integer, Meijer showed in~\cite[Page 344-356]{Mei} that the equation
\begin{equation}\label{ghe-fundamental-solutions}
\Big(\prod_{j=0}^{q}(\vartheta_x-b_j)-(-1)^{\nu}x\prod_{i=1}^{p}(\vartheta_x+1-a_i)\Big) f(x) =0
\end{equation}
has a fundamental system of solutions 
\begin{equation}
\left\{G_{p,q+1}^{q+1,0}\left(x\, e^{2\pi\sqrt{-1}\ell} \right), 
G_{p,q+1}^{q+1,1}\left( x\, e^{\pi\sqrt{-1}(1-2\lambda)}|| a_t \right)\bigg\vert 1-\nu \leq \ell\leq 0, 1\leq t\leq p \right\},
\end{equation}
in the sector such that
$$|\arg x+(1-2\lambda)\pi |<\left({\nu\over 2}+1\right)\pi, \quad |\arg x+2\pi\ell|<(\nu+\epsilon)\pi \quad (\forall \ell).$$
Here $G_{p,q+1}^{q+1,0}\left(x\right)$ is the abbreviation of the Meijer $G$-function in~\eqref{exponential-G-function} and 
\begin{equation}
\label{algebraic-meijer-G}
G_{p,q+1}^{q+1,1}\left( x || a_t \right):=
G_{p,q+1}^{q+1,1}
\left(
\begin{array}{c}
a_t, a_1, \ldots, a_{t-1}, a_{t+1}, \ldots, a_p\\
b_0, b_1, \ldots, b_q
\end{array}
; x
\right).
\end{equation}

The asymptotic behavior of these Meijer $G$-functions near $x=\infty$ has been studied by Barnes in~\cite{Bar}, given by the {\em exponential type asymptotic formula}~\eqref{exponential-asymptotic} and the {\em algebraic type asymptotic formula}~\cite[Theorem A]{Mei} for the Meijer $G$-function~\eqref{algebraic-meijer-G} below:
\begin{equation}
\label{asymptotic-algebraic}
G_{p,q+1}^{q+1,1}\left( x || a_t \right)\sim E_{p, q+1}(x || a_t):=x^{-1+a_t}\sum_{k=0}^{\infty}(-x)^{-k}
{\prod\limits_{j=0}^{q}\Gamma(1+b_j-a_t+k)\over k!\prod\limits_{\substack{j=1\\j\neq t}
}^{p}\Gamma(1+a_j-a_t+k)},
\end{equation}
as $x\to \infty$, $|\arg x|<({\nu\over 2}+1)\pi.$

The series of $E_{p, q+1}(x || a_t)$ in~\eqref{asymptotic-algebraic} is divergent. 
It also takes the form of generalized hypergeometric notation
$$
x^{-1+a_t}
{\prod\limits_{j=0}^{q}\Gamma(1+b_j-a_t)\over \prod\limits_{\substack{j=1\\j\neq t}}^{p}\Gamma(1+a_j-a_t)}
\, _{q+1}F_{p-1}\left(
\begin{array}{l}
1+b_0-a_t, \ldots, 1+b_q-a_t\\
1+a_1-a_t, \ldots, 1^*, \ldots, 1+a_p-a_t
\end{array}; -{1\over x}\right).
$$

If the condition that no two $\{a_i\}_{i=1}^{p}$ differ by an integer does not hold, the series is not well defined.
However, one can replace the set $\{E_{p, q+1}(x || a_t)\}_{t=1}^{p}$ by a set of linearly independent solutions using Frobenius method or the limiting processes in~\cite[Section 5.1]{Luk}, without changing the algebraic character. 
For convenience, we will denote such a set by $\{\widetilde{E}_{p, q+1}(x || a_t)\}_{t=1}^{p}$ temporarily, without specifying the elements explicitly. 

Let us denote the space of solutions of the generalized hypergeometric equation~\eqref{ghe-fundamental-solutions} near $x=0$ (or $x=\infty$) by $V^{0}(x)$ (or $V^{\infty}(x)$).
According to \cite[Section 5.8]{Luk},  we have
\begin{proposition}
\label{two-basis-meijer}
Near $x=0$, the set $\{Q_r(x)\}_{r=0}^{q}$ forms a basis of of the space $V^{0}(x)$.
Near $x=\infty$, the set 
\begin{equation}
\label{meijer-basis}
\mathcal{B}_{\infty}:=\{H_{p, q+1}(x\, e^{2\pi\sqrt{-1}\ell}), \widetilde{E}_{p, q+1}( x\, e^{\pi\sqrt{-1}(1-2\lambda)}|| a_t) \mid 1-\nu\leq \ell \leq 0,  1\leq t\leq p\}
\end{equation}
forms a basis of the space $V^{\infty}(x)$ in a certain sector.

Moreover, $V^{\infty}(x)$ has a decomposition
$$V^{\infty}(x)=V_{\rm exp}^{\infty}(x)\bigoplus V_{\rm alg}^{\infty}(x),$$
where the subspace $V_{\rm exp}^{\infty}(x)$ is spanned by $H_{p, q+1}(x\, e^{2\pi\sqrt{-1}\ell})$, $1-\nu\leq \ell \leq 0$,
and the subspace $V_{\rm alg}^{\infty}(x)$ is spanned by $\widetilde{E}_{p, q+1}(x\, e^{\pi\sqrt{-1}(1-2\lambda)}|| a_t)$, $1\leq t\leq p$.
\end{proposition}

\subsubsection{Dominance orders in asymptotic functions}
For each $\theta\in \mathbb{R}$, we consider a partial order $\preccurlyeq_\theta$ (and a strict order $\prec_\theta$) for complex functions (in $z$) via asymptotic expansions.
\begin{definition}
Fix $\arg z:=\theta\in  \mathbb{R}$.
We say $f(z)\preccurlyeq_\theta g(z)$ if
$$f(z)=O(g(z)), \quad  \text{as} \quad z\to 0.$$ 
We say $f(z)\prec_\theta g(z)$ if
$$f(z)=o(g(z)),  \quad  \text{as} \quad z\to 0.$$ 
\end{definition}
The order $\preccurlyeq_\theta$ (or $\prec_\theta$) is usually called the {\em dominance order}. 
It is clear that the order depends on the choices of $\theta$.

Recall that in the change of variable formula~\eqref{change-of-variable-formula}, $x$ is a scale of $z^{-\nu}.$
Thus $x\to \infty$ when $z\to 0$.
The dominance orders of the functions in $\mathcal{B}_{\infty}$ are determined by the exponential part of the asymptotic. 
\begin{lemma}
\label{lem-smallest-dominant}
Consider $\arg(z)=\theta=2\pi\ell/\nu$ such that $1-\nu\leq \ell\leq 0.$ 
The function $G_{p,q+1}^{q+1,0}(x\, e^{2\pi\sqrt{-1}\ell})$ has the smallest asymptotic expansion with respect to $\prec_\theta$ among all the solutions of the equation~\eqref{ghe-fundamental-solutions}. 
\end{lemma}
\begin{proof}
When $\arg(z)=\theta=2\pi\ell/\nu$, we have $\arg(x\, e^{2\pi\sqrt{-1}\ell})=0$.
We observe from the formula~\eqref{exponential-asymptotic} that the asymptotic expansion of $G_{p,q+1}^{q+1,0}(x\, e^{2\pi\sqrt{-1}\ell})$ has an exponential part $\exp(-\nu (x\, e^{2\pi\sqrt{-1}\ell})^{1\over \nu})$. 
Now the result follows from the fact that the real part of $-\nu (x\, e^{2\pi\sqrt{-1}\ell})^{1\over \nu}$ is strictly smaller than those obtained from any other elements in the basis~\eqref{meijer-basis}.
\end{proof}


Barnes' formula~\eqref{Barnes-formula} gives the corresponding linear combination the smallest asymptotic expansion in a small sector containing the ray $ \mathbb{R}_{>0}$. 
In the original paper, Barnes mainly considered the application to dominant asymptotic expansions~\cite[Theorem (A), Theoreom (B)]{Bar}, generalizing earlier results of Stokes and Orr. Now we recall some results on the dominant asymptotic expansions for generalized hypergeometric functions. 
According to~\cite[Section 20]{Mei} and~\cite[Section 5.11.2 (4)]{Luk}, there is a dominant asymptotic expansion
\begin{equation}
\label{dominant-expansion-ghe}
_pF_q\left(
\begin{array}{l}
\alpha^{(m)}_1, \ldots, \alpha^{(m)}_p\\
\rho^{(m)}_1,\ldots, \rho^{(m)}_q
\end{array}; x\right)
\sim {\prod\limits_{j=1}^{q}\Gamma(\rho^{(m)}_j)\over \prod\limits_{j=1}^{p}\Gamma(\alpha^{(m)}_j)} K_{p,q}(x),
\end{equation}
as $x\to \infty$, $|\arg x|<\pi$.
Here 
\begin{equation}
\label{exponential-term}
K_{p,q}(x):={(2\pi)^{{1\over 2}(\nu-1)}\over \nu^{{1\over 2}}}\exp(\nu x^{1\over \nu})\, x^{\gamma^{(m)}}\left(1+\sum_{n=1}^{\infty}N_n x^{-{n\over\nu}}\right)
\end{equation}
is the formal function given in \cite[Section 5.11 (19)]{Luk} with
$$\gamma^{(m)}\nu={\nu-1\over 2}+\sum_{h=1}^{p}\alpha^{(m)}_h-\sum_{h=1}^{q}\rho^{(m)}_h.$$


\subsubsection{Smallest asymptotic expansion and weak asymptotic classes}

According to Lemma~\ref{lemma-dual-asymptotic} and Corollary~\ref{barnes-corollary}, the weak asymptotic class $\cA_\ell $ in Proposition~\ref{theorem-asymptotic-classes} can be interpreted as a class such that  when $\arg(z)=\theta=2\pi\ell/\nu,$ the function $(-2\pi z)^{{\widehat{c}_W\over 2}}\<\Phi(\cA_\ell), \one\>$ has the smallest asymptotic expansion with respect to $\prec_\theta$ among 
$$\{(-2\pi z)^{{\widehat{c}_W\over 2}}\<\Phi(\alpha), \one\>\mid \alpha\in \cH_{W, \<J\>}\cap H_{\rm nar}\}.$$

As a consequence of Proposition~\ref{two-basis-meijer} and Lemma~\ref{lem-smallest-dominant}, we prove
\begin{proposition}
\label{cor-smallest}
If mirror conjecture~\ref{conjecture-i-function-formula} holds, then the space of weak asymptotic classes with respect to $T e^{2\pi\sqrt{-1}\ell/\nu}$ is spanned by $\cA_\ell$ for each $\ell=1-\nu, \ldots, 0$.
\end{proposition}

\subsubsection{Dominant asymptotic expansion and strong asymptotic classes}
Now we return to the strong asymptotic classes discussed in Section~\ref{sec-asymptotic-classes}.
According to Proposition~\ref{prop-strong-class}, when quantum spectrum conjecture~\ref{conjecture-C} holds, the strong asymptotic classes with respect to the eigenvalue $T e^{2\pi\sqrt{-1}\ell/\nu}$ will be determined by the dominant asymptotic expansion of $\widetilde{J}(\tau,z)$ as $z\to 0$ and $\arg(z)=2\pi\ell/\nu$. 
Furthermore, if mirror conjecture~\ref{conjecture-i-function-formula} holds, then 
$$\widetilde{J}(\tau,z)=z^{\widehat{c}_W\over 2} z^{\gr}  S(\tau, z)^{-1}(\one)=\widetilde{I}(t, z)\Big\vert_{t=1}.$$
According to~\eqref{I-function-hypergeometric}, the coordinate function of $\widetilde{J}(\tau,z)$ with respect to $\sJ_m$ is a scalar multiple of $f_{\mathscr{N}(m)}(x)$. 
So let us consider the dominant asymptotic expansion of  $f_{\mathscr{N}(m)}(x)$ for each $m\in {\bf Nar}$.

Using the change-of-coordinates formula~\eqref{change-of-variable-formula}, we consider  $\widetilde{x}:=e^{2\pi\sqrt{-1}\ell}x$.
As before, when $\arg(z)=\theta=2\pi\ell/\nu$, we have $\arg(\widetilde{x})
=0$.
For each $m\in {\bf Nar}$, we apply the dominant asymptotic expansion~\eqref{dominant-expansion-ghe} and obtain
\begin{align*}
f_{\mathscr{N}(m)}(x)
&=\omega^{\ell(1-m)}\, \widetilde{x}^{m-1\over d}\  _pF_q\left(
\begin{array}{l}
\alpha^{(m)}_1, \ldots, \alpha^{(m)}_p\\
\rho^{(m)}_1,\ldots, \rho^{(m)}_q
\end{array}; \widetilde{x}\right)\\
&\sim \omega^{\ell(1-m)} {\Gamma(\rho_Q^{(m)})\over {\Gamma(\alpha^{(m)}_P)}} {(2\pi)^{{1\over 2}(\nu-1)}\over \nu^{{1\over 2}}}\exp(\nu \widetilde{x}^{1\over \nu})\, \widetilde{x}^{{2-N\over 2\nu}-{1\over d}}.
\end{align*}
Here $\gamma^{(m)}={2-N\over 2\nu}-{m\over d}$ in~\eqref{exponential-term} is calculated using~\eqref{difference-alpha-rho-m}.

For each $f_{\mathscr{N}(m)}(x)$, the dominant asymptotic expansion share the same leading term $\exp(\nu \widetilde{x}^{1\over \nu})\, \widetilde{x}^{{2-N\over 2\nu}-{1\over d}}$. 
Applying~\eqref{I-function-hypergeometric}, we obtain
\begin{align*}
\lim_{\substack{\arg(z)=2\pi \ell/\nu\\|z|\to+0}}{\widetilde{J}(\tau, z) \over \<\one,  \widetilde{J}(\tau, z) \>}
\propto\sum_{m\in {\bf Nar}} \omega^{\ell (1-m)} \prod\limits_{j=1}^{N}{1 \over \Gamma(\{{w_j\cdot m\over d}\})}\,\sJ_m
\propto\cA_\ell.
\end{align*}
According to Proposition~\ref{prop-strong-class}, we have
\begin{proposition}
\label{weak=strong}
If both quantum spectrum conjecture~\ref{conjecture-C} and mirror conjecture~\ref{conjecture-i-function-formula} holds, 
then for each $\ell=1-\nu, \ldots, 0$, the subspace of strong asymptotic classes with respect to $T e^{2\pi\sqrt{-1}\ell/\nu}$ is spanned by $\cA_\ell$.
\end{proposition}



\section{Matrix factorization and Gamma structures}

Let $R=\C[x_1, \ldots, x_N]$ be a polynomial ring with $\deg(x_j)=w_j\in \mathbb{N}.$ Let $(W, G)$ be an admissible LG pair.
Let ${\rm MF}(R, W)$ be the category of {\em matrix factorizations of $W$} and ${\rm MF}_{G}(R, W)$ be the category of {\em $G$-equivariant matrix factorizations of $W$}~\cite{Wal, Orl, PV}. 
These categories and their homotopy categories have been used to describe the $D$-branes in topological Laudau-Ginzburg models~\cite{KaL, Orl1, Orl}. 

In~\cite{CIR}, Chiodo, Iritani, and Ruan constructed Gamma structures in FJRW theory and use the Gamma structure to build a connection from the category of matrix factorizations ${\rm MF}_{G}(R, W)$ to the FJRW theory of $(W, G)$. 

We review these results and then relate the Gamma structures to the asymptotic expansions in Section~\ref{sec-asymptotic} and Section~\ref{sec-weak}. 
The only new result in this section is Corollary~\ref{corollary-chern-asymptotic}.
Later in Section~\ref{sec-orlov}, we will explain the relation between the asymptotic classes and Orlov's semiorthogonal decomposition of the category of matrix factorizations.




\subsection{The category of $G$-equivariant matrix factorizations of $W$}
Following~\cite{PV}, an object in the category ${\rm MF}_{G}(R, W)$ is a pair
$$(E, \delta_E)=(E^0 \xrightleftarrows[\delta_1]{\delta_0} E^1),$$
where
$E=E^0\oplus E^1$ is a $G$-equivariant $\Z/2$-graded finitely generated projective $R$-module and
$\delta_E\in {\rm End}_R^1(E)$ is an odd $G$-equivariant endomorphism of E, such that $$\delta_E^2=W\cdot {\rm id}_E.$$

The morphisms between the $G$-equivariant matrix factorizations $\ovE:=(E, \delta_E)$ and $\ovF:=(F, \delta_F)$ are given by ${\rm Hom}_W(\ovE, \ovF)^G$, which is the $G$-equivariant part of the $\Z/2$-graded module of $R$-linear homomorphisms 
$${\rm Hom}_W(\ovE, \ovF):= {\rm Hom}_{\Z/2-Mod_R}(E,F)\oplus {\rm Hom}_{\Z/2-Mod_R}(E, F[1]).$$
We remark that the category possess a dg-structure given by a differential $d$ defined on $f\in {\rm Hom}_W(\ovE, \ovF)$ as 
$$df=\delta_F\circ f-(-1)^{|f|}f\circ\delta_E.$$

The homotopy category associated with the category ${\rm MF}(R, W)$ (or ${\rm MF}_{G}(R, W)$) will be denoted by ${\rm HMF}(R, W)$ (or ${\rm HMF}_{G}(R, W)$ respectively).
Both the categories ${\rm HMF}(R,W)$ and ${\rm HMF}_G(R,W)$ are triangulated categories with a natural translation functor $[1]$, where
$$\ovE[1]=(E^1 \xrightleftarrows[-\delta_0]{-\delta_1} E^0).$$
In \cite{Orl1, Orl}, the homotopy category ${\rm HMF}_{\<J\>}(R, W)$ is called the categories of $D$-branes of type $B$ in Landau-Ginzburg models.

\subsubsection{The singularity category}
We consider $R=\C[x_1, \ldots, x_n]$ as a finitely generated commutative graded algebra over the field $\C$. Consider the {\em hypersurface algebra} of $W$, which is a quotient graded algebra
$$S=R/(W)=\C[x_1, \ldots, x_n]/(W).$$

Let $\bD^b(S)$ be the bounded derived category of all complexes of $S$-modules with finitely generated cohomology
and $\bD^b_{\rm per}(S)$ be the full triangulated subcategory of $\bD^b(S)$ of perfect complexes
(i.e. finite complexes of finitely generated projective $S$-modules). 
The quotient
$$\bD^{\rm gr}_{\rm Sg}(S): =\bD^b(S)/\bD^b_{\rm per}(S)$$
is called the {\em stablized derived category} of $S$~\cite{Buc}, or {\em triangulated category of singularities}~\cite{Orl1}.
The category carries a natural triangulated structure. 

According to \cite{Eis, Buc, Orl1}, the functor {\rm Cok}: ${\rm MF}(R,W)\to {\rm gr-}S$ giving by ${\rm Cok}(\ovE)={\rm Coker} (\delta^{1})$ induces an exact functor $F: {\rm HMF}(R,W)\to \bD^{\rm gr}_{\rm Sg}(S)$ between the two triangulated categories, which completes the commutative diagram
\begin{equation}
\label{cokernel-diagram}
\begin{CD}
{\rm MF}(R,W)@>{\rm Cok}>>  {\rm gr-}S\\
@VV V @VV V\\
{\rm HMF}(R,W) @> F >> \bD^{\rm gr}_{\rm Sg}(S).
\end{CD}
\end{equation}
Here the vertical arrows are given by projections.
Since the algebra $R$ has a finite homological dimension, the functor $F$ is an equivalence ~\cite[Theorem 3.9]{Orl1}.

Let $q$ be the natural projection $q: \bD^b({\rm gr-}S)\to \bD^{\rm gr}_{\rm Sg}(S)$.
For each graded $S$-module $M$, we obtain a matrix factorization 
\begin{equation}
\label{stabilization-def}
M^{\rm st}:=F^{-1}q(M),
\end{equation}
which is called the {\em stablization of} $M$ in~\cite{Dyc}.

In general, there exists an equivalence $F^G$ between the triangulated category ${\rm HMF}_G(R,W)$ and the $G$-equivariant version of the singularity category, when $G$ is a finite group \cite[Theorem 7.3]{Qui}. The two equivalences $F$ and $F^G$ commutes with the natural forgetful functors.

\subsubsection{Koszul matrix factorizations}
Let ${\bf a}=(a_1, \ldots, a_n)$ and ${\bf b}=(b_1, \ldots, b_n)$ be two $n$-tuples of elements of $R$, such that
$$W={\bf a}\cdot{\bf b}=\sum_{j=1}^n a_j b_j.$$
The Koszul matrix factorization $\{{\bf a}, {\bf b}\}\in {\rm MF}(R, W)$ corresponding to the pair $({\bf a}, {\bf b})$ is isomorphic as a $\Z/2$-graded $R$-module to the Koszul complex 
\begin{equation}
\label{koszul-complex}
K_\bullet= \left(\land_R^{\bullet} (R^n), \quad \delta=\sum_{j=1}^{n}a_j e_j\wedge +\sum_{j=1}^{n} b_j\iota( e_j^*)\right).
\end{equation}
Here $\{e_j\}$ is the standard basis of $R^n$, $\{e_j^*\}$ is the dual basis, and $\iota( e_j^*)$ is the contraction of the element $e_j$.
If both elements $\sum_{j=1}^{n}a_j e_j$ and $\sum_{j=1}^{n} b_j\iota( e_j^*)$ are $G$-invariant,
we use the same symbol $\{{\bf a}, {\bf b}\}\in {\rm MF}_{G}(R, W)$ to denote the $G$-equivariant Koszul matrix factorization of the pair $({\bf a}, {\bf b})$.

\begin{example}
Since $W$ is quasihomogeneous, $W=\sum_{j=1}^{N}q_j\partial_j W \cdot x_j.$
The {\em Koszul matrix factorization} 
\begin{equation}
\label{stablization-pair}
\C^{\rm st}:= \left\{(q_1\partial_1 W, \ldots, q_N\partial_N W); (x_1, \ldots, x_N)\right\}
\end{equation}
is also called the {\em stablization of the residue field} $\C=R/(x_1, \ldots, x_N)$.
In fact, the cokernel of this matrix factorization is the residue field $\C=R/(x_1, \ldots, x_N)$.
Thus the notion $\C^{\rm st}=F^{-1}(q(\C))$ follows from the definition in~\eqref{stabilization-def}.

We write $M(j)$ for the twisted graded $S$-module with 
$$M(j)_i:=M_{j+i}.$$
In particular, let $\C(j)$ be the twist of the residue field $\C$ by $j$.
We will also be interested in objects 
$\C(j)^{\rm st}=F^{-1}(q(\C(j))).$
We remark that each $\C(j)^{\rm st}$ is $\<J\>$-equivariant.
So we will view $\C(j)^{\rm st}$ as an object in ${\rm HMF}_{\<J\>}(R,W)$.
\end{example}

\subsubsection{Hochshild homology and Chern character}
It is known that the equivariant Hochshild homology of the category ${\rm MF}_{G}(R, W)$ is isomorphic to the state space $\cH_{W, G}$ of the LG pair $(W, G)$ defined in~\eqref{state-space}~\cite[Theorem 2.5.4]{PV},
\begin{equation}
\label{hochshild-decomposition}
{\rm HH}_*({\rm MF}_{G}(R, W))\simeq \left(\bigoplus_{g\in G}{\rm Jac}(W_g)\cdot \omega_g\right)^G.
\end{equation}

For each $\ovE=(E, \delta_E)$, Polishchuk and Vaintrob constructed a $G$-equivariant {\em Chern character} 
$${\rm ch}_G(\ovE)\in {\rm HH}_0({\rm MF}_{G}(R, W))$$ given by a canonical map $\tau^{\ovE}: {\rm Hom}_G^*(\ovE, \ovE)\to {\rm HH}_*({\rm MF}_{G}(R, W))$ called {\em boundary-bulk map}~\cite[Proposition 1.2.4]{PV}, such that ${\rm ch}_G(\ovE)=\tau^{\ovE}({\rm id}_E)$. 
We can write 
\begin{equation}
\label{chern-character-component}
{\rm ch}_G(\ovE)=\sum_{g\in G}{\rm ch}_G(\ovE)_g,
\end{equation}
where each component ${\rm ch}_G(\ovE)_g\in \cH_g$ has an explicit super-trace formula expression as follows.
Let $x_{k+1}, \ldots, x_n$ be the $g$-invariant variables, $\partial_j \delta_E$ be the derivative of $\delta_E$ with respect to the variable $x_j$, and ${\rm str}_{R^g}(\cdot)$ be the supertrace of an operator on the quotient space $R^g:=R/(x_1, \ldots, x_k)$, then by~\cite[Theorem 3.3.3]{PV},
\begin{equation}
\label{chern-character-supertrace}
{\rm ch}_G(\ovE)_g={\rm str}_{R^g}\left([\partial_n\delta_E\circ \ldots \circ \partial_{k+1}\delta_E\circ g]\vert_{x_1=\ldots=x_k=0}\right)\, dx_{k+1}\wedge\ldots\wedge dx_n.
\end{equation}

The following result is a direct consequence of~\cite[Proposition 4.3.4]{PV}.
\begin{lemma}
\label{chern-residue-field}
Let $(W, \<J\>)$ be an admissible LG pair.
The $\<J\>$-equivariant Chern character of $\C(\ell)^{\rm st}$ is supported on $\cH_{\rm nar}$. More explicitly, we have
\begin{equation}
\label{stablization-chern}
{\rm ch}_{\<J\>}(\C(\ell)^{\rm st})=\sum_{m\in {\bf Nar}}\omega^{-\ell\cdot m}\prod_{j=1}^{N}(1-\omega^{w_j\cdot m})\ \sJ_m.
\end{equation}
\end{lemma}

\subsection{Gamma structures in FJRW theory}

Gamma structure has been explored in the study of integral structure in Gromov-Witten theory by \cite{KKP, Iri}.
In~\cite{CIR}, Chiodo, Iritani, and Ruan introduced Gamma structures in FJRW theory.
The Gamma structures allow them to discover the integral structures in FJRW theory as well as the connection between the Landau-Ginzburg/Calabi-Yau correspondence and the Orlov equivalence for Fermat type Calabi-Yau hypersurface singularities.
We apply the construction in~\cite{CIR} to admissible pairs $(W, \<J\>)$ of general type and reveal the connections between the category ${\rm MF}_{\<J\>}(R, W)$ and the asymptotic behavior of the corresponding FJRW theory discussed in Section~\ref{sec-asymptotic}.

\subsubsection{Gamma map and Gamma class}
Following~\cite{CIR}, we introduce an endormorphism on $\cH_{W,G}$ for any admissible LG pair $(W, G)$. This generalizes the construction for Fermat polynomials in~\cite{CIR}.
Cite~\cite{KRS}.
\begin{definition}
Let $(W, G)$ be an admissible LG pair.
We define a {\em Gamma map} $\widehat{\Gamma}$ on the state space $\cH_{W, G}$, given by
\begin{equation}\label{qst-gamma}
\widehat{\Gamma}:=\bigoplus_{g\in G}{(-1)^{-\mu(g)}} \prod_{j=1}^{N}\Gamma(1-\theta_j(g))\cdot {\rm Id}_{\cH_{g}}\in {\rm End}(\cH_{W,G}).
\end{equation}
\end{definition}

For each object $\ovE\in {\rm MF}_{G}(R, W)$, we call the cohomology class $\widehat{\Gamma}({\rm ch}_{G}(\ovE))$ the {\em Gamma class} of $\ovE$.
In particular, we define the {\em Gamma class} of the pair $(W, G)$ to be the Gamma class of $\C^{\rm st}$, and denote it by
\begin{equation}
\label{standard=gamma}
\widehat{\Gamma}(W, G):=\widehat{\Gamma}\left({\rm ch}_G(\C^{\rm st})\right).
\end{equation}




Applying Lemma \ref{chern-residue-field} to the LG pair $(W, \<J\>)$, we immediately have
\begin{proposition}
\label{theorem-chern-asymptotic}
Let $(W, \<J\>)$ be an admissible LG pair, not necessarily of general type.
For any $\ell\in\Z$, the  Gamma class of $\C(\ell)^{\rm st}$ is the class $\cA_{\ell}$ in~\eqref{asymptotic-class-definition}. Namely,
$$\widehat{\Gamma}\left({\rm Ch}_{\<J\>}(\C(\ell)^{\rm st})\right)=\cA_{\ell}.$$
\end{proposition}
\begin{proof}
Using the Gamma map~\eqref{qst-gamma} and the Chern character formula~\eqref{stablization-chern}, we have
\begin{eqnarray*}
&&\widehat{\Gamma}\left({\rm Ch}_{\<J\>}(\C(\ell)^{\rm st})\right)\\
&=&\widehat{\Gamma}\left(\sum_{m\in {\bf Nar}} \omega^{-\ell\cdot m}\prod_{j=1}^{N}(1-\omega^{w_j\cdot m})\, \sJ_m\right)\\
&=&\sum_{m\in {\bf Nar}}{(-1)^{-\mu(J^m)}}  \prod_{j=1}^{N}\Gamma\left(1-\left\{{w_j\cdot m\over d}\right\}\right) \omega^{-\ell\cdot m}
 \prod_{j=1}^{N}(1-\omega^{w_j\cdot m})\, \sJ_m\\
&=&(2\pi)^N\sum_{m\in {\bf Nar}} {\omega^{-\ell\cdot m}\over \prod\limits_{j=1}^{N}\Gamma(\left\{{w_j\cdot m\over d}\right\})}\, \sJ_m.
\end{eqnarray*}
The third equality uses the formula~\eqref{hodge-J^m} and the Euler reflection formula~\eqref{euler-reflection}.
\end{proof}

\begin{remark}
We remark that the Gamma map defined in \eqref{qst-gamma} differs from the Gamma map $\widehat{\Gamma}\in {\rm End}(\cH_{W,G})$ in~\cite[Definition 2.17]{CIR} by a factor of ${(-1)^{-\mu(g)}}$.  
The modification is caused by the identification of the Gamma classes and asymptotic classes in Proposition~\ref{theorem-chern-asymptotic}.
\end{remark}

\subsubsection{Gamma classes and results on Gamma Conjecture~\ref{algebraic-analytic}}
According to the results in Proposition~\ref{theorem-asymptotic-classes} and Corollary~\ref{corollary-weak-classes}, we obtain Theorem~\ref{corollary-chern-asymptotic} on Gamma Conjectures~\ref{algebraic-analytic}.

\subsection{Gamma structure and the Hirzebruch-Riemann-Roch formula}

\subsubsection{A canonical pairing on ${\rm HH}_*({\rm MF}_G(R, W))$}
In general, for a proper dg-category $\mathscr{C}$, there is a canonical pairing on the Hochshild homology~\cite{Shk}
$$\<\cdot, \cdot\>_{\mathscr{C}}: {\rm HH}_*(\mathscr{C}^{op})\otimes {\rm HH}_*(\mathscr{C})\to \C.$$ 
In~\cite[Theorem 4.2.1]{PV}, an explicit formula of this pairing  for the category ${\rm MF}_G(R, W)$ has been worked out in terms of the residue pairing in \eqref{residue-pairing}.
Using the identification 
$${\rm HH}_*({\rm MF}_G(R, -W))= {\rm HH}_*({\rm MF}_G(R, W)),$$
we will denote this canonical pairing by 
$$( \ , \ )^{\rm PV}: {\rm HH}_*({\rm MF}_G(R, W))\times {\rm HH}_*({\rm MF}_G(R, W))\to \C.$$
Recall that in~\eqref{hochshild-decomposition},  we have ${\rm HH}_*({\rm MF}_G(R, W))=\cH_{W, G}.$
For $A, B\in \cH_{W, G}$, the canonical pairing takes the form
\begin{equation}
\label{PV-pairing}
(A, B)^{\rm PV}={1\over |G|}\sum_{g\in G}{1\over \det\left[1-g; \C^N/\C^g\right]}\<A_{g^{-1}}, B_{g}\>.
\end{equation}
Here $\C^{g}\subset\C^N$ is the $g$-invariant subspace, $\C^N/\C^g$ is the quotient space, $A_g\in \cH_g$ is the projection of $A$ on $\cH_g$, and 
$\<\cdot, \cdot\>$ is the pairing in \eqref{fjrw=orbifold}.


\subsubsection{Euler pairing and Hirzebruch-Riemann-Roch formula}
For a pair of objects $\ovE$ and  $\ovF$ in  the category ${\rm MF}_G(R, W)$, the {\em Euler pairing} $\chi(\ovE, \ovF)$ is defined to be the Euler characteristic of the Hom-space 
\begin{equation}
\label{euler-pairing}
\chi(\ovE, \ovF):=\chi({\rm Hom}_W(\ovE, \ovF)^G)\in \mathbb{Z}.
\end{equation}

In~\cite[Theorem 4.2.1 (ii)]{PV}, Polishchuk and Vaintrob proved a {\em Hirzebruch-Riemann-Roch formula} for the category of $G$-equivariant matrix factorizations of $W$, which identifies the Euler pairing with the canonical pairing using the Chern character.
That is,
\begin{equation}
\label{PV-HRR-formula}
\chi(\ovE, \ovF)=({\rm ch}_G(\ovE), {\rm ch}_G(\ovF))^{\rm PV}.
\end{equation}
This generalized the earlier work of Walcher~\cite{Wal}, where the HRR formula for some particular cases was proven.

\subsubsection{Gamma class and an analog of Todd class in LG models}
By the isomorphisms in~\eqref{hochshild-decomposition}, we may consider the Gamma map in~\eqref{qst-gamma} as an isomorphism
$$\widehat{\Gamma}:{\rm HH}_*({\rm MF}_G(R, W))\to \cH_{W, G}, \quad A\mapsto \widehat{\Gamma}(A).$$
\begin{proposition}
\label{non-symmetric pairing-identity}
The non-symmetric pairing defined in~\eqref{non-symmetric-pairing-qst} and the canonical pairing~\eqref{PV-pairing} are compatible under the Gamma structure.
That is 
$$\left[ \widehat{\Gamma}(A), \widehat{\Gamma}(B)\right)=(A, B)^{\rm PV}.$$
\end{proposition}
\begin{proof}
By definition of the Gamma map~\eqref{qst-gamma}, we can rewrite 
$$\widehat{\Gamma}(A)=\sum_{g\in G}(-1)^{-\mu(g)} \prod_{\substack{1\leq j\leq N\\ \theta_j(g)\neq0}}\Gamma(1-\theta_j(g))A_g.$$
Recall $N_g$ is the dimension of $g$-fixed locus in $\C^N$.
Applying~\eqref{non-symmetric-pairing-qst}, we have 
\begin{align*}
\left[\widehat{\Gamma}(A),  \widehat{\Gamma}(B)\right)
=&{1\over |G|}\sum_{g\in G}{(-1)^{-\mu(g)}\over (-2\pi)^{N-N_g}}
\prod_{\substack{1\leq j\leq N\\ \theta_j(g)\neq0}}\Gamma(\theta_j(g))\Gamma(1-\theta_j(g))\<A_{g^{-1}}, B_{g}\>.\\
=&{1\over |G|}\sum_{g\in G}{(-1)^{-\mu(g)}\over (-2\pi)^{N-N_g}}\prod_{\substack{1\leq j\leq N\\ \theta_j(g)\neq0}}
{-2\pi\sqrt{-1}e^{\pi\sqrt{-1}\theta_j(g)} \over 1-e^{2\pi\sqrt{-1}\theta_j(g)}}
\<A_{g^{-1}}, B_{g}\>.\\
=&{1\over |G|}\sum_{g\in G}\prod_{\substack{1\leq j\leq N\\ \theta_j(g)\neq0}}
{\<A_{g^{-1}}, B_{g}\> \over 1-e^{2\pi\sqrt{-1}\theta_j(g)}}\\
=&(A, B)^{\rm PV}.
\end{align*}
Here the second equality follows from the Euler reflection formula~\eqref{euler-reflection}, and the third equality follows from the definition of Hodge grading number in~\eqref{hodge-grading-operator}.
\end{proof}
\begin{remark}
\begin{enumerate}
\item
This result in Proposition~\ref{non-symmetric pairing-identity} is similar to ~\cite[Theorem 4.6]{CIR}, where the weighted homogeneous polynomial $W$ is assumed to be Gorenstein. 
\item According to~\cite[Section 5]{Wal}, in LG models, the factor ${1\over \det\left[1-g; \C^N/\C^g\right]}$ in~\eqref{PV-pairing} can be viewed as an analog of Todd class of manifolds.
So we may view the Gamma map as a non-symmetric ``square root" of this Todd class analog.
\end{enumerate}
\end{remark}

As a consequence of~\eqref{PV-HRR-formula} and Proposition~\ref{non-symmetric pairing-identity}, we have
\begin{corollary}
\label{cor-gamma-pairing}
The Euler pairing in~\eqref{euler-pairing} is compatible with the non-symmetric pairing in~\eqref{non-symmetric-pairing-qst} via the Gamma map~\eqref{qst-gamma},
$$\chi(\ovE, \ovF)
=\left[ \widehat{\Gamma}({\rm ch}_G(\ovE)), \widehat{\Gamma}({\rm ch}_G(\ovF))\right).$$
\end{corollary}




\section{Orlov's SOD and  Stokes phenomenon}
\label{sec-orlov}

For quasi-homogeneous polynomials of general type, there is a remarkable connection between Orlov's work~\cite{Orl} on semiorthogonal decompositions of the triangulated category ${\rm HMF}_{\<J\>}(R,W)$ and the Stokes phenomenon that appears in the asymptotic expansions in the FJRW theory of $(W, \<J\>)$.
 This is similar to the case of Fano hypersurface, where semiorthogonal decompositions of the derived category of coherent sheaves of the Fano hypersurface \cite{KKP, GGI, SS} are related to the Stokes phenomenon that appears in the Gromov-Witten theory of the Fano hypersurface.


\subsection{Orlov's SOD for matrix factorizations}

\subsubsection{Semiorthogonal decompositions}
We collect a few definitions and facts on semiorthogonal decompositions of a triangulated category~\cite{BO, Orl}. 
Let $\cD$ be a triangulated category. Let $\cN\subset \cD$ be a full subcategory. The {\em right orthogonal to $\cN$} is a full subcategory  
$$\cN^\perp=\{M\in \cD\mid {\rm Hom}(N, M)=0, \forall N\in \cN.\}$$
The orthogonal $\cN^\perp$ is also a triangulated category.
The subcategory $\cN$ is called {\em right admissible} if there is a right adjoint functor for the embedding $\cN\hookrightarrow \cD.$
The left orthogonal and left admissible are defined analogously. 
The subcategory $\cN$ is called {\em admissible in $\cD$} if it is left admissible and right admissible.
A sequence of full triangulated categories $(\cN_1, \ldots, \cN_n)$ in $\cD$ is called a {\em semiorthogonal decomposition of the category $\cD$} (SOD for short) if 
\begin{itemize}
\item there is a sequence of left admissible subcategories 
$$\cD_1=\cN_1\subset \cD_2\subset \ldots \subset \cD_n=\cD$$
such that $\cN_p$ is left orthogonal to $\cD_{p-1}$ in $\cD_p$;
\item each $\cN_p$ is admissible in $\cD$;
\item the sequence $(\cN_1, \ldots, \cN_n)$ is {\em full} in $\cD$, i.e. the sequence generates the category $\cD$.
\end{itemize}
Such a semiorthogonal decomposition of the category $\cD$ will be denoted by 
$$\cD=\<\cN_1, \ldots, \cN_n\>.$$
Recall that an object $E$ is said to be {\em exceptional}, if for $p\in\Z$,
$$
{\rm Hom}(E, E[p])
=\begin{dcases}
0, & p\neq0;\\
\C, & p=0.
\end{dcases}
$$

\subsubsection{Orlov equivalence}



Consider the hypersurface 
$$\cX_W:=(W=0)\subset \bP^{N-1}(w_1, \ldots, w_N).$$
 If $N=1$, then $\cX_W$ is the empty set.
 If $N>1$, then the hypersurface has complex dimension $N-2$. 

In~\cite{Orl} Orlov established a beautiful connection between the triangulated category  ${\rm HMF}_{\<J\>}(R,W)$ and the derived category  ${\bf D}^b(\cX_W)$ of coherent sheaves of the projective variety $\cX_W$ using semiorthogonal decompositions. 
The details for quasihomogeneous polynomials have been developed in~\cite{BFK1, BFK2}.

For our purpose, we only consider the cases when the admissible LG pair $(W, \<J\>)$ is of general type. 
This implies that either $\cX_W$ is an empty set (if $N=1$) or it is a projective variety of general type (if $N>1$).
The Orlov equivalence can be summarized as follows:
\begin{theorem}
\cite[Theorem 40]{Orl}\cite[Theorem 3.5]{BFK2}
\label{orlov-equivalence}
Consider the admissible LG model $(W, \<J\>)$, where $W$ is of general type with index $\nu>0$, 
then the triangulated category ${\rm HMF}_{\<J\>}(R,W)$ admits a semiorthogonal decomposition
\begin{equation}
\label{orlov-sod-lg}
{\rm HMF}_{\<J\>}(R,W)\simeq\left< \C(\nu-1)^{\rm st}, \ldots, \C^{\rm st}, {\bf D}^b(\cX_W)\right>.
\end{equation}
\end{theorem}
Here the category ${\bf D}^b(\cX_W)$ is viewed as  the bounded derived category $\bD^b({\rm qgr}S)$ under equivalence~\cite[Theorem 28]{Orl}, 
where ${\rm qgr}S$ is the quotient of the abelian category of graded finitely generated $S$-modules by the subcategory of torsion modules. 
When $W$ is of general type, there is a fully faithful functor between the bounded derived category $\bD^b({\rm qgr}S)$ and the category\footnote{Here the symbol ${\gr}$ means graded, so the category is a $\<J\>$-equivariant version.} $\bD^{\rm gr}_{\rm Sg}(S)$~\cite[Theorem 16 (ii)]{Orl}.

The sequence $(\C(\nu+1)^{\rm st}, \ldots, \C^{\rm st})$ in the semiorthogonal decomposition \eqref{orlov-sod-lg} is an {\em exceptional collection}. That is,
each $\C(j)^{\rm st}$ is exceptional,
and the sequence satisfies the semiorthogonal condition that for all $p\in \Z$, when $i>j$,
$${\rm Hom}(\C(j)^{\rm st}, \C(i)^{\rm st}[p])=0.$$
In fact, for any $\ell\in\Z$, the sequence 
$
\big(\C(\ell)^{\rm st}, \C(\ell-1)^{\rm st}, \ldots, \C(\ell-\nu+1)^{\rm st}\big)
$
is an exceptional collection, and there is a semiorthogonal decomposition
$$
{\rm HMF}_{\<J\>}(R,W)\simeq\left< \C(\ell)^{\rm st}, \C(\ell-1)^{\rm st}, \ldots, \C(\ell-\nu+1)^{\rm st}, {\bf D}^b(\cX_W)\right>.
$$

\subsection{Combinatorics of the Gram matrix}
\label{sec-gram-matrix}
Let the admissible LG pair $(W, \<J\>)$ be of general type.
For any integer $\ell$, we consider the Gram matrix of the sequence $\big(\C(\ell)^{\rm st}, \C(\ell-1)^{\rm st}, \ldots, \C(\ell-\nu+1)^{\rm st}\big)$ by $M_\ell$.
It is a $\nu\times\nu$ square matrix, whose entries are given by the Euler pairing. 
We write
\begin{equation}
\label{sub-gram-matrix}
M_\ell:=\left(\chi\big(\C(\ell-j+1)^{\rm st}, \C(\ell-i+1)^{\rm st}\big)\right)_{\nu \times \nu}.
\end{equation}

\subsubsection{A polynomial of the weighted system}
Now we investigate the property of this matrix and its inverse (if exists).
First of all, we need some combinatorial preparation.
Let $\vec{w}=(w_1, \ldots, w_N)$ be the $N$-tuple of weights.
We define a set of coefficients 
$$\{a(n)\mid n\in\Z, 0\leq n<d\}$$ using the polynomial 
\begin{equation}
\label{weight-product-polynomial}
P^A(x):=\sum_{n=0}^{d-1} a(n) x^n:=\prod_{j=1}^{N}(1-x^{w_j})=1+\ldots+(-1)^N x^{\sum_{j=1}^{N}w_j}.
\end{equation}
By definition, we immediately have 
\begin{equation}
\label{property-coefficients-polynomial}
a(n)=
\begin{dcases}
1, & n=0;\\
(-1)^N,  & n= d-\nu;\\
0, & d-\nu< n< d.
\end{dcases}
\end{equation}
Using
\begin{equation*}
\begin{split}
\prod_{j=1}^{N}(1-x^{w_j})
&=\prod_{j=1}^{N}\left(-x^{w_j}(1-x^{-w_j})\right)\\
&=(-1)^Nx^{\sum_{j=1}^{N}w_j}\sum_{n=0}^{d-\nu} a(n) x^{-n}\\
&=\sum_{n=0}^{d-\nu} (-1)^Na(n) x^{d-\nu-n},
\end{split}
\end{equation*}
we obtain the following symmetry of the coefficients.
\begin{lemma}
\label{symmetry-coefficients}
The coefficients $a(n)$ of the degree $d-\nu$ polynomial $P^A(x)$ satisfies
$$a(d-\nu-n)=(-1)^Na(n), \quad \forall \quad  0\leq n\leq d-\nu.$$
\end{lemma}

On the other hand, we have the following combinatoric identities.
\begin{lemma}\label{lemma-magic-identity}
For each $n$ such that $0\leq n\leq d-1$, we have
\begin{equation}
\label{magic-identity}
{1\over d}\sum_{m\in {\bf Nar}}\omega^{n\cdot m}\prod_{j=1}^{N}(1-\omega^{-w_j\cdot m})=a(n).
\end{equation}
\end{lemma}
\begin{proof}
Consider a polynomial 
$$P(x)={1\over d}\sum_{m=0}^{d-1}\prod_{j=1}^{N}(1-\omega^{-m\cdot w_j})\sum_{n=0}^{d-1}(\omega^mx)^n-\prod_{j=1}^{N}(1-x^{w_j}).$$
For any $m$ such that $0\leq m\leq d-1$, if $m\notin {\bf Nar},$ then there exists some $j$ such that $1\leq j \leq N$ and $(1-\omega^{-m\cdot w_j})=0$. So we can rewrite
$$P(x)=\sum_{n=0}^{d-1}x^n \left({1\over d}\sum_{m\in {\bf Nar}}\omega^{n\cdot m}\prod_{j=1}^{N}(1-\omega^{-w_j\cdot m})-a(n)\right).$$
Using $\nu>0,$ we have 
\begin{equation}
\deg P(x)={\rm Max}\{d-1, \sum_{j=1}^N w_j\}=d-1.
\end{equation}
Now for each integer $k=0,1,\ldots, d-1$, using the following identities
\begin{equation}
\label{unity-identity}
\sum_{n=0}^{d-1}(\omega^mx)^n=
\begin{dcases}
d, & \text{if } x=\omega^{-m};\\
0, & \text{if }  x^d=1, x\neq\omega^{-m},
\end{dcases}
\end{equation}
we have 
$$P(\omega^{-k})={1\over d}\prod_{j=1}^{N}(1-\omega^{-k\cdot w_j})d-\prod_{j=1}^{N}(1-\omega^{-k\cdot w_j})=0.$$
Thus $P(x)$ have $d$ distinct roots and we must have $P(x)=0$. 
So the equality~\eqref{magic-identity} follows.
\end{proof}

Using the canonical pairing~\eqref{PV-pairing} and Lemma~\ref{lemma-magic-identity}, we have
\begin{proposition}
\label{left-upper-inverse-gram}
Let  $W$ be an invertible polynomial of general type. 
Consider the admissible LG pair $(W, \<J\>)$, the Gram matrix $M_\ell$ of the exceptional collection 
$\big(\C(\ell)^{\rm st}, \C(\ell-1)^{\rm st}, \ldots, \C(\ell-\nu+1)^{\rm st}\big)$ have the following properties:
\begin{enumerate}
\item It does not depend on $\ell$.
\item
It is upper-triangular and all the diagonal entries are $1$. 
\item  For each $1\leq j\leq \nu$ and $0\leq n\leq \nu-j$, the $(j, j+n)$-th entries of $M_\ell$ is $a(n)$.
\end{enumerate}
\end{proposition}
\begin{proof}
Using the HRR formula~\eqref{PV-HRR-formula}, the Chern character formula~\eqref{stablization-chern}, and the canonical pairing formula~\eqref{PV-pairing}, for any pair $(j, j+n)$ such that $1\leq j, j+n\leq \nu$, we can calculate the Euler pairing explicitly
\begin{equation*}
\begin{split}
&\chi\Big(\C(\ell-(j+n)+1)^{\rm st}, \C(\ell-j+1)^{\rm st}\Big)\\
=&\left({\rm ch}_{\<J\>}(\C(\ell-(j+n)+1)^{\rm st}), {\rm ch}_{\<J\>}(\C(\ell-j+1)^{\rm st})\right)^{\rm PV}\\
=&{1\over d}\sum_{m\in {\bf Nar}}\<{\prod_{j=1}^{N}(1-\omega^{w_j\cdot (d-m)})\over\omega^{(\ell-j-n+1)(d-m)}}\sJ_{d-m},{\prod_{j=1}^{N}(1-\omega^{w_j\cdot m})\over\omega^{(\ell-j+1)m}}\sJ_{m}\>\, \prod_{j=1}^{N}(1-\omega^{w_j\cdot m})^{-1}\\
=&{1\over d}\sum_{m\in {\bf Nar}}\omega^{n\cdot m}\prod_{j=1}^{N}(1-\omega^{-w_j\cdot m}).
\end{split}
\end{equation*}
The entries are independent on the choice of $\ell$.
According to the formula~\eqref{magic-identity}, the $(j,j+n)$-th entry of $M_\ell$ is given by 
$$a_{j, j+n}=
\begin{dcases}
a(d+n), & n<0;\\
a(n), & n\geq 0.
\end{dcases}
$$
Now the rest properties follow from the fact listed in~\eqref{property-coefficients-polynomial}.
\end{proof}

\subsubsection{A Laurent polynomial of the weighted system}
Let $L_{\vec{w}}(n)\in\mathbb{Z}_{\geq 0}$ be the number of partitions of $n$ using only the parts from $\vec{w}$. That is 
the number of nonnegative integer $N$-tuples $(k_1, k_2, \ldots, k_N)$ satisfying $$\sum_{j=1}^{N}k_jw_j=n.$$
We have a generating function
\begin{equation}
\label{weight-product-inverse-polynomial}
\sum_{n=0}^{\infty} L_{\vec{w}}(n) x^n=\prod_{j=1}^{N}(1-x^{w_j})^{-1}.
\end{equation}
It is straightforward to obtain
\begin{proposition}
\label{inverse-gram-formula}
The Gram matrix $M_\ell$ has an inverse $M_\ell^{-1}=(a^{i,j})$.
The inverse matrix $M_\ell^{-1}$ satisfies the following properties.
\begin{enumerate}
\item It does not depend on $\ell$.
\item It is upper-triangular and all the diagonal entries are $1$. 
\item For each $1\leq j\leq \nu$ and $0\leq n\leq \nu-j$, we have $a^{j, j+n}=L_{\vec{w}}(n)$.
\end{enumerate}
\end{proposition}

\subsection{Relation to Stokes phenomenon via Meijer G-functions}
\label{sec-asymptotic-luke}
In this subsection, we will show  in Theorem~\ref{theorem-gram-asymptotic} that the coefficients of the matrix $M_\ell^{-1}$ match certain coefficients in the asymptotic expansion of generalized hypergeometric functions.

\subsubsection{A partial fraction decomposition}


We now introduce a transformation law of Meijer $G$-functions using the method described in~\cite[Proposition 2]{Fie}.
Consider the partial fraction decomposition of the rational function
\begin{equation}
\label{pfd-rational}
{Q(y)\over P(y)}
=
{\prod\limits_{j=0}^{q}\left(y-e^{2\pi\sqrt{-1} (\rho_j-{1\over d})}\right)
\over
\prod\limits_{j=1}^{p}\left(y-e^{2\pi\sqrt{-1} (\alpha_j-{1\over d})}\right)}.
\end{equation}
Since the polynomial $y-e^{2\pi\sqrt{-1}(\alpha_j-{1\over d})}$ is irreducible in $\C[y]$, and only depends on the value $\alpha_j$ modulo $\Z$. 
We can rewrite the denominator of the rational function~\eqref{pfd-rational} as a product of powers of distinct irreducible polynomials by
$$P(y)=\prod\limits_{i=1}^{k}(y-p_i)^{n_i}, \quad p_i\in \{e^{2\pi\sqrt{-1}(\alpha_j-{1\over d})}\mid j=1, \ldots, p\},$$ 
where $n_i$ is the multiplicity of $p_i$ and $k$ is the number of distinct $p_i$'s.
A well known result in partial fraction decomposition says
\begin{proposition}
There are unique coefficients $$\{d_h\in \C\mid h=0, 1, \ldots, \nu\}\cup\{c_{i,j}\in \C\mid i=1, \ldots k; 1\leq j\leq n_i\}$$ such that the following fraction decomposition holds.
\begin{equation}
\label{classical-pfd}
{Q(y)\over P(y)}=\sum_{h=0}^{\nu}d_h y^h+\sum_{i=1}^{k}\sum_{r=1}^{n_i}{c_{i,r}\over (y-p_i)^{r}}.
\end{equation}
\end{proposition}
Recall that $L_{\vec{w}}(n)$ is defined in~\eqref{weight-product-inverse-polynomial}.
We have
\begin{corollary}
The partial fractional decomposition~\eqref{classical-pfd} can be written as 
\begin{equation}
\label{simple-D_h-formula}
{Q(y)\over P(y)}=\sum_{h=1}^{\nu}d_h y^h -(-1)^{N} +y\sum_{i=1}^{k}{b_i(y)\over (y-p_i)^{n_i}}
\end{equation}
where 
\begin{equation}
\label{stokes-exponential}
d_h=L_{\vec{w}}(\nu-h)=a^{1, \nu-h+1}, \quad \forall \quad h=1, \ldots, \nu,
\end{equation}
and $b_i(y)$ is a polynomial determined by 
\begin{equation}
\label{pfd-linear}
yb_i(y)=\sum_{r=1}^{n_i}c_{i,r}(y-p_i)^{n_i-r}\left(1-(1-{y\over p_i})^r\right)\in y\C[y].
\end{equation}
Or equivalently, 
\begin{equation}
\label{linear-term-algebraic}
b_i(y)=\sum_{r=1}^{n_i}c_{i,r}(y-p_i)^{n_i-r}\sum_{m=0}^{r-1}(-1)^m{r\choose m+1}p_i^{-m-1} y^m\in \C[y].
\end{equation}
\end{corollary}
\begin{proof}
Using~\eqref{difference-alpha-rho}, we obtain 
$${Q(0)\over P(0)}=-(-1)^{N}.$$
We evaluate the equation~\eqref{classical-pfd} at $y=0$ and obtain
$$-(-1)^N=d_0+\sum_{i=1}^{k}\sum_{r=1}^{n_i}{c_{i,r}\over (-p_i)^{r}}.$$
Using this equality, we obtain the formula~\eqref{pfd-linear} by taking the difference of~\eqref{classical-pfd} and~\eqref{simple-D_h-formula}.

Next we multiply the equation~\eqref{simple-D_h-formula} by $P^A(y)=\prod_{j=1}^{N}(1-y^{w_j})$. Using the definitions of $\rho_k$'s and $\alpha_j$'s, we obtain
\begin{equation}
\label{type-cd-stokes-coefficients}
(-1)^N(y^d-1)=P^A(y)\left(\sum_{h=1}^{\nu}d_h y^h-(-1)^N\right)+y\sum_{i=1}^{k}  b_i(y) P^A_i(y),
\end{equation}
where $P^A_i(y)$ is the polynomial given by $$P^A_i(y)={P^A(y)\over (y-p_i)^{n_i}}.$$
We have
$$\deg (y b_i(y) P^A_i(y))=\deg P^A(y)=d-\nu.$$
By Lemma~\ref{symmetry-coefficients}, we have 
$$P^A(y)=\sum_{n=0}^{d-\nu} (-1)^Na(n) y^{d-\nu-n}.$$
Now we compare the coefficients of $x^{d-j}$ in~\eqref{type-cd-stokes-coefficients} for $j=0, 1, \ldots, \nu-1$ and obtain $\nu$ equalities:
\begin{equation}
\label{recusion-d-type}
\sum_{i=0}^{j}  a(i) d_{\nu-j+i}=\delta_j^0, \quad j=0, 1,\ldots, \nu-1.
\end{equation}
Since $a(0)=1$, these $\nu$ equalities determines $\{d_h\mid h=\nu, \nu-1, \ldots, 1\}$ recursively.

On the other hand, the definitions of~\eqref{weight-product-polynomial} and~\eqref{weight-product-inverse-polynomial} imply that
\begin{align*}
1&=\left(\prod_{j=1}^N(1-y^{w_j})\right)\left(\prod_{j=1}^{N}(1-y^{w_j})^{-1}\right)\\
&=\left(\sum_{n=0}^{d-\nu}a(n)y^n\right)\left(\sum_{n=0}^{\infty}L_{\vec{w}}(n)y^n\right).
\end{align*}
So the sequence $\{L_{\vec{w}}(\nu-h)\mid h=\nu, \ldots, 1\}$ satisfies the same recursive formula~\eqref{recusion-d-type} as $\{d_h \mid h=\nu, \ldots, 1\}$.
Now the rest of the result follows from Proposition~\ref{inverse-gram-formula}.
\end{proof}

\subsubsection{A transformation law of Meijer $G$-functions and Stokes coefficients}

We see that the Meijer functions are multi-valued functions and the asymptotic expansions~\eqref{exponential-asymptotic} and~\eqref{asymptotic-algebraic} depend on the choices of sectors. 
The transformation of Meijer $G$-functions between different sectors has been well studied in the literature~\cite{Mei, Luk, Fie}.
Such transformation laws provide a powerful tool to describe the Stokes phenomenon of the ODE system, generalizing the seminal work of Stokes~\cite{Sto}.

Now we use the partial fractional decomposition formula~\eqref{simple-D_h-formula} to derive a transformation law of Meijer $G$-functions.
We set 
$$y:=e^{-2\pi\sqrt{-1}({1\over d}+s)}$$ 
and then multiply the equation~\eqref{simple-D_h-formula} by
$$
{1\over 2\pi\sqrt{-1}}{
\prod\limits_{j=0}^{q}\Gamma(1-\rho_j-s)
\over 
\prod\limits_{j=1}^{p}\Gamma(1-\alpha_j-s)
}\, x^s$$
and then integrate along the designated contour $L$ in~\eqref{meijer-def}.
The new equation will imply our transformation law.
\begin{proposition}
\label{transformation-exponential-type}
We have an identity among Meijer $G$-functions
\begin{align*}
&
(-1)^N G_{p,q+1}^{q+1,0}
\left( x \right)
-\sum_{h=1}^{\nu}d_h \omega^{-h}
G_{p,q+1}^{q+1,0}
\left(x\, e^{-2\pi\sqrt{-1}h}\right)\\
=&\sum_{i=1}^{k}
\sum_{r=1}^{n_i}c_{i,r}\sum_{m=0}^{r-1}{r\choose m+1}
{\omega^{r}e^{\pi \sqrt{-1}a_j (r+2m+2)}
\over (-1)^m(2\pi\sqrt{-1})^r}
G_{p,q+1}^{q+1,r}
\left(x\, e^{\pi\sqrt{-1}(r-2m-2)}|| 
a_j, \ldots, a_j\right).
\end{align*}
\end{proposition}
\begin{proof}
Firstly, we consider the LHS of the new equation. 
Applying Euler reflection formula~\eqref{euler-reflection} repeatedly for $x=\alpha_j+s$ and $x=\rho_j+s$, 
we obtain
\begin{equation*}
\begin{split}
{Q(y)\over P(y)}
&=(-2\pi\sqrt{-1}x)^{\nu}
e^{\pi\sqrt{-1}\left(\sum\limits_{j=0}^{q}(\rho_j+s)-\sum\limits_{j=1}^{p}(\alpha_j+s)\right)}
{\prod\limits_{j=1}^{p}\Gamma(\alpha_j+s)\Gamma(1-\alpha_j-s)
\over
\prod\limits_{j=0}^{q}\Gamma(\rho_j+s)\Gamma(1-\rho_j-s)
}.
\end{split}
\end{equation*}
So the LHS of the new equation becomes a scalar multiple of 
$$G_{p,q+1}^{0,p}
\left(
\begin{array}{l}
1-\alpha_1, \ldots, 1-\alpha_p\\
1-\rho_0,\ldots, 1-\rho_q
\end{array}
; x e^{-\pi\sqrt{-1}\nu}\right).$$
This Meijer $G$-function vanishes~\cite[Section 5.2 (8)]{Luk}, as it has no poles of $\Gamma(1-\rho_j-s)$, $j=1,\ldots, m$, lie to the right of the contour.

Secondly, the contribution from the term $y^k$ is given by
$$e^{-2\pi\sqrt{-1}{k\over d}}
G_{p,q+1}^{q+1,0}
\left(
\begin{array}{l}
a_1, \ldots, a_p\\
b_0, b_1,\ldots, b_q
\end{array}
; x\, e^{-2\pi\sqrt{-1}k}\right)
.$$

Finally, we consider the contribution from the term 
$${y^{m+1}\over (y-p_i)^r}, \quad 1\leq m+1\leq r\leq n_i\leq p.$$
Recall that $p_i=e^{2\pi\sqrt{-1}(\alpha_j-{1\over d})}$ for some $j$.
The integral have poles of $\Gamma(\alpha_j+s)$, with order $r$ at each $s=-\alpha_j-n$, $n=0, 1, \ldots,$
lie to the left of the contour. Applying Euler reflection formula~\eqref{euler-reflection} for $x=\alpha_j+s$, 
we obtain
$${1\over (y-p_i)\Gamma(1-\alpha_j-s)}
={\omega \over 2\pi\sqrt{-1}}\, e^{\pi\sqrt{-1}(a_j+s)} \Gamma(1-a_j+s).$$
Multiplying this formula repeatedly by $r$ times, we calculate
\begin{eqnarray*}
&&{1\over 2\pi\sqrt{-1}}\int_{L}
{y^{m+1}\over (y-p_i)^r}
{
\prod\limits_{j=0}^{q}\Gamma(1-\rho_j-s)
\over 
\prod\limits_{j=1}^{p}\Gamma(1-\alpha_j-s)
}\, x^s\, ds\\
&=&
{\omega^{r-m-1}\over (2\pi\sqrt{-1})^r}
e^{\pi \sqrt{-1} r a_j}
G_{p,q+1}^{q+1,r}
\left( x\, e^{\pi\sqrt{-1}(r-2m-2)} || a_j, \ldots, a_j\right).
\end{eqnarray*}
Now the result follows from the formula~\eqref{linear-term-algebraic}.
\end{proof}

\begin{remark}
If $\alpha_i-\alpha_j\notin \mathbb{Z}$ for any $i\neq j$, then $k=p$ and all the contribution in the last situation comes from the cases when $r=1$. 
The transformation law in Proposition~\ref{transformation-exponential-type} is just the formula (2.8) in \cite{Fie}.
\end{remark}

We call the coefficients $\{d_h\in\mathbb{Z}_{\geq 0}\}_{h=1}^{\nu}$ in the transformation law in Proposition~\ref{transformation-exponential-type} the {\em Stokes coefficients of exponential type}.
As a consequence of Proposition~\ref{inverse-gram-formula} and the formula~\eqref{stokes-exponential}, we have 
\begin{proposition}
\label{theorem-gram-asymptotic}

The Gram matrix $M_\ell$ of the exceptional collection 
$\big(\C(\ell)^{\rm st}, \C(\ell-1)^{\rm st}, \ldots, \C(\ell-\nu+1)^{\rm st}\big)$
 is determined by the Stokes coefficients via
\begin{equation}
\label{inverse-gram-stokes}
M_\ell^{-1}=
\begin{bmatrix}
1& d_{\nu-1}&  \cdots &d_2 & d_1\\
&1& \cdots&d_3&d_2\\
&&\ddots&\vdots&\vdots\\
&&&1&d_{\nu-1}\\
&&&&1
\end{bmatrix}_{\nu\times\nu}.
\end{equation}
\end{proposition}

\subsubsection{Stokes matrices for $r$-spin curves}

The Stokes coefficients in the transformation law~\eqref{transformation-exponential-type} can be used to calculate the {\em Stokes matrix}~\cite{Dub} of the underlying Dubrovin-Frobenius manifold.
If $N=1$ and $d=r$, this is the $r$-spin case $W=x^r$. 
The Stokes matrix was first calculated in~\cite{CV}.
The coefficients in the transformation law~\eqref{transformation-exponential-type} determine the Stokes matrix completely.
The calculation of Stokes matrix using the original definition in~\cite{Dub} is much involved.
Once we replace formula (16) in~\cite[Lemma 5]{Guz} by the transformation law~\eqref{transformation-exponential-type}, we can obtain the calculation for the $r$-spin case almost verbatimly as for the projective space $\mathbb{P}^{r-2}$ case in~\cite{Guz}.
Similar as \cite[Theorem 2]{Guz}, both the matrices $M_\ell$ and $M_{\ell}^{-1}$ can be realized as Stokes matrices, and they are equivalent to each other, under certain braid group action. 
Moreover, the symmetric matrix $M_\ell+M_{\ell}^T$ is the Cartan matrix of the $A_{r-1}$~\cite{CV}.

The calculation of Stokes matrices for general cases is much more complicated. 
We plan to discuss the general situations in the future.

\section{Evidence of quantum spectrum conjectures
}
\label{sec-evidence}

\subsection{Quantum spectrum of mirror simple singularities}
\label{sec-mirror-simple}
Now we prove quantum spectrum conjecture~\ref{quantum-spectrum-conj} holds true if the polynomial $W$ is one of the forms below:
\begin{enumerate}
\item $A_{n-1}$ singularity $W=x_1^{n}$, $n\geq 2.$
\item $D^T_{n+1}$-singularity $W=x_1^nx_2+x_2^2$, $n\geq 3.$
\item $E$-type simple singularities: $W=x_1^3+x_2^4,$ or $x_1^3+x_2^5$, or $x_1^3x_2+x_2^3$.
\end{enumerate}
Here the transpose mirror polynomial $W^T$ is a simple singularity of $ADE$-type. 
We will calculate the quantum multiplication $\tau'\star_\tau$ explicitly in each case.

\begin{remark}
For all the cases above, quantum spectrum conjecture~\ref{quantum-spectrum-conj} and quantum spectrum conjecture~\ref{quantum-spectrum-conj-invariant} are equivalent because $\<J\>=G_W$ always hold. 
By Proposition~\ref{mirror-theorem-fermat}, the mirror conjecture~\ref{mirror-conjecture} always hold.
\end{remark}

Recall that $\tau(t)$ is defined by the $I$-function via~\eqref{def-tau}. 
We denote
\begin{equation}
\label{power-of-x}
(\tau'(t))^{j}:=\underbrace{\tau'(t)\star_{\tau(t)} \tau'(t)\star_{\tau(t)} \ldots \star_{\tau(t)} \tau'(t)}_{j\text{-copies}}.
\end{equation}

\subsubsection{Quantum spectrum of $A_{n-1}$-singularities}
We start with $n$-spin singularity $$W=x_1^n, \quad n\geq 2.$$
The weight system is $(d; w_1)=(n;1)$, the index is $\nu=n-1$.
The state space $\cH_{W, \<J\>}$ has a basis $\{\sJ_1, \sJ_2, \ldots, \sJ_{n-1}\}$, with the dual basis $\{\sJ_{n-1}, \sJ_{n-2}, \ldots, \sJ_{1}\}.$
There are two situations.

If $n=2$, we have $\tau={\sJ_1\over 4}$, $\tau'={\sJ_1\over 2},$ and
$${\nu\over n}\tau'\star_{\tau}\sJ_1={1\over 4}
\sJ_1.$$
So Conjecture~\ref{quantum-spectrum-conj} holds because the quantum spectrum of ${\nu\over d}\tau'\star_{\tau}$ is $\{{1\over 4}\}$.

If $n\geq 3$, we have $\tau=\tau'=\sJ_2.$ Let us calculate $\sJ_2\star_{\sJ_2}.$
\begin{lemma}
For each $\sJ_i\in \cH_{W, \<J\>}$, we have quantum multiplication
\begin{equation*}
\begin{split}
\sJ_2\star_{\sJ_2}\sJ_i&=\sum_{j=1}^{n-1}\sum_{k=0}^{\infty}{1\over k!}\LD\sJ_2, \sJ_i, \sJ_j, \sJ_2, \ldots, \sJ_2\RD_{0, k+3}\sJ_{n-j}
=\begin{dcases}
\sJ_{i+1}, & \text{ if } i<n-1;\\
{1\over n} \sJ_1, & \text{ if } i=n-1.
\end{dcases}
\end{split}
\end{equation*}
\end{lemma}
\begin{proof}
By the degree constraint~\eqref{virtual-degree}, if $\LD\sJ_2, \sJ_i, \sJ_j, \sJ_2, \ldots, \sJ_2\RD_{0, k+3}\neq0$, we will have $(n-1)k=i+j+1-n\leq n-1$. So if $k=0$, then $n-j=i+1$; if $k=1$, then $i=j=n-1$. 
Now the result follows from computation of the following FJRW invariants
$$
\begin{dcases}
\LD\sJ_2, \sJ_i, \sJ_{n-1-i}\RD_{0,3}=1; \\
\LD\sJ_2, \sJ_2, \sJ_{n-1}, \sJ_{n-1}\RD_{0,4}={1\over n}.
\end{dcases}
$$
Both invariants are obtained by the {\em concavity axiom} in~\cite[Theorem 4.1.8]{FJR}.
\end{proof}
We summarize the result as below.
\begin{proposition}
For the LG pair $(W=x_1^n, \<J\>)$ with $n\geq 3$, 
we have
\begin{enumerate}
\item
The set $\{\sJ_1, \tau', \ldots, (\tau')^{n-2}\}$ is a basis of $\cH_{W, \<J\>}$.
\item There is a quantum relation
\begin{equation}\label{A-relation}
 (\tau')^{n-1}={1\over n}\sJ_1.
 \end{equation}
\item
Conjecture~\ref{quantum-spectrum-conj} holds because the quantum spectrum of ${\nu\over d}\tau'\star_{\tau}$ is
$$\fE=\left\{{n-1\over n}\left({1\over n}\right)^{1\over n-1} e^{2\pi\sqrt{-1}j\over n-1}\mid  j=0, 1, \ldots,  n-2\right\}.$$
\end{enumerate}
\end{proposition}

\subsubsection{Quantum spectrum of $D^T_{n+1}$-singularities}
Now we consider the $D^T_{n+1}$-singularity $$W=x_1^nx_2+x_2^2, \quad n\geq 3.$$
The weight system is 
$(d; w_1, w_2)=(2n; 1, n)$. 
The index is $\nu=n-1\geq 2$. 
We have a basis of $\cH_{W, \<J\>}$:
$$\left\{\sJ_{2k-1}, \sJ_0:=x_1^{n-1}dx_1 dx_2\vert J^0\> \mid 1\leq k\leq n\right\}.$$
We have
$$\tau(t)={t^2\over 4}\sJ_3, \quad \tau=\tau(1)={1\over 4}\sJ_3, \quad \tau'=\tau'(1)={1\over 2}\sJ_3.$$
Using the degree constraint~\eqref{virtual-degree}, the nontrivial contribution to $\sJ_3\star_{\sJ_3}$ are obtained from the following genus zero FJRW invariants
$$
\begin{dcases}
\<\sJ_3, \sJ_{2k-1},\sJ_{2n-2k-1}\>_{0,3}=1, &\text{ if } 1\leq k\leq n-1;\\
\<\sJ_3, \sJ_3, \sJ_{2n-3},\sJ_{2n-1}\>_{0,4}={1\over 2n}. &
\end{dcases}$$
Here the values of the FJRW invariants can be found in~\cite[Section 6.3.7]{FJR}.
So we see that the multiplication $\sJ_3\star_{\sJ_3}$ on the basis 
$$\{\sJ_1, \sJ_3, \ldots, \sJ_{2n-1},  \sJ_0\}$$ is given by the matrix
$$\begin{bmatrix}
 0 & 0 & \cdots & {1\over 2n}& 0 & 0 \\
 1 & 0  & \cdots & 0 & {1\over 2n} & 0 \\
 0 & 1 & \cdots & 0 & 0 & 0 \\
 \vdots &\vdots &\ddots & 0 & 0 & 0 \\
 0 & 0  & 0 & 1 & 0 & 0 \\
 0 & 0  & 0 & 0 & 0 & 0 
\end{bmatrix}.
$$
So we have quantum relations
$$(\tau')^k=
\begin{dcases}
2^{2-k}\sJ_{2k+1}, & \text{ if } 1\leq k\leq n-2;\\
2^{1-n}\left(\sJ_{2n-1}+{1\over 8n}\sJ_1\right), & \text{ if } k=n-1;\\
2^{-2-n}{1\over n}\sJ_{3}, & \text{ if } k=n.
\end{dcases}
$$
Now these calculations implies
\begin{proposition}
For the LG pair $(W=x_1^nx_2+x_2^2, \<J\>)$ with $n\geq 3$, 
we have
\begin{enumerate}
\item
The set $\{\sJ_1, \tau', \ldots, (\tau')^{n-1}, \sJ_0\}$ is a basis of $\cH_{W, \<J\>}$.
\item There are quantum relations
\begin{equation}\label{mirror-D-relation}
\begin{dcases}
 (\tau')^n={1\over 2^{n+1} n}\tau'; \\
\tau'\star_{\tau} \sJ_0=0.
\end{dcases}
\end{equation}
\item
Conjecture~\ref{quantum-spectrum-conj} holds because the quantum spectrum of ${\nu\over d}\tau'\star_{\tau}$ is
$$\fE=\left\{0^2, {n-1\over 2n}\left({n^n\over (2n)^{2n-(n-1)}}\right)^{1\over n-1}e^{2\pi\sqrt{-1}j\over n-1}\mid  j=0, 1, \ldots,  n-2\right\}.$$
\end{enumerate}
\end{proposition}

\subsubsection{Quantum spectrum of $E_6$-singularity}

For the $E_6$-singularity  $W=x_1^3+x_2^4$,
the weight system of $(W, \<J\>)$ is $(d; w_1, w_2)=(12; 4, 3)$ and the index is $\nu=5$. 
The set 
$$\{\sJ_1, \sJ_2, \sJ_5, \sJ_7, \sJ_{10}, \sJ_{11}\}$$
is a basis of the state space $\cH_{W, \<J\>}$. We have
$$\tau(t)=t\sJ_2, \quad \tau=\tau'=\sJ_2.$$
Using the degree constraint~\eqref{virtual-degree}, the nonzero FJRW invariants involved in the quantum multiplication ${\nu\over d}\tau'\star_{\tau}$ are listed below:
$$\begin{dcases}
\LD\sJ_2, \sJ_1, \sJ_{10}\RD_{0,3}=1; \\
\LD\sJ_2, \sJ_2, \sJ_{5}, \sJ_{5}\RD_{0,4}=a={1\over 3};\\
\LD\sJ_2, \sJ_2, \sJ_{2}, \sJ_{2}, \sJ_{7}\RD_{0,5}=b={1\over 6}; \quad
\LD\sJ_2, \sJ_2, \sJ_{2}, \sJ_{10}, \sJ_{11}\RD_{0,5}=c={1\over 12}.
\end{dcases}
$$
These invariants are calculated in~\cite[Section 6]{FJR}. 
Using these invariants, we obtain the relations
$$\begin{dcases}
(\tau')^2={b\over 2}\sJ_5, \quad (\tau')^3={ab\over 2}\sJ_7, \quad 
(\tau')^4={ab^2\over 4}\sJ_2, \\
 (\tau')^5={ab^2\over 4}\sJ_{11}+{ab^2c\over 8}\sJ_1, \quad (\tau')^6={ab^2c\over 4}\tau'.
\end{dcases}
$$

\begin{proposition}
For $E_6$-singularity $W=x_1^3+x_2^4$, we have:
\begin{enumerate}
\item The set $\{\sJ_1, \tau', (\tau')^2, (\tau')^3, (\tau')^4, (\tau')^5\}$ is a basis of  $\cH_{W, \<J\>}$;
\item There is a quantum relation 
$$(\tau')^6
={1\over 5184}\tau'={4^43^3\over 12^{12-5}}\tau'.$$
\item 
Conjecture~\ref{quantum-spectrum-conj} holds because the quantum spectrum of ${\nu\over d}\tau'\star_{\tau}$ is
$$\fE=\left\{0, {5\over 12}\left({4^43^3\over 12^{12-5}}\right)^{1\over 5} e^{2\pi\sqrt{-1}j\over 5}\mid j=0, 1, 2, 3, 4\right\}.$$
\end{enumerate}
\end{proposition}

\subsubsection{Quantum spectrum of $E_7$-singularity}
\label{sec-E7}
For the $E_7$-singularity $W=x_1^3x_2+x_2^3,$
the weight system of $(W, \<J\>)$ is $(d; w_1, w_2)=(9; 2, 3)$ and the index is $\nu=4$. 
The set
$$
\{\sJ_1, \sJ_2, \sJ_4, \sJ_5, \sJ_7, \sJ_8, \sJ_0\}
$$
is a basis of the state space $\cH_{W, \<J\>}$. We have
$$\tau(t)=t\sJ_2, \quad \tau=\tau'=\sJ_2.$$
Using WDVV equation and the results in~\cite{FJR},
the nonzero FJRW invariants involved in the quantum multiplication ${\nu\over d}\tau'\star_{\tau}$ can be calculated:
$$\begin{dcases}
\LD\sJ_2, \sJ_1, \sJ_{7}\RD_{0,3}=1, \\
\LD\sJ_2, \sJ_2, \sJ_{2}, \sJ_{5}\RD_{0,4}
={1\over 3}, \\
\LD\sJ_2, \sJ_2, \sJ_{2}, \sJ_{2}, \sJ_{4}\RD_{0,5}
={2\over 27}, \quad
\LD\sJ_2, \sJ_2, \sJ_{2}, \sJ_{7}, \sJ_{8}\RD_{0,5}
={2\over 27}.
\end{dcases}
$$
We remark that the vanishing of $\LD\sJ_2, \sJ_2, \sJ_{7}, \sJ_{0}\RD_{0,4}$ can not be achieved using the degree constraint~\eqref{virtual-degree}.
The matrix of the quantum multiplication $\tau'\star_{\tau}$ with respect to  the basis 
$$\{\sJ_1, \sJ_2, \sJ_4, \sJ_5, \sJ_7, \sJ_8, \sJ_0\}$$ 
is given by
$$
\begin{bmatrix}
 0 & 0 & 0 & 0 & \frac{1}{27} & 0 & 0 \\
 1 & 0 & 0 & 0 & 0 & \frac{1}{27} & 0 \\
 0 & \frac{1}{3} & 0 & 0 & 0 & 0 & 0 \\
 0 & \frac{1}{27} & 0 & 0 & 0 & 0 & 0 \\
 0 & 0 & \frac{1}{27} & \frac{1}{3} & 0 & 0 & 0 \\
 0 & 0 & 0 & 0 & 1 & 0 & 0 \\
 0 & 0 & 0 & 0 & 0 & 0 & 0 \\
\end{bmatrix}.
$$
\begin{proposition}
For $E_7$-singularity $W=x_1^3x_2+x_2^3$, we have:
\begin{enumerate}
\item
The state space $\cH_{W, \<J\>}$ has a basis
\begin{equation}
\label{E7-basis}
\{\sJ_1, \tau', (\tau')^2, (\tau')^3, (\tau')^4-{4\over 2187}\sJ_1, 9\sJ_4-\sJ_5, \sJ_0\}.
\end{equation}
\item There are quantum relations 
\begin{equation*}\label{E7-relation}
\begin{dcases}
(\tau')^5={4\over 2187}\tau'={2^2 3^3\over 9^{9-4}}\tau'; \\
\tau'\star_{\tau} \sJ_0=0.
\end{dcases}
\end{equation*}
\item 
Conjecture~\ref{quantum-spectrum-conj} holds because the quantum spectrum of ${\nu\over d}\tau'\star_{\tau}$ is
$$\fE=\left\{0^3, {4\over 9}\left({2^2 3^3\over 9^{9-4}}\right)^{1\over 4} e^{2\pi\sqrt{-1}j\over 4}\mid j=0, 1, 2, 3\right\}.$$
\end{enumerate}
\end{proposition}

\begin{remark}\label{remark-E7}
In this case, we see that 
the set $\{\one, \tau', (\tau')^2, (\tau')^3, (\tau')^4, (\tau')^5\}$ is linearly dependent and the last three elements in~\eqref{E7-basis} span the kernel of ${\nu\over d}\tau'\star_{\tau}$.
\end{remark}

\subsubsection{Quantum spectrum of $E_8$-singularity}

For the $E_8$-singularity  $W=x_1^3+x_2^5$, the weight system of $(W, \<J\>)$ is $(d; w_1, w_2)=(15; 5, 3)$ and the index is $\nu=5$. 
The set
$$\{\sJ_1, \sJ_2, \sJ_4, \sJ_7, \sJ_{8}, \sJ_{11}, \sJ_{13}, \sJ_{14}\}$$
 is a basis of the state space $\cH_{W, \<J\>}$. 
Again, we have
$$ \tau(t)=t\sJ_2, \quad \tau=\tau'=\sJ_2.$$
 The nonzero FJRW invariants involved in the quantum multiplication ${\nu\over d}\tau'\star_{\tau}$ are 
$$\begin{dcases}
\LD\sJ_2, \sJ_1, \sJ_{13}\RD_{0,3}=
\LD\sJ_2, \sJ_7, \sJ_7\RD_{0,3}=1, \\
\LD\sJ_2, \sJ_2, \sJ_{2}, \sJ_{11}\RD_{0,4}=a={1\over 3}, \\
\LD\sJ_2, \sJ_2, \sJ_{2}, \sJ_{4}, \sJ_{8}\RD_{0,5}=b={2\over 15}, \quad
\LD\sJ_2, \sJ_2, \sJ_{2}, \sJ_{13}, \sJ_{14}\RD_{0,5}=c={1\over 15}. 
\end{dcases}
$$
Using these invariants, we obtain quantum relations
$$
\begin{dcases}
 (\tau')^2=a\sJ_4, \quad (\tau')^3={ab\over 2}\sJ_7,\quad 
(\tau')^4={ab\over 2}\sJ_8, \quad   (\tau')^5={ab^2\over 4}\sJ_{11}, \\
 (\tau')^6={a^2b^2\over 4}\sJ_{13}, \quad  (\tau')^7={a^2b^2\over 4}\sJ_{14}+{a^2b^2c\over 8}\sJ_1, \quad  (\tau')^8={a^2b^2c\over 4}\sJ_2.
\end{dcases}
$$

\begin{proposition}
For $E_8$-singularity $W=x_1^3+x_2^5$, we have:
\begin{enumerate}
\item The set $\{\sJ_1, \tau', (\tau')^2, (\tau')^3, (\tau')^4, (\tau')^5,  (\tau')^6, (\tau')^7\}$ is a basis of  $\cH_{W, \<J\>}$;
\item There is a quantum relation 
$$(\tau')^8
={1\over 30375}\tau'={5^5 3^3\over 15^{15-7}}\tau'.$$
\item 
Conjecture~\ref{quantum-spectrum-conj} holds because the quantum spectrum of ${\nu\over d}\tau'\star_{\tau}$ is
$$\fE=\left\{0, {7\over 15}\left({5^5 3^3\over 15^{15-7}}\right)^{1\over 7} e^{2\pi\sqrt{-1}j\over 7}\mid j=0, 1, 2, 3, 4, 5, 6\right\}.$$
\end{enumerate}
\end{proposition}

\subsection{Quantum spectrum of Fermat homogeneous polynomials}
\label{sec-spectrum-fermat}
Now we prove quantum spectrum conjecture~\ref{quantum-spectrum-conj}  for any Fermat homogeneous polynomial $W=\sum_{i=1}^N x_i^d$ with $\nu=d-N>1$.
The proof is similar to the proof of the Gamma conjecture $\mathcal{O}$ for Fano hypersurfaces in $\mathbb{P}^{N-1}$~\cite{GGI, Ke}.
We need some preparation. 

\subsubsection{A quantum relation}
We first recall a well-known property of the quantum product. 
\begin{lemma}\label{lem-derivative-multiplication}
For any admissible LG pair $(W, G)$, let $t_i$ be the coordinate of $\phi_i\in \cH_{W, G}$. 
We have 
  \[
    \frac{\partial}{\partial t_i}S^{-1}(\bt, z)
    =\frac{1}{z}S^{-1}(\bt, z)(\phi_i\star_{\bt}).
  \]
  \end{lemma}

Using this property, we obtain a quantum relation under a mirror symmetry assumption.

\begin{proposition}
\label{mirror-quantum-relation}
If mirror conjecture~\ref{conjecture-i-function-formula} is true, the following relation of the quantum product $\tau'(t)\star_{\tau(t)}$ holds:
\begin{equation}
\label{eq-quantum-relation}
\left(\prod_{j=1}^Nw_j^{w_j}\right)\left({t\over d}\right)^{d-\nu}(\tau'(t))^{p}=(\tau'(t))^{q+1}.
\end{equation}
\end{proposition}
\begin{proof}
Using~\eqref{tau-component} and Lemma~\ref{lem-derivative-multiplication}, we have 
\begin{equation}
\label{s-derivative-tau}
\frac{\partial}{\partial t}zS^{-1}(\tau(t), z)=S^{-1}(\tau(t), z)(\tau'(t)\star_{\tau(t)}).
\end{equation}

If mirror conjecture~\ref{conjecture-i-function-formula} holds, we can apply the equation~\eqref{s-derivative-tau} repeatedly to obtain
\begin{align*}
&t^{d}z^p\prod_{j=1}^N{w_j^{w_j}\over d^{w_j}}\prod_{i=1}^{p}\left(t{\partial\over\partial t}+\alpha_i d-1\right)
 I_{\rm FJRW}^{\rm sm}(t,z)\\
=&t^{d}z^p\prod_{j=1}^N{w_j^{w_j}\over d^{w_j}}\prod_{i=1}^{p}\left(t{\partial\over\partial t}+\alpha_i d-1\right)\big(tz S^{-1}(\tau(t),z)\one\big)\\
=&t^{d}z^{p-1}\prod_{j=1}^N{w_j^{w_j}\over d^{w_j}}\prod_{i=2}^{p}\left(t{\partial\over\partial t}+\alpha_i d-1\right) t S^{-1}(\tau(t),z)  \big(t\tau'(t)    +O(z)\big)\\
 &\dots\\
  =&t^{d}\prod_{j=1}^N{w_j^{w_j}\over d^{w_j}} S^{-1}(\tau(t),z)\big(t^{p+1} \big(\underbrace{\tau'(t)\star_{\tau(t)}  \ldots \star_{\tau(t)} \tau'(t)}_{p\text{-copies}}\big)+O(z)\big).
\end{align*}
Similarly, we have
\begin{align*}
&z^{q+1}\prod_{j=0}^{q}\left(t{\partial \over \partial t}+(\rho_j-1)d-1\right)
 I_{\rm FJRW}^{\rm sm}(t,z) \\
 =& S^{-1}(\tau(t),z)\big(t^{q+2} \big(\underbrace{\tau'(t)\star_{\tau(t)}  \ldots \star_{\tau(t)} \tau'(t)}_{q+1\text{-copies}}\big)+O(z)\big).
\end{align*}
Applying the differential equation~\eqref{i-function-ode} and take the leading term, we have 
$$\left(\prod_{j=1}^N{w_j^{w_j}\over d^{w_j}}\right) t^{d+p+1}(\tau'(t))^{p}=t^{q+2}(\tau'(t))^{q+1}.$$
This implies the quantum relation~\eqref{eq-quantum-relation}.
\end{proof}

The quantum relation~\eqref{eq-quantum-relation} can be used to calculate the eigenvalues of ${\nu\over d}\tau'\star_{\tau}$. 
\begin{corollary}
\label{cor-linearlity}
For an admissible LG pair $(W, \<J\>)$, if 
\begin{enumerate}
\item mirror conjecture~\ref{conjecture-i-function-formula} holds for the LG pair $(W, \<J\>)$, and 
\item the set $\{1, \tau', \ldots, (\tau')^{q}\}$ is linearly independent;
\end{enumerate}
then the set 
$$\{0^p, Te^{2\pi\sqrt{-1}j\over \nu}\mid j=0, 1, \ldots, \nu-1 \}$$ is a subset of the set of eigenvalues of ${\nu\over d}\tau'\star_{\tau}$ on $\cH_{W, \<J\>}$.
Here the notation $0^p$ means the multiplicity of the eigvenvalue $0$ is $p$, and $T$ is the positive real number defined in~\eqref{largest-positive-number}. 
\end{corollary}

\begin{remark}
The set $\{1, \tau', \ldots, (\tau')^{q}\}$ maybe linearly dependent even if mirror conjecture~\ref{conjecture-i-function-formula} holds. 
In particular, $(W=x_1^3x_2+x_2^3, \<J\>)$ is such an example.
See Remark~\ref{remark-E7}.
\end{remark}

The following result is a consequence of Proposition~\ref{mirror-quantum-relation}.
\begin{proposition}\label{lem-quantum-relation}
For the Fermat pair 
$
(W=x_1^d+\ldots +x_N^d, \<J\>)
$ of general type, we have a quantum relation
$$(\tau')^{d-1}=d^{-N}(\tau')^{N-1}.$$
Here $\tau(t)$ is given by the formulas in~\eqref{tau-fermat}.
\end{proposition}

\subsubsection{Bases of $\cH_{W, \<J\>}^{G_W}$}
According to Proposition~\ref{fermat-broad}, the subspace $\cH_{W, \<J\>}^{G_W}$ has a basis
$$\{\sJ_{m}\mid m=1, 2, \ldots, d-1\}.$$ 
We consider two other (possibly different) bases.
\begin{proposition}
\label{basis-fermat-subspace}
Let $W=x_1^d+\ldots +x_N^d$ be a Fermat polynomial of general type, 
then
$$\{\sJ_2^{m-1}\mid m=1, 2,  \ldots, d-1\}$$
is a basis of  the subspace $\cH_{W, \<J\>}^{G_W}$.
\end{proposition}
\begin{proof}
We now prove that the change-of-basis matrix between the ordered sets
$$\{\sJ_{m}\mid m=1, 2, \ldots, d-1\}\ \quad \text{and}\quad \{\sJ_2^{m-1}\mid m=1, 2, \ldots, d-1\}$$ is upper triangular and the result follows. 

It is easy to see that $\sJ_j$ is dual to $\sJ_{d-j}$, $\<\sJ_2, \sJ_{m-1}, \sJ_{d-m}\>_{0,3}=1$, and
$$\sJ_2\star_{t\sJ_2}\sJ_{m}=\sJ_{m+1}+\sum_{j=1}^{d-1}\sJ_{d-j} \sum_{n\geq 1}{t^n\over n!}\<\sJ_2, \sJ_{m}, \sJ_j, \sJ_2, \ldots, \sJ_2\>_{0,n+3}.$$
If the coefficient of $t^n$ is nonzero, then by the homogeneity property~\eqref{virtual-degree} and the selection rule ~\eqref{selection-rule}, we must have 
$$
\begin{dcases}
{(d-2)N\over d}+n={N(n+1)\over d}+{N(m-1)\over d}+{N(j-1)\over d}, \\
-{n+1\over d}-{m+j\over d}\equiv 0 \mod \mathbb{Z}.
\end{dcases}
$$
The second formula says $m+j+n+1=kd$ for some $k\in\mathbb{Z}_+$. 
This implies the RHS of the first formula becomes ${N(kd-2)\over d}$ and thus we obtain $nd=N(k-1)d.$
This implies 
$$d-m-j-1=d-kd+n=n-{nd\over N}<0.$$
So we must have $d-j<m+1$. 
Thus the change-of-basis matrix is upper triangular. 
\end{proof}

\begin{proposition}
\label{basis-fermat}
Let $W=x_1^d+\ldots +x_N^d$ be a Fermat polynomial of general type, 
then
$$\{(\tau')^{m-1}\mid m=1, 2,  \ldots, d-1\}$$
is a basis of  the subspace $\cH_{W, \<J\>}^{G_W}$.
\end{proposition}
\begin{proof}
It suffices to prove the statement for the cases when $d=N+1>2$. We have 
\begin{eqnarray*}
\alpha\star_{\tau(t)}\beta&=&\sum_{i=1}^{r}\sum_{k=0}^{\infty}{\phi^i\over k!}\LD\alpha,\beta,\phi_i, t \sJ_2+\frac{t^{d}}{d!d^N}\sJ_1, \ldots, t\sJ_2+\frac{t^{d}}{d!d^N}\sJ_1 \RD_{0,k+3}\\
&=&\sum_{i=1}^{r}\sum_{k=0}^{\infty}\sum_{j=0}^{k}{\phi^i\over k!}{k\choose j}t^{k-j}\left(\frac{t^{d}}{d!d^N}\right)^{j}\LD\alpha,\beta,\phi_i, \underbrace{\sJ_2, \ldots \sJ_2}_{(k-j)\text{-copies}}, 
\underbrace{\sJ_1, \ldots \sJ_1}_{j\text{-copies}}\RD_{0,k+3}\\
&=&\sum_{i=1}^{r}\sum_{k=0}^{\infty}{\phi^i\over k!}t^{k}\LD\alpha,\beta,\phi_i, \underbrace{\sJ_2, \ldots \sJ_2}_{k\text{-copies}}\RD_{0,k+3}\\
&=&\alpha\star_{t\sJ_2}\beta.
\end{eqnarray*}
Here in the third equality, all the terms with $j>0$ vanishes by the string equation.
Now the result follows from the formula $\tau'=\sJ_2+\frac{1}{d!d^{N-1}}\sJ_1$ and Proposition~\ref{basis-fermat-subspace}.
\end{proof}

Now according to Corollary~\ref{cor-linearlity}, Proposition~\ref{mirror-theorem-fermat}, and Proposition~\ref{basis-fermat}, we have 
\begin{proposition}
\label{quantum-spectrum-fermat-invariant}
Quantum spectrum conjecture~\ref{quantum-spectrum-conj-invariant} holds true for Fermat homogenous polynomials of general type.
\end{proposition}

\subsubsection{A proof of Proposition~\ref{prop-quantum-spectrum-inv}}
Now we consider the cases when $\nu>1$. For these cases, we can prove that quantum spectrum conjecture~\ref{quantum-spectrum-conj} holds. 
It suffices to show that Proposition~\ref{prop-quantum-spectrum-inv} holds as the action $\sJ_2\star_{t\sJ_2}$ on the broad subspace $\cH_{\rm bro}$  is trivial. 
\begin{lemma}
\label{trivial-action-nu>1}
Let $W=x_1^d+\ldots +x_N^d$. If $\nu>1$, then we have $\sJ_2\star_t\phi=0$ for each homogeneous element $\phi\in \cH_{\rm bro}$.
That is, $$\sJ_2\star_t\vert_{\cH_{J^0}}=0.$$
\end{lemma}
\begin{proof}
According to Proposition~\ref{fermat-broad}, each homogeneous element $\phi\in \cH_{\rm bro}$ is a broad element in $\cH_{J^0}$. 
Using the Hodge grading operator formula~\eqref{hodge-grading-operator} and degree formula~\eqref{degree-hodge}, we must have 
$$\deg_\C\phi={\widehat{c}_W\over 2}={N\over 2}-{N\over d}.$$
Now we consider the quantum product 
\begin{equation}
\label{fermat-broad-vanishi}
\sJ_2\star_{t}\phi=\sum_{j=1}^{r}\sum_{n=0}^{\infty}{t^n\phi^j\over n!}\LD\sJ_2, \phi, \phi_j, \underbrace{\sJ_2, \ldots \sJ_2}_{n\text{-copies}}\RD_{0,n+3}.
\end{equation}
Here $\{\phi_j\}$ is a homogeneous basis of $\cH_{W, \<J\>}$ and $\{\phi^j\}$ is the dual basis.
Using Lemma~\ref{gmax-inv-inv}, we observe that the nonzero invariant 
$\LD\sJ_2, \phi, \phi_j, \sJ_2, \ldots, \sJ_2\RD_{0,n+3}$ must come from some element $\phi_j\in \cH_{\rm bro}.$
Thus we must have 
$$\deg_\C\phi_j={\widehat{c}_W\over 2}.$$
Since we have the degree formula $\deg_\C\sJ_2=1-{\nu\over d}$, we can apply the homogeneity property~\eqref{virtual-degree} to the invariant $\LD\sJ_2, \phi, \phi_j, \sJ_2, \ldots, \sJ_2\RD_{0,n+3}$ and obtain $${n+1\over d}={1\over d-N}.$$
On the other hand, the selection rule~\eqref{selection-rule} implies that 
$${n+1\over d}\equiv 0 \mod \Z.$$
But we assume $d-N=\nu>1$, this is a contradiction. 
So all the FJRW invariants in the formula~\eqref{fermat-broad-vanishi} vanish and the result follows.
\end{proof}

\vskip .1in

\noindent{\small Department of Mathematics, University of Oregon, Eugene, OR 97403,
USA}

\noindent{\small E-mail: yfshen@uoregon.edu}

\vskip .1in

\noindent{\small Department of Mathematics, University of California, San Diego, La Jolla, CA 92093, USA}

\noindent{\small E-mail: miz017@ucsd.edu}


\end{document}